\documentclass[reqno,10pt]{amsart}

\title[]{Metric properties of boundary maps, Hilbert entropy and non-differentiability}

\author[]{Beatrice Pozzetti}

\author[]{Andr\'es Sambarino}

\thanks{}
\date{\today}

\subjclass[]{}

\usepackage[colorlinks = true,backref = page,hyperindex,breaklinks]{hyperref} 
\usepackage{ stmaryrd }
\usepackage{amssymb}
\usepackage{mathtools}      
\usepackage{mathabx}        
\usepackage[bb = fourier,cal = euler,scr = boondox]{mathalfa}	
\usepackage{enumitem}       
\usepackage[english]{babel}
\usepackage{tikz}			
\usepackage{tikz-cd}		
\usetikzlibrary{intersections}
\usepackage[font = small]{caption}	
\usepackage{nicefrac}
\usetikzlibrary{calc}
\usepackage{todonotes}
\usepackage{comment}

\renewcommand*{\backref}[1]{}
\renewcommand*{\backrefalt}[4]{\quad \tiny
  \ifcase #1 (\textbf{NOT CITED.})%
  \or    (Cited on page~#2.)%
  \else   (Cited on pages~#2.)%
  \fi}

\makeatletter

\def\MRbibitem{\@ifnextchar[\my@lbibitem\my@bibitem}

\def\mybiblabel#1#2{\@biblabel{{\hyperref{http://www.ams.org/mathscinet-getitem?mr=#1}{}{}{#2}}}}

\def\myhyperanchor#1{\Hy@raisedlink{\hyper@anchorstart{cite.#1}\hyper@anchorend}}

\def\my@lbibitem[#1]#2#3#4\par{%
  \item[\mybiblabel{#2}{#1}\myhyperanchor{#3}\hfill]#4%
  \@ifundefined{ifbackrefparscan}{}{\BR@backref{#3}}%
  \if@filesw{\let\protect\noexpand\immediate
    \write\@auxout{\string\bibcite{#3}{#1}}}\fi\ignorespaces%
}

\def\my@bibitem#1#2#3\par{%
  \refstepcounter\@listctr
  \item[\mybiblabel{#1}{\the\value\@listctr}\myhyperanchor{#2}\hfill]#3%
  \@ifundefined{ifbackrefparscan}{}{\BR@backref{#2}}%
  \if@filesw\immediate\write\@auxout
    {\string\bibcite{#2}{\the\value\@listctr}}\fi\ignorespaces%
}

\makeatother

\newcommand{\xqedhere}[2]{%
  \rlap{\hbox to#1{\hfil\llap{\ensuremath{#2}}}}}

\newcommand{\Z}{\mathbb{Z}} 

\newcommand{\R}{\mathbb{R}} \newcommand{\RR}{\R}
\newcommand{\C}{\mathbb{C}}
\newcommand{\N}{\mathbb{N}}
\renewcommand{\P}{\mathbb{P}}
\newcommand{\T}{\mathbb{T}}
\newcommand{\I}{\mathbb{I}}
\renewcommand{\H}{\mathbb{H}}

\newcommand{\lb}{\llbracket}
\newcommand{\rb}{\rrbracket}

\newcommand{\G}{\sf{\Gamma}}    
\newcommand{\Gr}{\cal G}    
\newcommand{\cone}{{\cal C}}
\newcommand{\<}{\langle}
\renewcommand{\>}{\rangle}

\newcommand{\g}{\gamma}

\newcommand{\bord}{\partial}

\newcommand{\posgen}{\cal F^{(2)}}
\renewcommand{\t}{\theta}

\newcommand{\p}{{\mathfrak{p}}}

\newcommand{\wk}{\check}
\newcommand{\bb}{{\sf{b}}}

\renewcommand{\/}{\backslash}

\newcommand{\II}{\mathbf{I}}

\newcommand{\grupo}{\varLambda}

\renewcommand{\circle}{\mathbb{S}^1}

\newcommand{\scr}{\mathscr}
\renewcommand{\sf}[1]{{\mathsf{#1}}}

\newcommand{\cal}{\mathcal}
\renewcommand{\frak}{\mathfrak}

\newcommand{\hJ}[1]{{\scr h}_{#1}}
\newcommand{\hC}[1]{{\scr h}^{#1}}

\renewcommand{\flat}{{\sf b}}  
\newcommand{\hol}{ \scr b} 

\DeclareMathOperator{\bus}{b}
\newcommand{\ii}{\mathrm{i}}
\DeclareMathOperator{\spa}{span}
\DeclareMathOperator{\class}{C}
\DeclareMathOperator{\diam}{diam}
\DeclareMathOperator{\SL}{{\mathsf{SL}}}
\DeclareMathOperator{\PSL}{{\mathsf{PSL}}}
\DeclareMathOperator{\GL}{{\mathsf{GL}}}
\DeclareMathOperator{\SO}{{\mathsf{SO}}}
\DeclareMathOperator{\PGL}{{\mathsf{PGL}}}
\DeclareMathOperator{\PO}{{\mathsf{PO}}}

\DeclareMathOperator{\id}{id}
\DeclareMathOperator{\inte}{int}

\DeclareMathOperator{\Hff}{dim_{Hf{}f}}

\DeclareMathOperator{\ann}{Ann}
\DeclareMathOperator{\codim}{codim}
\DeclareMathOperator{\Ext}{\cal{H}}

\newcommand{\grass}{\mathrm{Gr}}
\renewcommand{\angle}{\measuredangle}

\DeclareMathOperator{\Sp}{\mathsf{Sp}}
\DeclareMathOperator{\SU}{\mathsf{SU}}

\newcommand{\K}{\mathbb K}

\renewcommand{\sl}{\mathfrak {sl}}

\renewcommand{\varXi}{\mathscr{G}}

\newcommand{\piS}{\pi_1S}
\renewcommand{\ge}{{\frak g}}
\newcommand{\poids}{{\sf \Pi}}
\renewcommand{\a}{{\frak a}}
\newcommand{\roots}{{\mathsf{\Phi}}}
\newcommand{\Weyl}{{\cal W}}
\renewcommand{\k}{{\frak k}}
\newcommand{\ps}{\mu}
\newcommand{\cc}{{\scr v}}
\newcommand{\simple}{{\sf \Delta}}
\newcommand{\Fund}{\Phi}
\newcommand{\peso}{\varpi}
\newcommand{\rfr}{\beta}
\newcommand{\Fundeb}{\mathrm{Leb}}
\newcommand{\BM}{\Omega}
\newcommand{\df}{\omega}

\newcommand{\Ann}{\mathrm{Ann}}
\newcommand{\vi}{{\varphi^\infty_\hol}}

\newcommand{\cartan}{a}
\newcommand{\sroot}{{\sf{a}}}
\newcommand{\slroot}{{\tau}}

\newcommand{\Ndiff}{{\sf{ NDiff}}}
\DeclareMathOperator{\Diff}{Diff}

\newcommand{\ovrho}{\overline \rho}
\newcommand{\ovxi}{\overline \xi}
\renewcommand{\epsilon}{\varepsilon}

\hypersetup{bookmarksdepth  =  3} 
\numberwithin{equation}{section}     

\setlist[enumerate,1]{label = {\upshape(\roman*)},ref = \roman*}
\setlist[enumerate,2]{label = {\upshape(\alph*)},ref = \alph*}

\newtheorem{thm}{Theorem}[section]

\newtheorem{thmA}{Theorem}

\newtheorem{corA}{Corollary}

\newtheorem*{fact}{Fact}
\newtheorem{cor}[thm]{Corollary}
\newtheorem{lemma}[thm]{Lemma}
\newtheorem{prop}[thm]{Proposition}

\theoremstyle{definition}

\newtheorem{defi}[thm]{Definition}
\newtheorem{assu}[thm]{Assumption}

\newtheorem{remark}[thm]{Remark}
\newtheorem{ex}[thm]{Example}

\theoremstyle{remark}
\newtheorem{obs}[thm]{Remark}

\newcommand{\calF}{\mathcal F}
\newcommand{\calL}{\mathcal L}

\newcommand{\ov}{\overline}
\newcommand{\cdc}[1]{\big(\calL_{#1}\big)^*}

\begin{document}
\thanks{B. P. is supported by the Deutsche Forschungsgemeinschaft under Germany’s Excellence Strategy EXC-2181/1 - 390900948 (the Heidelberg STRUCTURES Cluster of Excellence), and acknowledges further support by DFG grant 338644254 (within the framework of SPP2026) and 427903332 (in the Emmy Noether program). A. S. was partially financed by ANR DynGeo ANR-16-CE40-0025. Part of this work was carried out in Oberwolfach, we thank the institute for its great hospitality.  B.P. additionally thanks Prof. Farkas and Humboldt University for their hospitality, the Familienzimmer des mathematiches Institut made it possible to work on this paper while taking care of a small child.}
\begin{abstract}
	We interpret the Hilbert entropy of a convex projective structure on a closed higher-genus surface as the Hausdorff dimension of the non-differentiability points of the limit set in the full flag space $\cal F(\R^3)$. Generalizations for regularity properties of boundary maps  between locally conformal representations are also discussed. An ingredient for the proofs is the concept of \emph{hyperplane conicality} that we introduce for a $\t$-Anosov representation into a reductive real-algebraic Lie group $\sf G$. In contrast with directional conicality, hyperplane-conical points always have full mass for the corresponding Patterson-Sullivan measure.
\end{abstract}

\maketitle

\tableofcontents

\section{Introduction}\label{intro}

Consider a closed connected orientable surface $S$ of genus at least two, and let $\rho:\piS\to\PSL(3,\R)$ be a faithful representation preserving an open convex set $\Omega=\Omega_\rho\subset\P(\R^3)$, properly contained in an affine chart. The group $\rho(\piS)$ is necessarily discrete and acts co-compactly on $\Omega$:  one says that $\rho$ \emph{divides} $\Omega$.

The geometry of such convex set $\Omega$ is well studied, by Benoist \cite{convexes1} it is strictly convex with $\class^{1+\nu}$ boundary $\bord\Omega$ (that is not $\class^2$ unless it is an ellipse), and the Hilbert metric of $\Omega$ is Gromov-hyperbolic. The geodesic flow of $\Omega/\rho(\piS)$ is an Anosov flow and its topological entropy, the \emph{Hilbert entropy} $\hJ{\sf H}={(\hJ{\sf H})}_\rho$, satisfies $$\hJ{\sf H}\leq1,$$ an inequality proved by Crampon \cite{crampon} that is  strict if $\Omega$ is not an ellipse.


A consequence of Theorem \ref{tutti} below is a new geometric interpretation of the Hilbert entropy which we now explain. For each $x\in\bord\Omega$ let $\Xi(x)\in\mathrm{Gr}_{2}(\R^3)$ be the unique plane whose projectivisation is tangent to $\bord\Omega$ at $x.$ By \cite{convexes1}, the image curve $\Xi(\bord\Omega)\subset\mathrm{Gr}_2(\R^3)\simeq\P((\R^3)^*)$ is also the boundary of a strictly convex divisible set $\Omega^*$ and is thus again a $\class^{1+\nu}$-circle. 
The \emph{full-flag-curve} $$\{(x,\Xi(x)):x\in\bord\Omega\}\subset\cal F(\R^3),$$
is the graph of a monotone map between $\class^1$ circles and thus is a Lipschitz submanifold that is therefore differentiable almost everywhere. We establish the following:

\begin{corA}\label{hilbert} Let $\rho:\piS\to\PSL(3,\R)$ divide a strictly convex set that is not an ellipse. Then, the set of non-differentiability points of the full flag curve 
	has Hausdorff dimension ${(\hJ{\sf H})}_\rho.$
\end{corA}

Throughout the paper the Hausdorff dimension is computed with respect to a(ny) Riemannian metric on the flag space. When $\Omega$ is an ellipse the result does not apply as the associated curve is differentiable everywhere while $\hJ{\sf H}=1$.

A classical result by Choi-Goldman \cite{choigoldman} states that the space of representations dividing a convex set forms a connected component of the character variety $\frak X\big(\piS,\PSL(3,\R)\big)$ of homomorphisms up to conjugation. This component is known today as \emph{the Hitchin component} of $\PSL(3,\R)$ and is diffeomorphic to a ball of dimension $-8\chi(S).$ Nie \cite{Nie0} and Zhang \cite{ZhangSL3} have found paths $(\rho_t)$ in this Hitchin component such that $({\hJ{\sf H}})_{\rho_t}\to 0$ as $t\to\infty$. Together with Corollary \ref{hilbert} this suggest that the closer $\Omega$ is to being an ellipse (the \emph{Fuchsian locus}), the less differentiable the flag curve is whilst the furthest away from Fuchsian locus, the more regular the flag curve becomes.

The proof of Corollary \ref{hilbert} is outlined in \S\,\ref{corAProof} and serves as a guide path for the strategy on the general case (Theorems \ref{LCdiff} and \ref{tutti}).

\subsection{Locally conformal representations and concavity properties}\label{2hyper}

Let $\K$ be $\R$, $\C$ or the non-commutative field of Hamilton's quaternions $\H$.
Denote by 
$$\a=\big\{(a_1,\ldots,a_d)\in\R^d:\sum_i a_i=0\big\}$$ 
the Cartan subspace of the real-algebraic group $\SL(d,\K)$, by 
\begin{equation}\label{e.slroot}
	\slroot_i(a_1,\ldots,a_d)=a_i-a_{i+1}
\end{equation}	 
the $i$-th simple root and by $\a^+\subset \a$ the Weyl chamber whose associated set of simple roots is $\simple=\{\slroot_i:i\in\lb1,d-1\rb\}.$ Let $\cartan:\SL(d,\K)\to\a^+$ be the \emph{Cartan projection} with respect to the choice of an inner (or Hermitian) product on $\K^d$. The $e^{\cartan_i(g)}$'s are the \emph{singular values} of the matrix $g$, namely the square roots of the modulus of the eigenvalues of the matrix $gg^*$. We also let  $d_\P$ denote the distance on $\P(\K^d)$ induced by the chosen Hermitian product. 

Let $\G$ be a finitely generated word-hyperbolic group, consider a finite symmetric generating set and let $|\,|$ be the associated word-length. For $k\in\lb1,d-1\rb$,  a representation $\rho:\G\to\SL(d,\K)$ is $\{\slroot_k\}$-\emph{Anosov} if there exist positive constants $\mu$ and $c$ such that for all $\g\in\G$ one has $$\slroot_k\big(\cartan(\rho(\g))\big)\geq\mu|\g|-c.$$ 

A $\{\slroot_k\}$-{Anosov} representation is also $\{\slroot_{d-k}\}$-Anosov. Under such assumption there exists an equivariant H\"older-continuous map $$\xi^{k}_\rho:\bord\G\to\grass_k(\K^d),$$ 
called the \emph{limit map} in the Grassmannian $\grass_k(\K^d)$ of $k$-dimensional subspaces of $\K^d$, which is a homeomorphism onto its image. If $k\leq l\in\lb1,d-1\rb$ and $\rho$ is also $\{\slroot_l\}$-Anosov then the limit maps are {compatible}, i.e. $\xi^k_\rho(x)\subset\xi^l_\rho(x)$ $\forall x$, see \S\ref{Anosov} for references and details. 

\begin{defi}\label{d.hyp}
Fix $p\in\lb2,d-1\rb$.
	A $\{\slroot_1,\slroot_{d-p}\}$-Anosov representation $\rho:\G\to\SL(d,\K)$ is \emph{$(1,1,p)$-hyperconvex} if for every pairwise distinct triple $x,y,z\in\bord\G$ one has 
	\begin{equation}\label{hypdefi}\big(\xi^{1}_\rho(x)+\xi^{1}_\rho(y)\big)\cap\xi^{{d-p}}_\rho(z)=\{0\}.
	\end{equation}If in addition one has $\cartan_2(\rho(\g))=\cartan_{p}(\rho(\g))$ $\forall\g$, we say that $\rho$ is \emph{locally conformal}.
\end{defi}	

Hyperconvex representations form an open subset of the character variety $$\frak X\big(\G,\SL(d,\K)\big)=\hom\big(\G,\SL(d,\K)\big)/\SL(d,\K)$$ and appear  naturally. For example, when $\K=\R$,  strictly convex divisible sets give rise to $(1,1,d-1)$-hyperconvex representations, while higher rank Teichm\"uller theory provides many examples of $(1,1,2)$-hyperconvex representations of surface groups, see Example \ref{ex.hyperconvex}.

When $p=2$ the second part of the definition is trivially true, so $(1,1,2)$-hyperconvex representations over $\K$ \emph{are} locally conformal, when $p>2$ the assumption constrains the Zariski closure of $\rho(\G)$. However, Zariski-dense locally conformal representations exist (and form open sets) for the groups locally isomorphic to $\SL(n,\R)$, $\SL(n,\C)$, $\SL(n,\H)$, $\SU(1,n)$, $\Sp(1,n)$, $\SO(p,q)$, see P.-S.-Wienhard \cite[\S\,8]{PSW1} for details, and, of course, $\SO(1,n)$ where every convex co-compact representation is locally conformal.

A concrete example in $\SU(1,n)$ consist on considering a convex co-compact group in $\H^n_\C$ whose limit set intersects the projectivization of any complex line in at most 2 points. These subgroups are locally conformal (\cite[Proposition 8.3]{PSW1}) and their limit set (though fractal) is tangent to the contact distribution of $\bord\H^n_\C$.

Consider also $\ov\K\in\{\R,\C,\H\}$ and positive integers $d$ and $\ov d$. Throughout the paper we mainly deal with a pair of locally conformal representations $$\rho:\G\to\SL(d,\K)\textrm{ and }\ov\rho:\G\to\SL(\ov d,\ov\K),$$ with equivariant maps $\xi=\xi^1_\rho\textrm{ and }\ov\xi=\xi^1_{\ov\rho},$ and we study regularity properties of the equivariant H\"older-continuous homeomorphism $$\Xi=\ov\xi\circ\xi^{-1}:\xi(\bord\G)\to\ov{\xi}(\bord\G).$$

To avoid confusion we denote the simple roots of $\SL(\ov d,\ov\K)$ by $\big\{\ov{\slroot}_i:i\in\lb1,\ov d-1\rb\big\},$ and to simplify notation we identify $\g$ with $\rho(\g)$ and we let $\ov\g=\ov\rho(\g)$. We consider also the graph of $\Xi$, or equivalently the \emph{graph map}, $$\varXi:\big(\xi,\ov{\xi}\big):\bord\G\to\P(\K^d)\times\P(\ov\K{}^{\ov d}).$$

\begin{defi}Fix  $\hol\in(0,1]$. We will say that that $\Xi$ is $\hol$\emph{-concave} at $x\in\bord\G$, or that $x$ is a $\hol$\emph{-concavity point for} $\Xi$,  if there exists a sequence $(y_k)$ converging to $x$ as $k\to\infty$ such that the incremental quotient \begin{equation}\label{incri}
	\frac{d_\P\big(\ov\xi(x),\ov\xi(y_k)\big)}{d_\P\big(\xi(x),\xi(y_k)\big)^\hol}
\end{equation} is bounded away from $\{0,\infty\}$. The set of $\hol$-concavity points is denoted by $\Ext_{\rho,\ov\rho}^\hol$. 
\end{defi}

Observe that $\Xi$ can be $\hol$-concave at $x$ for several $\hol$'s and that it is a $1$-concave point if one has $y_k\to x$ such that $d(\xi(x),\xi(y_k))$ and $d_\P(\ov\xi(x),\ov\xi(y_k))$ are comparable.

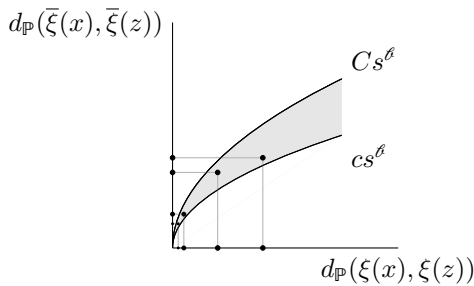
\begin{figure}[h]\centering
      \begin{tikzpicture}
      \begin{scope}[scale=1.5]
         \begin{scope}[rotate=-90]
      \draw[name path=Q] (0,0) parabola bend (0,0) (-1.5,1.5);
      \fill[color=gray!20]        (0,0) parabola bend (0,0) (-1.5,1.5)--(-1,1.5)--cycle;
      \fill[color=white]  (0,0) parabola bend (0,0) (-1,1.5)--cycle;
      \draw (0,0) parabola bend (0,0) (-1,1.5);
       \draw[name path=Q] (0,0) parabola bend (0,0) (-1.5,1.5);
      \node at (-1.5,1.5) [above right] {$Cs^\hol$};
      \node at (-1,1.5) [below right] {$cs^\hol$};
      \end{scope}
      
      \draw[gray!60] (0.8,0)--(0.8,0.8)--(0,0.8);
      \draw[gray!60] (0.4,0)--(0.4,0.67)--(0,0.67);
       \draw[gray!60] (0.1,0)--(0.1,0.3)--(0,0.3);
       \draw[gray!60] (0.05,0)--(0.05,0.215)--(0,0.215);
       
   \draw (0,0)--(2,0);
   \node[below] at (2,0) {$d_\P(\xi(x),\xi(z))$};
   \draw(0,0)--(0,2);
   \node[left] at (0,2) {$d_\P(\ov\xi(x),\ov\xi(z))$};
   
   \draw [fill] (0.8,0) circle [radius = 0.02];
   \draw [fill] (0.8,0.8) circle [radius = 0.02];
   \draw [fill] (0,0.80) circle [radius = 0.02];

      \draw [fill] (0.4,0) circle [radius = 0.02];
      \draw [fill] (0.4,0.67) circle [radius = 0.02];
       \draw [fill] (0,0.67) circle [radius = 0.02];

      \draw [fill] (0.1,0.3) circle [radius = 0.017];
      \draw [fill] (0.1,0) circle [radius = 0.017];
       \draw [fill] (0,0.3) circle [radius = 0.017];
       
      \draw [fill] (0.05,0) circle [radius = 0.008];
      \draw [fill] (0.05,0.215) circle [radius = 0.008];
      \draw [fill] (0,0.215) circle [radius = 0.008];
      
          \begin{scope}[rotate=-90]
      \draw (0,0) parabola bend (0,0) (-1.5,1.5);
      
      \draw (0,0) parabola bend (0,0) (-1,1.5);
      \end{scope}

      \end{scope}
\end{tikzpicture}
\caption{A $\hol$-concave point $x$. The marked points on the axis' represent $d_\P(\xi(x),\xi(y_k))$ and $d_\P(\ov\xi(x),\ov\xi(y_k))$ respectively.}
\end{figure}

In what follows we will compute the Hausdorff dimension of $\varXi(\Ext_{\rho,\ov\rho}^\hol)$ with respect to the product metric on $\P(\K^d)\times\P(\ov\K{}^{\ov d})$ for $\hol$ lying on an interval that we now define. The \emph{dynamical intersection} between $\rho$ and $\ov\rho$ with respect to $\ov\slroot_1$ and $\slroot_1$ is defined by $$\II_{\ov\slroot_1}(\slroot_1)=\lim_{t\to\infty}\frac1{\#\sf R_t(\ov\slroot_1)}\sum_{\g\in \sf R_t(\ov\slroot_1)}\frac{\slroot_1(\lambda(\g))}{\ov\slroot_1(\lambda(\ov\g))},$$ 
where $\mathsf R_t(\ov\slroot_1)=\big\{[\g]\in[\G]:\ov\slroot_1\big(\lambda(\ov\g)\big)\leq t\big\}$ and $\lambda:\SL(d,\K)\to\a^+$ is the Jordan projection. This concept  (from Bridgeman-Canary-Labourie-S. \cite{pressure}, Burger \cite{burger}, Knieper \cite{kni95}, among others) generalizes Bonahon's intersection number between two elements in Teichm\"uller space. 

Let us say that $\rho$ and $\ov\rho$ are \emph{gap-isospectral} if for all $\g\in\G$ one has $$\slroot_1\big(\lambda(\g)\big)=\ov\slroot_1\big(\lambda(\ov\g)\big).$$ Corollary \ref{corraiz} (a consequence of \cite{pressure} together with Proposition \ref{nonA}) implies that if $\rho$ and $\ov\rho$ are not gap-isospectral, then $\II_{\slroot_1}(\ov\slroot_1)>\big(\II_{\ov\slroot_1}(\slroot_1)\big)^{-1}$.  We will study $\hol$-concavity for any $\hol\in(0,1]$ with $$\II_{\slroot_1}(\ov\slroot_1)>\hol>\big(\II_{\ov\slroot_1}(\slroot_1)\big)^{-1}.$$

Finally, consider the critical exponents \begin{alignat*}{2}\hC{\slroot_1} & =\lim_{t\to\infty}\frac1t\log\#\big\{\g\in\G:\slroot_1\big(\cartan(\g)\big)\leq t\big\},\\
\hC{\infty,\hol} & =\lim_{t\to\infty}\frac1t\log\#\big\{\g\in\G:\max\big\{\hol\slroot_1\big(\cartan(\g)\big),\ov{\slroot}_1\big(\cartan(\ov\g)\big)\big\}\leq t\big\}.\end{alignat*} 

\begin{thmA}[Theorem \ref{tutti.LCdiff}]\label{LCdiff} Let $\{\K,\ov\K\}\subset\{\R,\C\}$ and let $\rho:\G\to\SL(d,\K)$  and $\ov\rho:\G\to\SL(\ov d,\ov\K)$ be locally conformal, $\R$-irreducible and not gap-isospectral. Then for any $\hol\in(0,1]$ with $\II_{\slroot_1}(\ov\slroot_1)>\hol>\big(\II_{\ov\slroot_1}(\slroot_1)\big)^{-1}$, one has 
\begin{alignat*}{2}\hol\hC{\infty,\hol}\leq\Hff(\varXi({\Ext}_{\rho,\ov\rho}^\hol))&\leq\min\{\hC{\infty,\hol},\hol\hC{\infty,\hol}+1-\hol \}\nonumber\\&<\min\{\hC{\ov\slroot_1},\hC{\slroot_1}/\hol\}\nonumber\\&\leq\Hff(\varXi(\bord\G))\\&=\max\{\hC{\slroot_1},\hC{\ov\slroot_1}\}.\end{alignat*} If $\K=\H$ (resp. $\ov\K=\H$) we further assume that the Zariski closure if $\rho$ (resp. $\ov\rho$) does not have compact factors, then the same conclusion holds.
\end{thmA}	

The proof of the above Theorem is completed in \S\,\ref{ProofA}. For representations in $\PSL(2,\C)$ we can furthermore give a geometric interpretation of the 1-weakly-bi-H\"older points, see \S\,\ref{klein}.

\subsection{Surface-group representations} Observe that the first line of inequalities in Theorem \ref{LCdiff} becomes an equality when $\hol=1$. We pursue now this situation while further restricting the source and ambient groups.

Let then $\K=\R$ and assume $\bord\G$ is homeomorphic to a circle. Real representations of $\G$ that are $(1,1,2)$-hyperconvex are necessarily locally conformal and form the prototype example of Anosov representations with $\class^1$ limit sets: indeed we have the following result from P.-S.-Wienhard \cite{PSW1} and Zhang-Zimmer \cite{ZZ}.

\begin{thm}\label{t.C1}
	Assume $\bord\G$ is homeomorphic to a circle and let $\rho:\G\to\PGL(d,\R)$ be $\{\slroot_1\}$-Anosov. \begin{itemize}\item[\cite{PSW1},\cite{ZZ}:] If $\rho$ is $(1,1,2)$-hyperconvex, then $\xi^1(\bord\G)\subset\P(\R^{d})$ is a $\class^1$ submanifold tangent at $\xi^1(x)$ to $\xi^2(x)$.
		\item[{\cite{ZZ}}:]If $\rho$ is irreducible and $\xi(\bord\G)$ is a $\class^1$ circle then $\rho$ is $(1,1,2)$-hyperconvex.
	\end{itemize}
\end{thm}

The graph map $\varXi=\big(\xi,\ov{\xi}\big):\bord\G\to\P(\R^{d})\times\P(\R^{\ov d})$ has image contained in the $\class^{1+\nu}$ torus $\xi(\bord\G)\times\ov{\xi}(\bord\G)$ and $\varXi(\bord\G)$ is the graph of $\Xi$, a H\"older-continuous homeomorphism between $\class^{1+\nu}$-circles. By monotonicity of $\Xi$, $\varXi(\bord\G)$ is a Lipschitz curve and is thus differentiable almost everywhere. We let $$\Ndiff_{\rho,\ov\rho}\subset\varXi(\bord\G)$$ be the subset of points where the curve $\varXi(\bord\G)$ is not differentiable. The combination of Lemma \ref{generalcase} and Corollary \ref{nondiffbconicalR} establishes that in the current situation (with mild additional assumptions) $$\varXi\big(\Ext^1_{\rho,\ov\rho}\big)=\Ndiff_{(\rho,\ov\rho)},$$ whence with Theorem \ref{LCdiff} one obtains the following:

\begin{thmA}\label{tutti} Assume $\bord\G$ is homeomorphic to a circle and let $\rho:\G\to\SL(d,\R)$ and $\ovrho:\G\to\SL(\ov d,\R)$ be $(1,1,2)$-hyperconvex and not gap-isospectral. Then, $$\Hff\big(\Ndiff_{\rho,\ov\rho}\big)=\hC{\infty,1}<1.$$
\end{thmA}

We emphasize  that no irreducibility assumption is made on the representations $\rho$ and $\ov\rho.$ On the other hand, if the representations are irreducible and gap-isospectral, we show that there exists an isomorphism between the Zariski closures of $\rho(\G)$ and of $\ov\rho(\G)$ intertwining the two representations. It follows then that $\varXi(\bord\G)$ is the diagonal of the  $\class^{1+\nu}$ torus, and thus differentiable everywhere. To prove this we give the following preliminary classification of Zariski-closures, established in \S\,\ref{s.Zd}. 

Recall that if $\sf G$ is a semi-simple real-algebraic group of non-compact type, then irreducible proximal  representations $\Fund:\sf G\to\PGL(V)$ are determined by their highest restricted weight $\chi_\Fund$. A special subset of dominant weights are the so-called \emph{fundamental weights} $\{\peso_\sroot:\sroot\in\simple\}$, and are indexed by the set of simple roots $\simple$ of $\sf G$ (see \S\,\ref{representaciones} for definitions and details).  

\begin{thmA}\label{t.Zcl} Assume $\bord\G$ is homeomorphic to a circle and let $\rho:\G\to\PGL(d,\R)$ be irreducible and $(1,1,2)$-hyperconvex.  Then the Zariski closure $\sf G$ of $\rho(\G)$ is simple and the highest weight of the induced representation $\Fund:\sf G\to\PGL(d,\R)$ is a multiple of a fundamental weight associated to a root whose root-space is one-dimensional.
\end{thmA}

In light of the following examples it is unclear if further restrictions can occur.

\begin{ex}\label{ex.hyperconvex} 
	Any pair of representations $\rho:\pi_1S\to\sf G$ and $\Fund:\sf G\to\PGL(V)$ in each of the following classes (and small deformations), gives rise to a $(1,1,2)$-hyperconvex representation via post-composition $\Fund\circ\rho$. In particular the limit set of $\rho$ in the specified flag manifold of $\sf G$ is a $\class^{1+\nu}$ curve:
\begin{itemize}
	\item[-] $\sf G$ is split, $\rho:\piS\to\sf G$ is Hitchin, and  $\Fund$ satisfies $\chi_\Fund=n\peso_\sroot$ for any $\sroot\in\simple$ and $n\in\N_{>0}$. This is non-trivial and requires results from Fock-Goncharov \cite{FG} and Labourie \cite{labourie} together with Lusztig's canonical basis \cite[Proposition 3.2]{Lusztig-TP} (see S. \cite[\S\,5.8]{clausurasPos} for details). As a result the limit set of $\rho$ in any maximal flag manifold $\calF_{\{\sroot\}}$ of $\sf G$ is a $\class^{1+\nu}$ curve.
	\item[-]  $\rho:\piS\to\PO(p,q)$ is $\Theta$-positive  and $\Fund$ has highest weight $\peso_\sroot$ for any root $\sroot$ in the interior\footnote{i.e. $\sroot$ is only connected to roots in $\Theta$ in the Dynkin diagram of $\simple$} of $\Theta$ (P.-S.-Wienhard \cite[Theorem 10.3]{PSW2}, see also Beyrer-P. \cite[Remark 4.6]{BP}). In particular the limit set in any flag manifold of the form ${\sf{Is}}_k(\R^{p,q})$ for $k\leq p-2$ is a $\class^{1+\nu}$-curve. When $\rho$ is moreover Zariski-dense, we can consider any $\Fund$ with $\chi_\Fund^+=n\peso_\sroot$ for any $\sroot \in\inte\Theta$ and $n\in\N_{>0}$. 
	\item[-] for all $k\geq 1$, $k$-positive representations $\rho:\piS\to\PSL(d,\R)$ introduced in Beyrer-P. \cite{BeyP2} are  $(1,1,2)$-hyperconvex.
\end{itemize}
\end{ex}

For these examples also  the following applies:

\begin{corA}\label{cor.B} Assume $\bord\G$ is homeomorphic to a circle, let $\sf G$ be a simple Lie group and let $\rho:\G\to\sf G$ have Zariski-dense image. Assume there exist $\{\sroot,\bb\}\subset\simple$ distinct such that both $\Fund_\sroot\circ\rho$ and $\Fund_\bb\circ\rho$ are $(1,1,2)$-hyperconvex. Then:
	\begin{enumerate}
		\item The image of  $\xi^{\{\sroot,\bb\}}:\bord\G\to\cal F_{\{\sroot,\bb\}}$ is Lipschitz  and the Hausdorff dimension of the points where it is non-differentiable is $\hC{\max\{\sroot,\bb\}}.$ 
		\item If the opposition involution $\ii$ on $\ge$ is non-trivial and $\bb=\ii\sroot$ then $$\hC{\max\{\sroot,\bb\}}=\hC{(\sroot+\bb)/2}.$$
	\end{enumerate}
\end{corA}

\begin{obs}A different approach to Theorem \ref{tutti}, relying on Theorem \ref{t.Zcl} and Theorem \ref{t.C1}, would be to code the action of $\piS$ on $\bord\piS$ via Bowen-Series and apply Jordan-Kesseböhmer-Pollicott-Stratmann \cite[Theorem 1.1]{JKPS}. This method, followed by Pollicott-Sharp \cite{PS} for two representations in the Teichm\"uller space of $S$, is not applicable for groups other than $\piS$, in particular this approach cannot be used in the generality of Theorem \ref{LCdiff}. \end{obs}

\subsection{Hyperplane vs directional conicality} To prove Theorems \ref{LCdiff} and \ref{tutti} we introduce the concept  of \emph{hyperplane conicality}, a generalization of directional conicality from Burger-Landesberg-Lee-Oh \cite{BLLO}.	

Let $\sf G$ be a real-algebraic semi-simple Lie group of non-compact type, $\a\subset\ge$ a Cartan subspace, $\roots\subset\a^*$ the associated root system and $\simple\subset\roots$ a choice of simple roots with associated Weyl chamber $\a^+.$

Consider a non-empty $\t\subset\simple$ and let $\a_\t$ be the associated Levi space. Fix a $\t$-Anosov representation $\rho:\G\to\sf G$ and denote by $\cal L_{\t,\rho}\subset\a_\t$ its $\t$-limit cone. We will recall in \S\,\ref{rfr} that,  when $\rho(\G)$ is Zariski-dense, there are natural bijections 
 \begin{alignat*}{2}
	\inte\P(\cal L_{\t,\rho})&   \leftrightarrow\cal Q_{\t,\rho}=\{\varphi  \in(\a_\t)^*:\hJ\varphi=1\} \\ & \leftrightarrow \big\{\textrm{Patterson-Sullivan measures supported on $\xi^\t(\bord\G)$}\big\}.
\end{alignat*} For $\varphi\in\cal Q_{\t,\rho}$ we let $\sf u_\varphi\in \inte\P(\cal L_{\t,\rho})$ be the associated direction and $\mu^\varphi$ the associated Patterson-Sullivan measure.

Consider now a hyperplane $\sf W\subset\a_\t$ and assume, for the notion to be interesting, that $\sf W$ intersects the relative interior of $\calL_{\t,\rho}.$ Then $x\in\bord\G$ is \emph{$\sf W$-conical} if there exists a conical sequence $(\g_n)_0^\infty\subset\G$ converging to $x,$ a constant $K$ and a sequence $(w_n)_0^\infty\in\sf  W$ such that for all $n$ one has 
$$\big\|\cartan_\t\big(\rho(\g_n))- w_n\big\|\leq K,$$ where $\cartan_\t:\sf G\to\a^+_\t$ is the $\t$-Cartan projection. The set of $W$-conical points will be denoted by $\bord_{\sf W,\rho}\G=\bord_{\sf W}\G.$ Inspired by \cite{BLLO}, in Theorem \ref{ps-sandwich} we show the following.

\begin{thmA}\label{Ahyper} Let $\rho:\G\to\sf G$ be a Zariski-dense $\t$-Anosov representation and $\sf W$ be a hyperplane of $\a_\t$ intersecting non-trivially the interior of $\calL_{\t,\rho}.$ Then for every $\varphi\in\cal Q_{\t,\rho}$ with $\sf u_\varphi\in\P(\sf W)$ one has $\mu^\varphi(\bord_\sf W\G)=1.$ 
\end{thmA}

\subsection{Strategy of the proof of Corollary \ref{hilbert}}\label{specialcases} \label{corAProof}
Corollary \ref{hilbert} is  a consequence of Theorem \ref{tutti} where $\ov\rho$ is the dual representation of $\rho$. We sketch a direct  proof of Corollary \ref{hilbert} serving as a guide-path for the general result.

	Let  $\rho:\piS\to\SL(3,\R)$ be the holonomy of a strictly convex projective structure dividing the convex set $\Omega$.  We consider  the $L^\infty$ distance on the product $(\P(\R^3),d_\P)\times(\P((\R^3)^*),d_{\P}),$ which is equivalent to the Riemannian distance, and thus induces the same Hausdorff dimension.

As a replacement of Sullivan's shadows we use \emph{coarse cone type at infinity}, inspired by Cannon's work on \emph{cone types}  \cite{CannonCones} (see also \S\ref{cont}). Fix a finite symmetric generating set on $\piS$ and let $|\,|$ be the associated word length. For $\g\in\piS$ and $c>0$, the  \emph{coarse cone type at infinity} $\cone^{c}_\infty(\g)$ of  $\gamma$ is the set of endpoints at infinity of $(c,c)$-quasi geodesic rays based at $\g^{-1}$ passing through the identity.  See Figure \ref{figure:conetype}.

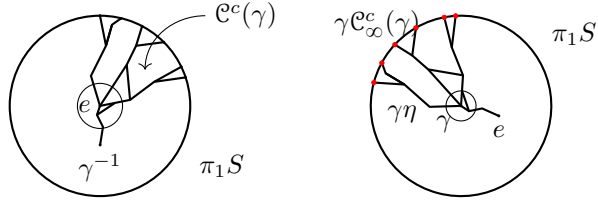
\begin{figure}[hh]
	\centering 
	\begin{tikzpicture}
	
	\begin{scope}[scale=0.4, rotate=75]
	\draw circle [radius=.5];
	\draw [thick] circle [radius = 3];
	\draw [fill] circle [radius = 0.03];	
	\draw[thick] (-0.1,-0.3 ) -- (0.5,0.1); 
	\node [below left] at (0,0) {$\g$};
	\node [below left] at (0,1.2) {$\g\eta$};
	\node [above] at (1.2,3.2) {$\g\cone^{c}_{\infty}(\g)$};
	\draw [thick] (0.004,-1.3) -- (0.1,-0.7) -- (-0.1,-0.3 ) -- (0,0);
	\draw [thick] (0,0) -- (-0.3,1) -- (0.2,2) -- (0,3);
	\draw [thick] (0.2,2) -- (0.4, 2.7) -- (0.7,2.91);
	\draw [thick] (0.2,2) -- (0.4, 2.7);
	\draw [thick] (1,0.2) -- (0.89,1.37);
	\draw [thick] (1.8,0.8) -- (2.7,1.3);
	\draw [thick] (0,0) -- (1,0.2) -- (1.8,0.8) -- (2.846,0.948);
	\node [below] at (0.004,-1.3) {$e$};
	\draw [fill] (0.7,2.91) circle [radius = 0.03];
	\draw [fill] (2.7,1.3) circle [radius = 0.03];
	\draw [fill] (0.004,-1.3) circle [radius = 0.03];
	\draw [fill] (2.12,2.12) circle [radius = 0.03];
	\draw [fill] (0,3) circle [radius = 0.03];	
	\draw [fill] (2.846,0.948) circle [radius = 0.03];
	\draw [thick] (0,0) -- (0.89,1.37) -- (1.2,2) -- (2.12,2.12);
	\draw [thick] (1.2,2) -- (1.4,2.65);
	\draw [fill] (1.4,2.65) circle [radius = 0.03];
		\node [right] at (3,-2) {$\piS$};
 \filldraw [red](0,3) circle [radius=2pt];
 \filldraw [red](0.7,2.91) circle [radius=2pt];
\filldraw [red](2.846,0.948) circle [radius=2pt];
\filldraw [red](2.7,1.3) circle [radius=2pt];
\filldraw [red](2.12,2.12) circle [radius=2pt];
\filldraw [red] (1.4,2.65) circle [radius=2pt];
\end{scope}

\begin{scope}[scale=0.4,shift = {(-12,0)}]
	\draw [thick] circle [radius = 3];
	\draw circle [radius=.75];
	\draw [fill] circle [radius = 0.03];	
	\node [left] at (0,0) {$e$};
	\draw[thick] (-0.1,-0.3 ) -- (0.5,0.1); 
	\draw [thick] (0.004,-1.3) -- (0.1,-0.7) -- (-0.1,-0.3 ) -- (0,0);
	\draw [thick] (0,0) -- (-0.3,1) -- (0.2,2) -- (0,3);
	\draw [thick] (0.2,2) -- (0.4, 2.7) -- (0.7,2.91);
	\draw [thick] (1,0.2) -- (0.89,1.37);
	\draw [thick] (0.2,2) -- (0.4, 2.7);
	\draw [thick] (1.8,0.8) -- (2.7,1.3);
	\draw [thick] (0,0) -- (1,0.2) -- (1.8,0.8) -- (2.846,0.948);
	\node [below] at (0.004,-1.3) {$\g^{-1}$};
	\draw [fill] (0.7,2.91) circle [radius = 0.03];
	\draw [fill] (2.7,1.3) circle [radius = 0.03];
	\draw [fill] (0.004,-1.3) circle [radius = 0.03];
	\draw [fill] (2.12,2.12) circle [radius = 0.03];
	\draw [fill] (0,3) circle [radius = 0.03];	
	\draw [fill] (2.846,0.948) circle [radius = 0.03];	
	\draw [thick] (0,0) -- (0.89,1.37) -- (1.2,2) -- (2.12,2.12);
	\draw [thick] (1.2,2) -- (1.4,2.65);
	\draw [fill] (1.4,2.65) circle [radius = 0.03];
	\node [right] at (3,-2) {$\piS$};
	\draw [->] (3.5,3) to [out = 180,in = 90] (1.5,1.2);
	\node [right] at (3.5,3) {$\cone^{c}(\g)$};
	\end{scope}
	\end{tikzpicture}
	\caption{The coarse cone type  of $\g\in\G$ (left). The set $\g\cdot\cone^{c}_\infty(\g)$ (right). Pictures  from P.-S.-Wienhard \cite{PSW1}.}\label{figure:conetype}
\end{figure}

We let $\xi:\bord\piS\to\bord\Omega$ be the natural identification via the action of $\rho(\piS)$ on $\Omega,$ and analogously $\ov\xi:\piS\to\bord\Omega^*.$ We denote by $\varXi:=(\xi,\ov \xi): \piS\to\bord\Omega\times\bord\Omega^*$ the flag curve. Consider $x\in\bord\piS$ and let $\alpha_i\to x$ be a geodesic ray on $\piS.$ The following fact is a consequence of Proposition \ref{conetypesBalls}.

\begin{fact} For big enough $i,$ the subset $\xi\big(\alpha_i\cone^{c}_\infty(\alpha_i)\big)\subset\bord\Omega$ is coarsely the intersection of a ball of radius $e^{-\slroot_1(\alpha_i)}$ about $\xi(x)$ with $\bord\Omega.$ By duality, one has $\ov\xi\big(\alpha_i\cone^{c}_\infty(\alpha_i)\big)\subset\bord\Omega^*$ is coarsely the intersection of a ball of radius $e^{-\slroot_2(\alpha_i)}$ about $\ov\xi(x)$ with $\bord\Omega^*.$
\end{fact}

The coarse constants and the minimal length $i$ required in the above statement depend only on the representation and not on the point $x.$ 

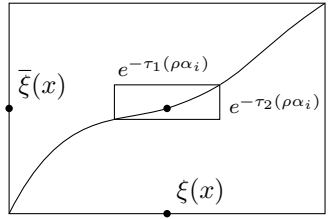
\begin{figure}[h]\centering
\begin{tikzpicture}

\begin{scope}[rotate=0,scale=0.7,shift={(-3,-2)}]

\draw (0,0) rectangle (6,4);
\fill (3,0) circle (2pt);
\fill (0,2) circle (2pt);
\fill (3,2) circle (2pt);


\draw[name path=c1,-] (0,0) ..  controls (1,2)  and (2,1.7)..   (3,2);
\draw[name path=c2,-] (3,2) .. controls (4.5,2.5) and (4.5,3) .. (6,4);

\draw[name path=l1,white] (2,0.2)--(2,3.8);
\draw[name path=l2,white] (4,0.2)--(4,3.8);

\draw[name intersections={of=c1 and l1, by={a}},
name intersections={of=c2 and l2, by={b}}] (a) rectangle (b);

\node [above left] at (b) {\footnotesize{$e^{-\slroot_1(\rho\alpha_i)}$}};
\node [below right] at (b) {\footnotesize{$e^{-\slroot_2(\rho\alpha_i)}$}};



\begin{scope}[shift={(3,2)}]
\end{scope}

\node [above right] at (3,0) {$\footnotesize{\xi(x)}$};
\node [above right] at (0,2) {$\footnotesize{\ov\xi(x)}$};


%



\end{scope}

\end{tikzpicture}
\caption{The image of the cone type $\alpha_i\cone^{c}_\infty(\alpha_i)$ by the graph curve $\varXi$ in the $\class^{1+\nu}$-torus $\bord\Omega\times\bord\Omega^*.$}\label{figg}
\end{figure}

For any point $x\in\bord\piS$ we  distinguish two complementary situations that don't depend on the choice of the geodesic ray $(\alpha_i)_{i\in\N}$ converging to $x$:

\begin{itemize}
	\item[i)] For all $R>0$ there exists $N\in\N$ with $|\slroot_1(\cartan(\alpha_i))-\slroot_2(\cartan(\alpha_i))|\geq R$ for all $i\geq N;$
	\item[ii)] There exists $R>0$ and an infinite set of indices $\I\subset\N$ such that for all $k\in\I$ one has $|\slroot_1(\cartan(\alpha_k))-\slroot_2(\cartan(\alpha_k))|\leq R.$ We say in this case that $x$ is $\flat$\emph{-conical} ($\flat$ stands for 'barycenter of the chamber').
\end{itemize}

In the first case one is easily convinced by looking at Figure \ref{figg} that the rectangle becomes flatter along one of its sides  (see \S\,\ref{s.Hffproof} for details in the general case). Furthermore, since $\slroot_1(\cartan(\alpha_i))-\slroot_1(\cartan(\alpha_{i+1}))$ is uniformly bounded, its sign is eventually constant, and thus the longer side only depends on the point. As a result $x$ is necessarily a differentiability point of the graph curve $\varXi,$ with either horizontal or vertical derivative. 

We are thus bound to understand the set of $\flat$-conical points. We show (see Corollary \ref{nondiffbconicalR}):

\begin{fact}
	The non-differentiabilty points of the curve $\varXi(\bord\piS)$ and the $\flat$-conical points coincide.
\end{fact}
The main idea for this is to show that if a $\flat$-conical point $x$ were a differentiability point, then the derivative could not be horizontal nor vertical, and thus (by Proposition \ref{c.indepPer}) $\Xi$ would be bi-Lipschitz. In turn, this would force the periods of the two roots to agree, which in turn would imply that the representation is Fuchsian, contradicting the assumption that $\Omega$ is not an ellipse.

It remains to understand the Hausdorff dimension of the set of $\flat$-conical points. The upper bound (Proposition \ref{upper}) 
\begin{equation}\label{u}
	\Hff\big(\{\flat-\textrm{conical}\}\big)\leq\hC{\max\{\slroot_1,\slroot_2\}}
\end{equation} 

\noindent follows readily: since for a $\flat$-conical point the lengths $e^{-\slroot_1(\alpha_k)}$ and $e^{-\slroot_2(\alpha_k)}$ are comparable independently on $k\in\I,$ one can replace the rectangle in Figure \ref{figg} by the (smaller) square of length $$e^{-\max\{\slroot_1(\cartan(\alpha_k)),\slroot_2(\cartan(\alpha_k))\}}$$ 

\noindent and still get a covering\footnote{Choosing the longer side $e^{-\min\{\slroot_1(\cartan(\alpha_k)),\slroot_2(\cartan(\alpha_k))\}}$ gives the bound {$\Hff\varXi(\bord\piS)\leq 1.$}} (this time by balls on the $L^\infty$ metric) of the set $\{\flat-\textrm{conical}\}.$ Standard arguments on Hausdorff dimension give  Equation \eqref{u}.

Finding a lower bound for the Hausdorff dimension is more subtle; we use here an appropriate Patterson-Sullivan measure  to study how the mass of a ball of radius $r$ scales with $r$.

Since $\varXi(\bord\piS)$ is a subset the full flag space $\cal F(\R^3)$ and $$\|v\|_\infty:=\max\{|\slroot_1(v)|,|\slroot_2(v)|\}$$ is a norm on $\a_{\PSL(3,\R)},$ we can apply results by Quint \cite{Quint-Div} to determine a linear form $\varphi^\infty_\flat\in\a^*$ whose associated growth direction is the barycenter  $\scr{b}=\ker(\slroot_1-\slroot_2)$. By Quint \cite[Proposition 3.3.3]{Quint-Div}
$$	\hC{\max\{\slroot_1,\slroot_2\}}=\|\varphi^\infty_\flat\|^1,$$
\noindent where $\|\,\|^1$ is the operator norm on $\a^*$ defined by $\|\,\|_\infty,$ which turns out to be the $L^1$ norm $\|a\slroot_1+b\slroot_2\|^1=|a|+|b|.$
The form $\varphi^\infty_\flat$ additionally admits an associated \emph{Patterson-Sullivan} probability measure, namely a measure  $\ps^\infty$ such that for all $\g\in\piS$ one has (see Corollary \ref{existe})
\begin{equation}\label{radi}
	\ps^\infty\big(\varXi(\g\cone^{c}_\infty(\g))\big)\leq Ce^{-\varphi^\infty_\flat(\cartan(\g))}.
\end{equation}

A key extra information available in the case  of $\PSL(3,\R)$ is that the form $\varphi^\infty_\flat$ is explicit and doesn't depend on $\rho$. For this we need a small parenthesis on the \emph{critical hypersurface} $\cal Q_\rho$ of $\rho,$ depicted in Figure \ref{calQ},
and  characterized by $$\cal Q_\rho=\big\{\varphi\in\a^*: \hC\varphi =1\big\},$$ 
where the \emph{critical exponent} of a functional $\varphi\in\a^*$ is $$\hC\varphi:=\lim_{t\to\infty}\frac1t\log\#\big\{\g\in\piS:\varphi(\cartan(\g))\leq t\}\in(0,\infty].$$ 
\noindent
The critical hypersurface $\cal Q_\rho\subset\a^*$ is a closed analytic curve that bounds a strictly convex set (S. \cite{exponential} and Potrie-S. \cite{exponentecritico}), and thus by Quint \cite{Quint-Div}, the linear form $\varphi^\infty_\flat$ is uniquely determined by 
\begin{equation}\label{vi=norm}
\|\varphi^\infty_\flat\|^1=\inf\big\{\|\varphi\|^1:\varphi\in\cal Q_\rho\big\}.
\end{equation}

\begin{figure}[ht!]\centering

\includegraphics[width=90mm]{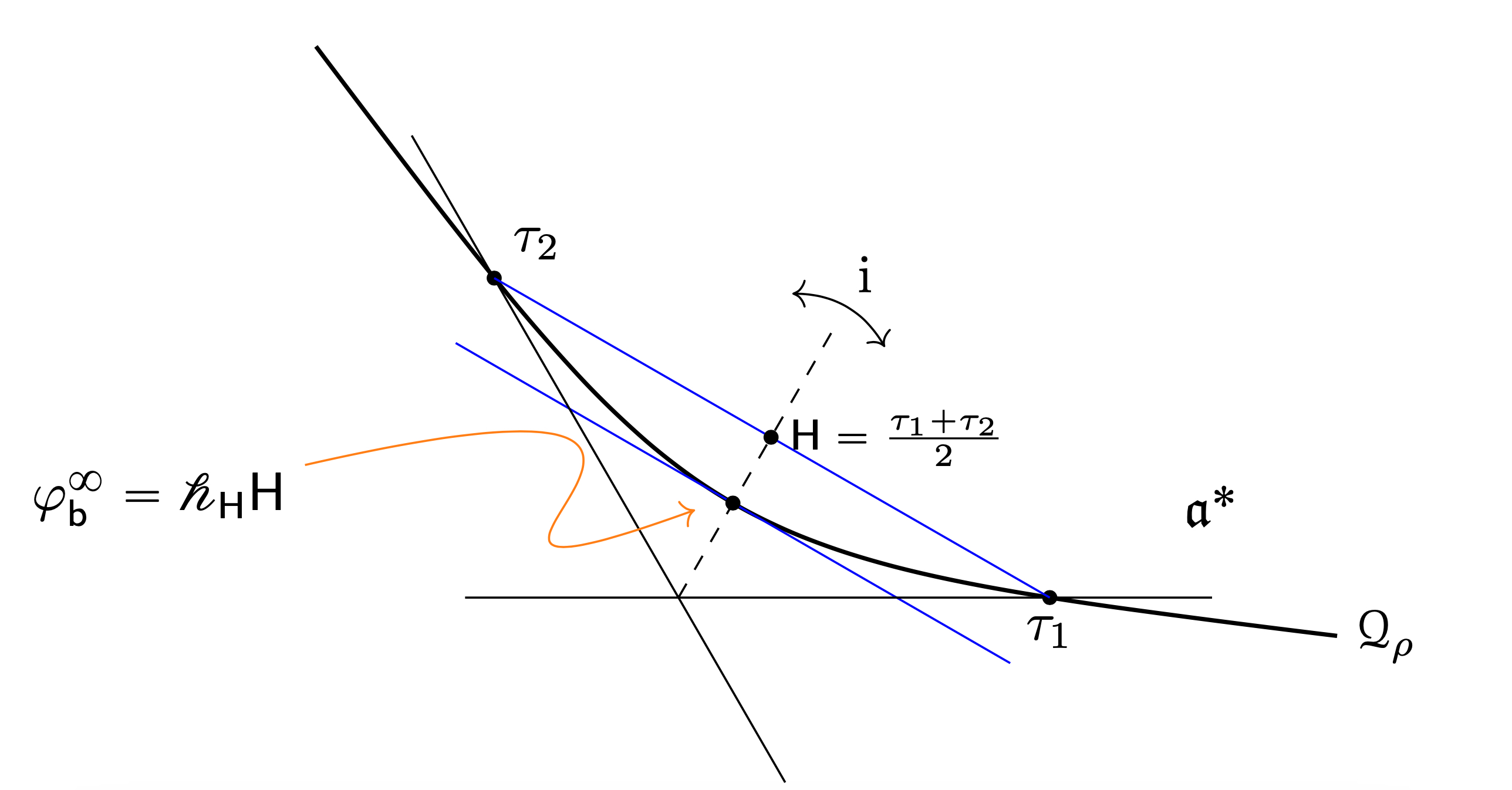}

  \caption{The critical hypersurface of a strictly convex projective structure on $S.$ Since $\sf H$ is a convex combination of $\{\slroot_1,\slroot_2\}$ one has $\|\sf H\|^1=1$ and thus $\|\varphi^\infty_\flat\|^1=\hC{\sf H}.$}\label{calQ}
\end{figure}

Again by \cite{exponentecritico} one has $\{\slroot_1,\slroot_2\}\subset\cal Q_\rho$. Since both $\cal Q_\rho$ and the norm $\|\,\|^1$ are invariant by the opposition involution $\ii$ (see again Figure \ref{calQ}) we deduce that, if we let $\sf H=(\slroot_1+\slroot_2)/2,$ then  \begin{equation}\label{>min}
	\varphi^\infty_\flat=\hC{\sf H}\cdot\sf H\geq \hC{\sf H}\min\{\slroot_1,\slroot_2\}.
\end{equation}
In particular, using Equation \eqref{vi=norm},  we obtain that $\hC{\max\{\slroot_1,\slroot_2\}}=\hC{\sf H}$. Moreover, since the geodesic flow is Anosov (by Benoist \cite{convexes1}) we can apply Bowen's characterization of entropy \cite{axiomA} (and Remark \ref{r.entropy=crit}), to obtain that the Hilbert entropy $\hJ{\sf H}=\hC{\sf H}$.

After this small parenthesis on the critical hypersurface, we come back to the lower bound on the Hausdorff dimension. 
Since $\varXi$ is a graph, $\varXi(\bord\piS)$ has the same intersection with the rectangle in Figure \ref{figg} than with the larger square of size 
$$e^{-\min\{\slroot_1(\cartan(\alpha_i)),\slroot_2(\cartan(\alpha_i))\}};$$  
this square is now a ball (for the $L^\infty$ metric) of radius $e^{-\min\{\slroot_1(\cartan(\alpha_i)),\slroot_2(\cartan(\alpha_i))\}}.$ Thus for all $i,$ $\varXi\big(\alpha_i\cone^{c}_\infty(\alpha_i)\big)$ is coarsely a ball of the latter radius and one has 
\begin{alignat*}{2}
	\ps^\infty\big(B(\varXi(x),e^{-\min\{\slroot_1(\cartan(\alpha_i)),\slroot_2(\cartan(\alpha_i))\}}\big) & \leq \ps^\infty\big(\varXi(\alpha_i\cone^{c}_\infty(\alpha_i)\big)\leq Ce^{-\varphi^\infty_\flat(\cartan(\alpha_i))}\\ 
	& \leq C\big(e^{-\min\{\slroot_1(\cartan(\alpha_i)),\slroot_2(\cartan(\alpha_i))\}}\big)^{\hC{\sf H}},
\end{alignat*}

\noindent
where the last inequalities follow from Equations \eqref{radi} and \eqref{>min}. This  gives a possibly bigger constant $C'$ such that, for all $r$, 
$$\ps^\infty\big(B(\varXi(x),r)\big)\leq C'r^{\hC{\sf H}}.$$ 
Again, classical Hausdorff dimension arguments (c.f. Corollary \ref{edgar} below) give that, for any measurable subset $E\subset\varXi(\bord\piS)$ with full $\ps^\infty$ mass, one has $\Hff(E)\geq \hC{\sf H}.$

Since $\PSL(3,\R)$ has rank smaller than 3 and $\rho$ is $\simple$-Anosov we can apply Burger-Landesberg-Lee-Oh \cite[Theorem 1.6]{BLLO} to obtain that $\ps^\infty(\{\flat\textrm{-conical}\})=1$ and thus we have the desired lower bound $$\Hff\big(\{\flat\textrm{-conical}\}\big)\geq \hC{\sf H},$$ 
which combined with the upper bound \eqref{u} and the equality $\hC{\max\{\slroot_1,\slroot_2\}}=\hC{\sf H},$ gives the proof of Corollary \ref{hilbert}.

In the general case \cite[Theorem 1.6]{BLLO} is not applicable and we replace it with Theorem \ref{Ahyper}.
\qed\medskip

\subsection*{Structure of the paper}
The preliminaries of the paper are standard facts about linear algebraic groups, recalled in \S\ref{Liegroups}, the work of S. \cite{dichotomy} about linear cocycles over the boundary of a hyperbolic group (in \S\ref{cocycles}), as well as basic facts about Anosov representations and their Patterson-Sullivan theory  recalled from \cite{GW, BPS, PSW1, dichotomy} in the first part of \S\ref{Anosov}.  In the rest of \S\ref{Anosov} we prove Theorem \ref{hyper} a more precise statement than Theorem \ref{Ahyper}, discussing the Patterson-Sullivan measure of $(\sf W,\varphi)$-conical points.  The heart of the proof is to construct and study a rank 2 flow whose recurrence set is related to $(\sf W,\varphi)$-conical points.

In \S\ref{s.Hffbcon} we consider two locally conformal representations. We prove Theorem \ref{Hffconical}, stating that for such a pair the Hausdorff dimension of the set of $\hol$-conical points belongs to $$[\hol\hC{\infty,\hol},\min\{\hC{\infty,\hol},\hol\hC{\infty,\hol}+1-\hol\}].$$ The lower bound is obtained by analyzing properties of the linear form $\vi$ whose associated growth direction is $(\hol,1)$; its Patterson-Sullivan measure $\mu^{\vi}$ gives full mass to the set of $\hol$-conical points thanks to Theorem \ref{hyper}. Using cone-types we can show that for a fine set of balls $\mu^{\vi}(B(x,r))\leq Cr^{-\hol\hC{\infty,\hol}}$. The upper bound uses results of \cite{PSW1} to construct a fine covering of the set of $\hol$-conical points with balls of radius $e^{-\max\{\hol\slroot,\ov\slroot\}}$. In \S\ref{teoLCdiff} we prove Theorem \ref{LCdiff}.

In \S\ref{s.6} we prove that if the graph map between $\R$-hyperconvex representations has an oblique derivative, then the  map is bi-Lipschitz (Proposition \ref{c.indepPer}). This only relies on basic properties of hyperconvex representations, and is crucial for the proof of Theorem \ref{tutti}, achieved in \S\ref{s.Hffproof}, as it allows the identification of $\flat$-conical points and points of non-differentiability.

\subsection*{Acknowledgements} We thank Katie Mann, Gabriele Viaggi, Anna Wienhard and Maxime Wolff  for insightful conversations and  Andr\'es Navas for pointing us to useful literature.

\section{Linear algebraic groups}\label{Liegroups}

Throughout the text $\sf G$ will denote a real-algebraic semi-simple Lie group of  non-compact type and $\ge$ its Lie algebra.

\subsection{Linear algebraic groups} Fix a Cartan involution $o:\ge\to\ge$ with associated Cartan decomposition $\ge=\k\oplus\p.$ Let $\a\subset\p$ be a maximal abelian subspace and let $\roots\subset\a^*$ be the set of restricted roots of $\a$ in $\ge.$ For $\sroot\in\roots$, we denote by $$\ge_\sroot=\{u\in\ge:[a,u]=\sroot(a)u\;\forall a\in\a\}$$ its associated root space. The (restricted) root space decomposition is $\ge=\ge_0\oplus\bigoplus_{\sroot\in\roots}\ge_\sroot,$ where $\ge_0$ is the centralizer of $\a.$ Fix a Weyl chamber $\a^+$ of $\a$ and let $\roots^+$ and $\simple$ be, respectively, the associated sets of positive and simple roots. Let $\Weyl$ be the Weyl group of $\roots$ and $\ii:\a\to\a$ be the opposition involution: if $u:\a\to\a$ is the unique element in $\Weyl$ with $u(\a^+)=-\a^+$ then $\ii=-u.$

We denote by $(\cdot,\cdot)$ both the Killing form of $\ge,$ its restriction to $\a,$ and its associated dual form on $\a^*$, the dual of $\a$. For $\chi,\psi\in\a^*$ let $$\<\chi,\psi\>=2\frac{(\chi,\psi)}{(\psi,\psi)}.$$ 

\noindent
The \emph{restricted weight lattice} is defined by 
$$\poids=\{\varphi\in\a^*:\<\varphi,\sroot\>\in\Z\;\forall\sroot\in\roots\}.$$ 
It is spanned by the \emph{fundamental weights} $\{\peso_\sroot:\sroot\in\simple\}$, defined by \begin{equation}\label{pesoFund}
	\<\peso_\sroot,\bb\>=d_\sroot\delta_{\sroot\bb}
\end{equation} 
for every $\sroot,\bb\in\simple,$ where $d_\sroot=1$ if $2\sroot\notin\roots^+$ and $d_\sroot=2$ otherwise.

A subset $\t\subset\simple$ determines a pair of opposite parabolic subgroups $\sf P_\t$ and $\check{\sf P}_\t$ whose Lie algebras are
\begin{alignat*}{2}\frak p_\t & =\bigoplus_{\sroot\in\roots^+\cup\{0\}}\frak g_{\sroot} \oplus\bigoplus_{\sroot\in\<\simple-\t\>}\frak g_{-\sroot},\\ \check{\frak p}_\t & =\bigoplus_{\sroot\in\roots^+\cup\{0\}}\frak g_{-\sroot} \oplus\bigoplus_{\sroot\in\<\simple-\t\>}\frak g_{\sroot}.
\end{alignat*}

\noindent
The group $\check{\sf P}_\t$ is conjugated to the parabolic group $\sf P_{\ii\t}.$  We denote the \emph{flag space} associated to $\t$ by $\cal F_\t=\sf G/\sf P_\t.$ The $\sf G$ orbit of the pair $([\sf P_{\t}],[\check{\sf P}_{\t}])$ is the unique open orbit for the action of $\sf G$ in the product $\cal F_\t\times\cal F_{\ii\t}$ and is denoted by $\posgen_\t.$

\subsection{Cartan and Jordan projection}

Denote by $\sf K=\exp\k$ and $\sf A=\exp\a.$ The \emph{Cartan decomposition} asserts the existence of a continuous map $\cartan:\sf G \to \a^+,$ called the \emph{Cartan projection}, such that every $g\in\sf G$ can be written as $g=ke^{a(g)}l$ for some $k,l\in\sf K.$

We will need the following uniform continuity of the Cartan projection:

\begin{prop}[{Benoist \cite[Proposition 5.1]{Benoist-HomRed}}]\label{p.Ben}
	For any compact $ L\subset \sf G$ there exists a compact set $H\subset \frak a$ such that, for every $g\in \sf G$, one has
	$$\cartan(LgL)\subset \cartan(g)+H.$$ 	
\end{prop}

By the Jordan's decomposition,  every element $g\in\sf G$ can be uniquely written as a commuting product $g=g_eg_{ss}g_u$ where $g_e$ is conjugate to an element in $\sf K,$ $g_{ss}$ is conjugate to an element in $\exp(\a^+)$ and $g_u$ is unipotent. The \emph{Jordan projection} $\lambda=\lambda_{\sf G}:\sf G\to\a^+$ is the unique map  such that $g_{ss}$ is conjugated to $\exp\big(\lambda(g)\big).$

\begin{defi} Let $\Gamma\subset\sf G$ be a discrete subgroup, then its \emph{limit cone} $\cal L_\Gamma$ is the smallest closed cone of the closed Weyl chamber $\a^+$ that contains $\{\lambda(g):g\in\Gamma\}$.
\end{defi}

We will need the following results by Benoist.

\begin{thm}[{Benoist \cite{limite,benoist2}}]\label{densidad} Let $\Gamma\subset\sf G$ be a Zariski-dense sub-semigroup, then its limit cone $\cal L_\Gamma$ has non-empty interior. Moreover, the group generated by the Jordan projections $\{\lambda(g):g\in\Gamma\}$ is dense in $\frak a.$
\end{thm}

\subsection{Representations of $\sf G$}\label{representaciones}
The standard references for the following are Fulton-Harris \cite{FultonHarris}, Humphreys  \cite{james} and Tits \cite{Tits}.

Let $\Fund:\sf G\to\PGL(V)$ be a finite dimensional rational\footnote{Namely a rational map between algebraic varieties.} irreducible representation and denote by $\phi_\Fund:\frak g\to\frak{sl}(V)$ the Lie algebra homomorphism associated to $\Fund.$ The \emph{weight space} associated to $\chi\in\a^*$ is the vector space 
$$V_\chi=\{v\in V:\phi_\Fund(a) v=\chi(a) v\ \forall a\in\sf A \}.$$ 

\noindent
We say that $\chi\in\a^*$ is a \emph{restricted weight} of $\Fund$  if $V_\chi\neq0$.  Tits \cite[Theorem 7.2]{Tits} states that the set of weights has a unique maximal element with respect to the partial order $\chi\succ\psi$ if $\chi-\psi$ is a $\N$-linear combination of positive roots. This is called \emph{the highest weight} of $\Fund$ and denoted by $\chi_\Fund.$ By definition, for every $g\in\sf G$ one has 
\begin{equation}\label{eq:spectralrep}
\lambda_1\big(\Fund(g)\big)=\chi_\Fund(\lambda(g)\big),
\end{equation} 
where $\lambda_1$ is the  logarithm of the spectral radius of $\Fund(g)$.

We denote by $\poids(\phi)$ the set of restricted weights of the representation $\phi_\Fund$ $$\poids(\phi)=\big\{\chi\in\a^*:V_\chi\neq\{0\}\big\},$$ these are all bounded above by $\chi_\Fund$ (see for example Humphreys \cite[\S13.4 Lemma B]{james}),
namely every weight $\chi\in\poids(\phi)$ has the form $$\chi_\Fund-\sum_{\sroot\in\simple} n_\sroot\sroot\textrm{ for }n_\sroot\in\N.$$ The \emph{level} of a weight $\chi$ is the integer $\sum_\sroot n_\sroot$, the highest weight is thus the only weight of level zero.
Additionally, if $\chi\in\poids(\phi_\Fund)$ and $\sroot\in\roots^+$ then the elements of $\poids(\phi_\Fund)$ of the form $\chi+j\sroot,\, j\in\Z$ form an unbroken string $$\chi+j\sroot,\, j\in\lb-r,q\rb$$ and $r-q=\<\chi,\sroot\>.$ One can then recover algorithmically the set $\poids(\phi_\Fund)$ level by level starting from $\chi_\Fund$, as follows: 

\begin{itemize} 
	\item[-] Assume the set of weights of level at most $ k$ is known and consider a weight $\chi$ of level $k$. 
	\item[-] For each $\sroot\in\simple$ compute $\<\chi,\sroot\>$, this gives the length $r-q$ of the $\sroot$-string through $\chi$. The weights of the form $\chi +j\sroot$, for positive $j$, have level smaller than $k$ and are thus known, thus we can decide whether $\chi-\sroot$ is a weight or not, determining the set of weights of level $k+1.$
\end{itemize} 

The following lemma  follows at once from the algorithmic description above.
Let $\ge=\bigoplus_i\ge_i$ be the decomposition in simple factors of a semi-simple real Lie algebra of non-compact type. Recall that if $\a_i \subset\ge_i$ is a Cartan subspace, then $\a=\bigoplus_i\a_i$ is a Cartan subspace of $\ge$. Any $\varphi\in(\a_i)^*$ extends to a functional on $\a$, still denoted $\varphi$, by vanishing on the remaining factors. The restricted root system of $\ge$ is then $\simple_\ge=\bigcup \simple_{\ge_i}.$ The \emph{associated simple factor} to  $\sroot\in\simple_\ge$ is  $\ge_i$ such that $\sroot\in\simple_i.$

\begin{lemma}\label{simple} Let $\ge$ be a semi-simple real Lie algebra of  non-compact type and $\phi$ be an irreducible representation  of $\ge$ whose highest restricted weight is a multiple of a fundamental weight, $\chi_\phi=k\peso_\sroot$ for some $\sroot\in\simple$. Then $\phi$ factors as a representation of the simple factor associated to $\sroot$.
\end{lemma}

\begin{proof} Proceeding by induction on the levels of $\phi$, one readily sees that for every $\slroot\in\simple_j$ for $j\neq i$ and all $\chi\in\poids(\phi)$ one has $\<\chi,\slroot\>=0$. Thus the associated root space $(\ge_j)_{-\slroot}$ acts trivially on every weight space of $\phi$ and so the whole factor $\ge_j$ acts trivially.\end{proof}

The following set of simple roots plays a special role in  representation theory.

\begin{defi}\label{trep}
	Let $\Fund:\sf G\to\PGL(V)$ be a representation. We denote by $\t_\Fund$ the set of simple roots $\sroot\in\simple$ such that $\chi_\Fund-\sroot$ is still a weight of $\Fund$. Equivalently 
	 \begin{equation}\label{peso<}
		\t_\Fund=\big\{\sroot\in\simple:\<\chi_\Fund,\sroot\>\neq0\big\}.
	\end{equation}
\end{defi}

The following lemma will be needed in the proof of Theorem \ref{t.Zcl}.

\begin{lemma}\label{not} Let $\ge$ be semi-simple of non-compact type and $\phi:\ge\to\frak{gl}(V)$ an irreducible representation. For $\sroot\in\t_\phi$ and $v\in V_{\chi_\phi}-\{0\}$, the map $n\mapsto \phi(n)v$ is injective when defined on $\ge_{-\sroot}.$
\end{lemma}

\begin{proof} By definition of $\chi_\phi$ every $n\in\ge_{\sroot}$ acts trivially on $V_{\chi_\phi}$. For $y\in\ge_{-\sroot}-\{0\}$, there exists  $x\in\ge_\sroot$ such that $\{x,y,h_\sroot\}$ spans a Lie algebra isomorphic to $\sl_2(\R)$, where $h_\sroot$ is defined by $\varphi(h_\sroot)=\<\varphi,\sroot\>$ for all $\varphi\in\a^*$. If $\phi(y)v=0$ then, since $\phi(x)V_{\chi_\phi}=0$ one concludes $\phi(h_\sroot)v=0$ and since $V_{\chi_\phi}$ is a weight-space one has $\phi(h_\sroot)V_{\chi_\phi}=0$. This in turn implies that 
	$$\<\chi_\phi,\sroot\>=\chi_\phi(h_\sroot)=0,$$ 
	contradicting that $\sroot\in\t_\phi.$
\end{proof}

We denote by $\|\,\|_\Fund$ an Euclidean norm on $V$ invariant under $\Fund \sf K$ and such that $\Fund\sf A $ is self-adjoint, see for example Benoist-Quint's book \cite[Lemma 6.33]{BQLibro}. By definition of $\chi_\Fund$ and $\|\,\|_\Fund$, and Equation \eqref{eq:spectralrep} one has, for every $g\in \sf G, $ that  
\begin{equation}\label{eq:normayrep}
\log\|\Fund g\|_\Fund=\chi_\Fund(\cartan(g)).
\end{equation} 
Here, with a slight abuse of notation, we denote by $\|\,\|_\Fund$ also the induced operator norm, which doesn't depend on the scale of $\|\,\|_\Fund$.

Denote by $W_{\chi_\Fund}$ the $\Fund\sf A $-invariant complement of $V_{\chi_\Fund}.$ The stabilizer in $\sf G $ of $W_{\chi_\Fund}$ is $\wk{\sf P}_{\t_\Fund},$ and thus one has a map of flag spaces 
\begin{equation}\label{maps}
	(\zeta_\Fund,\zeta^*_\Fund):\cal F_{\t_\Fund}^{(2)}(\sf G )\to \mathrm{Gr}_{\dim V_{\chi_\Fund}}^{(2)}(V).
\end{equation} This is a proper embedding which is an homeomorphism onto its image. Here, as above, $\mathrm{Gr}_{\dim V_{\chi_\Fund}}^{(2)}(V)$ is the open $\PGL (V )$-orbit in the product of the Grassmannian of $(\dim V_{\chi_\Fund})$-dimensional subspaces and the Grassmannian of $(\dim V-\dim V_{\chi_\Fund})$-dimensional subspaces. One has  the following proposition (see also Humphreys \cite[Chapter XI]{LAG}).

\begin{prop}[Tits \cite{Tits}]\label{prop:titss} 
For each $\sroot\in\simple$ there exists a finite dimensional rational irreducible representation $\Fund_\sroot:\sf G\to\PSL(V_\sroot),$ such that $\chi_{\Fund_\sroot}$ is an integer multiple $l_\sroot\peso_\sroot$ of the fundamental weight and $\dim V_{\chi_{\Fund_\sroot}}=1.$ \end{prop}

We will fix from now on such a set of representations and call them, for each $\sroot\in\simple,$ the \emph{Tits representation associated to $\sroot$}.

\subsection{The center of the Levi group $\sf P_{\t}\cap\check{\sf P}_{\t}$}\label{s.Levi}

We now consider the vector subspace 
$$\a_\t=\bigcap_{\sroot\in\simple-\t}\ker\sroot.$$ 
Denoting by $\Weyl_\t=\{w\in \Weyl:w(v)=v\quad \forall v\in\a_\t\}$ the subgroup of the Weyl group generated by reflections associated to roots in $\simple-\t$, there is a unique projection $\pi_\t:\a\to\a_\t$ invariant under  $\Weyl_\t$.

The dual $(\a_\t)^*$ is canonically identified with the subspace of $\a^*$ of $\pi_\t$-invariant linear forms. Such space is spanned by the fundamental weights of roots in $\t$,
$$(\a_\t)^*=\{\varphi\in\a^*:\varphi\circ\pi_\t=\varphi\}=\<\peso_\sroot|\a_\t:\sroot\in\t\>.$$ 
We will  denote, respectively, by \begin{alignat*}{2}\cartan_\t& =\pi_\t\circ \cartan:\sf G\to \a_\t\ \\ \lambda_\t & =\pi_\t\circ\lambda:\sf G\to \a_\t,\end{alignat*} the compositions of the Cartan and Jordan projections with $\pi_\t$.

\subsection{The Buseman-Iwasawa cocycle}

 The \emph{Iwasawa decomposition} of $\sf G$ states that every $g\in\sf G$ can be written uniquely as a product $lzu$ with $l\in\sf K,$ $z\in\sf A$ and $u\in\sf U_\simple,$ where $\sf U_\simple$ is the unipotent radical of $\sf P_{\simple}.$

The \emph{Buseman-Iwasawa cocycle} of $\sf G$ is the map $\bus:\sf G\times\cal F\to\frak a$ such that, for all $g\in\sf G$ and $k[\sf P_\simple]\in\cal F,$
$$\bus(g,k[\sf P_\simple])=\log(z)$$
 where $\log:\sf A\to \frak a$ denotes the inverse of the exponential map, and $gk=lzu$ is the Iwasawa decomposition of $gk$. Quint \cite[Lemmes 6.1 and 6.2]{Quint-PS} proved that the function $\bus_\t=\pi_\t\circ\bus$ factors as a cocycle $\bus_\t:\sf G\times \cal F_\t\to\frak a_\t.$

The Buseman-Iwasawa cocycle can also be read from the representations of $\sf G.$ Indeed, Quint \cite[Lemme 6.4]{Quint-PS} shows that for every $g\in\sf G$ and $x\in\cal F_\t$ one has 
\begin{equation}\label{busnorma}
	l_\sroot\peso_\sroot(\bus(g,x))=\log\frac{\|\Fund_\sroot(g)v\|_\Fund}{\|v\|_\Fund},
\end{equation} where $v\in\zeta_{\Fund_\sroot}(x)\in\P(\sf V_\sroot)$ is non-zero, and $l_\sroot$ is as in Proposition \ref{prop:titss}.

\subsection{Gromov product and Cartan attractors}\label{GryBus} Let $\K$ be either $\C$ or $\R$. For a decomposition $\K^d=\ell\oplus V$ into a line $\ell$ and a hyperplane $V$ together with an inner (Hermitian) product $o$ on $\K^d$, one defines the \emph{Gromov product} by $$\Gr(V,\ell)=\Gr^o(V,\ell):=\log\frac{|\varphi(v)|}{\|\varphi\|\|v\|}=\log\sin\angle_o(\ell,V),$$ for any non-vanishing $v\in\ell$ and $\varphi\in (\K^d)^*$ with $\ker\varphi=V.$

This induces, for any semisimple Lie group $\sf G$ and subset $\t<\Delta$, a  \emph{Gromov product} $\Gr_\t:\posgen_\t\to\frak a_\t$ defined, for every $(x,y)\in\posgen_\t$ and $\sroot\in\t,$ by 
$$l_\sroot\peso_\sroot\big(\Gr_\t(x,y)\big)= \Gr^{\Fund_\sroot}(\zeta_{\Fund_\sroot}^*x,\zeta_{\Fund_\sroot} y)=\log\sin\angle_o\big(\zeta_{\Fund_\sroot} y,\zeta_{\Fund_\sroot}^*x),$$ where $\zeta^*_{\Fund_\sroot}$ and $\zeta_{\Fund_\sroot}$ are the equivariant maps from Equation (\ref{maps}), and the Hermitian product $o$ is induced by an Euclidean norm $\|\,\|_{\Fund_\sroot}$ invariant under $\Fund_\sroot\sf K$.

 From S. \cite[Lemma 4.12]{orbitalcounting} one has, for all $g\in\sf G$ and $(x,y)\in\posgen_\t,$ 
\begin{equation}\label{forGr}
\Gr_\t(gx,gy)-\Gr_\t(x,y)=-\big(\ii\bus_{\ii\t}(g,x)+\bus_\t(g,y)\big).
\end{equation}

If $g=k\exp(\cartan(g))l$ is a Cartan decomposition of $g\in\sf G$ we define its $\t$-\emph{Cartan attractor} (resp. \emph{repeller}) by $$U_\t(g)=k[\sf P_\t]\in\cal F_\t\quad\textrm{ and }\quad U_{\ii\t}(g^{-1})=l^{-1}[\check{\sf P}_\t]\in\cal F_{\ii\t}.$$

\noindent The \emph{Cartan basin of $g$} is defined, for $\alpha>0,$ by 
\begin{equation}\label{e.CartBasin}
	B_{\t,\alpha}(g)= \big\{x\in\cal F_\t:\peso_\sroot \Gr_\t\left(U_{\ii\t}(g^{-1}),x\right)>-\alpha,\; \forall \sroot\in\theta\big\}.
	\end{equation} 
\begin{obs}\label{fty}
	Observe that a statement of the form $\peso_\sroot\Gr_\t(x,y) \geq-\kappa$ for all $\sroot\in\t$ is a quantitative version (depending on the choice of $\sf K$) of the transversality between $x$ and $y;$  in particular it implies that $x$ and $y$ are transverse.
\end{obs}
Neither the Cartan attractor nor its basin are uniquely defined unless for all $\sroot\in\t$ one has $\sroot\big(\cartan(g)\big)>0,$ regardless one has the following:

\begin{obs} Given $\alpha>0$ there exists a constant $K_\alpha$ such that if $y\in\cal F_{\t}$ belongs to $B_{\t,\alpha}(g)$ then one has
\begin{equation}\label{comparision-shadow}
\big\|\cartan_\t(g)-\bus_\t(g,y)\big\|\leq K_\alpha.
\end{equation} Indeed, using Tits's representations of $\sf G$ and Equations \eqref{eq:normayrep} and \eqref{busnorma} this boils down to the elementary fact that if $A\in\GL_d(\R)$ verifies\footnote{Recall from Equation \eqref{e.slroot} that we denote by $\slroot_i$ the simple roots of $\GL_d(\RR)$} $\slroot_1(\cartan(A))>0$ then for every $v\in\R^d$ one has $$\log\frac{\|Av\|}{\|v\|}\geq \log\|A\|+\log\sin\angle\big(\R\cdot v,U_{d-1}(A^{-1})\big)$$ (see for example \cite[Lemma A.3]{BPS}).\end{obs}

\section{H\"older cocycles on $\bord\G$}\label{cocycles} 
Let $\G$ be a finitely generated group, and fix a finite generating set $S$. A group $\G$ is \emph{Gromov hyperbolic} if its Cayley graph $\sf{Cay}(\G, S)$ is a Gromov hyperbolic geodesic metric space. In this case we denote by $\bord\G$ its Gromov boundary, namely the equivalence classes of (quasi)-geodesic rays. It is well known that, up to Hölder homeomorphism, $\bord\G$  doesn't depend on the choice of the generating set $S$. We will  denote by $\bord^2\G$ the set of distinct pairs in $\bord\G$:
$$\bord^2\G:=\{(x,y)\in\bord\G\times\bord\G|\, x\neq y\}.$$

For a finitely generated, non-elementary, word-hyperbolic group $\G$ we denote by $\sf g=\big(\sf g_t:\sf U\G\to\sf U\G\big)_{t\in\R}$ the \emph{Gromov-Mineyev geodesic} flow of $\G$ (see Gromov \cite{gromov} and Mineyev \cite{mineyev}). Throughout this section we will have the same assumptions as in S. \cite[\S\,3]{dichotomy}, namely that $\sf g$ is metric-Anosov  and that the lamination induced on the quotient by $\widetilde{\cal W}^{cu}=\{(x,\cdot,\cdot)\in\widetilde{\sf U\G}\}$ is the central-unstable lamination of $\sf g.$ 

Since we will mostly recall needed results from S. \cite[\S\,3]{dichotomy} we do not overcharge the paper with the definitions of metric-Anosov and central-unstable lamination:  
by Bridgeman-Canary-Labourie-S. \cite{pressure}, word-hyperbolic groups admitting an Anosov representation verify the required assumptions.

\begin{defi}Let $V$ be a finite dimensional real vector space. A \emph{H\"older cocycle} is a function $c:\G\times\bord\G\to V$ such that: 
\begin{itemize}
	\item[-] for all $\g,h\in \G$ one has $c\big(\g h,x\big)=c\big(h,x\big)+c\big(\g,h(x)\big),$ 
	\item[-]there exists $\alpha\in(0,1]$ such that for every $\g\in\G$ the map $c(\g,\cdot)$ is $\alpha$-H\"older continuous.
\end{itemize}
\end{defi}

Recall that every \emph{hyperbolic element}\footnote{Namely an infinite order element} $\g\in\G$ has two fixed points on $\bord\G,$ the attracting $\g_+$ and the repelling $\g_-.$ If $x\in\bord\G-\{\g_-\}$ then $\g^nx\to\g_+$ as $n\to\infty.$ The \emph{period} of a H\"older cocycle for a hyperbolic $\g\in\G$ is $\ell_c(\g):=c\big(\g,\g^+\big).$ A cocycle $c^*:\G\times\bord\G\to\R$ is \emph{dual to $c$} if for every hyperbolic $\g\in\G$ one has $$\ell_{c^*}(\g)=\ell_c\big(\g^{-1}\big).$$

\subsection{Real-valued coycles}\label{Sreal}

Assume now $V=\R$ and consider a cocycle $\kappa$ with non-negative (and not all vanishing) periods. For $t>0$ we let $$\sf R_t(\kappa)=\big\{[\g]\in[\G]\textrm{ hyperbolic}:\ell_\kappa(\g)\leq t\big\}$$ and define the \emph{entropy} of $\kappa$ by 
$$\hJ\kappa=\limsup_{t\to\infty}\frac1t\log\#\sf R_t(\kappa)\in(0,\infty].$$ 
For such a cocycle consider the action of $\G$ on $\bord^2\G\times\R$ via $\kappa$: 
\begin{equation}\label{gaction}
	\g\cdot(x,y,t)=\left(\g x,\g y,t-\kappa\left(\g,y\right)\right).
\end{equation}

The following is a straightforward consequence of S. \cite[Theorem 3.2.2]{dichotomy}.

\begin{prop}\label{-infty} 
	Let $\kappa$ be a H\"older cocycle with non-negative periods and finite entropy. Then, the above action of $\G$ on $\bord^2\G\times\R$ is properly-discontinuous and co-compact. If moreover $c$ is another H\"older cocycle with non-negative periods and finite entropy then there exists a $\G$-equivariant bi-H\"older-continuous homeomorphism $E:\bord^2\G\times\R\to\bord^2\G\times\R$ which is an orbit equivalence between the $\R$-translation actions.
\end{prop}

We recall the notion of \emph{dynamical intersection}, a concept from Bridgeman-Canary-Labourie-S. \cite{pressure} for H\"older functions over a metric-Anosov flow, that can be pulled back to this setting via the existence of the \emph{Ledrappier potential} of $\kappa$ from S. \cite[\S\,3.1]{dichotomy}.

 The \emph{dynamical intersection} of two real valued cocycles $\kappa, c$ is
 \begin{equation}\label{defiII}\II(\kappa,c)=\lim_{t\to\infty}\frac1{\sf R_t(\kappa)}\sum_{\g\in\sf R_t(\kappa)}\frac{\ell_c(\g)}{\ell_\kappa(\g)}.
\end{equation} 
We record in the following proposition various needed facts about $\II$:

\begin{prop}[{\cite[\S\,3]{pressure}}]\label{ineq} 
	The dynamical intersection defined above is well defined, linear in the second variable and for all positive $s$ satisfies $\II(s\kappa,c)=\II(\kappa,c)/s.$ If also $c$ has non-negative periods and finite entropy then  $\II(\kappa,c)\geq \hJ\kappa/\hJ c.$  Moreover, if $\II(\kappa,c)= \hJ\kappa/\hJ c$ then for every $\g\in\G$ one has $\hJ\kappa\ell_\kappa(\g)=\hJ c\ell_c(\g)$.
\end{prop}

We will also need the following definitions. 

\begin{defi}
	\item\begin{itemize}
		\item[-]A \emph{Patterson-Sullivan measure for $\kappa$ of exponent} $\delta\in\R_+$ is a probability measure $\ps$ on $\bord\G$ such that for every $\g\in\G$ one has 
		\begin{equation}\label{defP}
			\frac{d \g_*\ps}{d\ps}(\cdot)=e^{-\delta\cdot \kappa\left(\g^{-1},\,\cdot\,\right)}.
		\end{equation} 
	\item[-]Let $\kappa^*$ be a cocycle dual to $\kappa$, then a \emph{Gromov product} for the ordered pair $(\kappa^*, \kappa)$ is a function $[\cdot,\cdot]:\bord^2\G\to\R$ such that for all $\g\in\G$ and $(x,y)\in\bord^2\G$ one has 
	$$[\g x,\g y]-[x,y]=-\big(\kappa^*(\g,x)+\kappa(\g,y)\big).$$
\end{itemize}
\end{defi}

\subsection{The critical hypersurface and intersection}\label{QQ}

Let now $c:\G\times\bord\G\to V$ be a H\"older cocycle. Its \emph{limit cone} is denoted by $$\calL_c=\overline{\bigcup_{\g\in\G}\R_+\cdot\ell_c(\g)}$$ 
and its \emph{dual cone} by $\cdc{c}=\{\psi\in V^*:\psi|_{\calL_c}\geq0\}.$ 
Observe that for every $\varphi\in\inte\cdc{c}$, $\varphi\circ c$ is a real-valued cocycle, so the concepts from Section \ref{Sreal} apply. We denote by 
\begin{alignat}{2}\label{qh1}
	\cal Q_c  & =\Big\{\varphi\in\inte\cdc{c}:\hJ{\varphi\circ c}=1\Big\},\\ 
	\cal D_c & = \Big\{\varphi\in\inte\cdc{c}:\hJ{\varphi\circ c}\in(0,1)\Big\},\nonumber \end{alignat}
respectively the \emph{critical hypersurface} and the \emph{convergence domain} of $c.$ 

For $\varphi\in\inte\cdc c$ we consider the linear map $\II_{\varphi}=\II_{\varphi}^c:V^*\to\R$ defined by $$\II_{\varphi}^c(\psi):= \II(\varphi\circ c,\psi\circ c),$$ as in Equation \eqref{defiII}. The natural identification between the set of hyperplanes in $V^*$ and $\P(V)$ is used in the next proposition.

\begin{cor}[{S. \cite[Cor. 3.4.3]{dichotomy}}]\label{strictly}
	Assume $\calL_c$ has non-empty interior and that there exists $\psi\in\cdc{c}$ such that $\hJ{\psi}<\infty.$ Then $\cal D_c$ is a strictly convex set with boundary $\cal Q_c.$  The latter is an analytic co-dimension one sub-manifold of $V.$ The map $\sf u^c:\cal Q_c\to\P(V)$ defined by $$\varphi\mapsto\sf u^c_\varphi:=\sf T_\varphi\cal Q_c=\ker\II_\varphi$$ is an analytic diffeomorphism between $\cal Q_c$ and $\inte\big(\P(\calL_c)\big).$\end{cor}

\subsection{Ergodicity of directional flows}\label{directional}

It follows from Proposition \ref{-infty} that if there exists $\psi\in\cdc{c}$ with $\hJ\psi<\infty$ then the $\G$-action $\bord^2\G\times V$ $$\g(x,y,v)=\big(\g x,\g y,v-c(\g,y)\big)$$ is properly discontinuous.

\begin{defi}A Hölder cocycle $c$ is \emph{non-arithmetic} if the periods of $c$ generate a dense subgroup in $V.$
\end{defi}

We fix $\varphi\in\cal Q_c$ and denote by $u_\varphi\in\sf u_\varphi$ the unique vector in $\cal L_c\cap\sf u_\varphi$ with $\varphi(u_\varphi)=1.$ We define then the \emph{directional flow} $\df^\varphi=\big(\df^\varphi_t:\G\/\big(\bord^2\G\times V\big)\to\G\/\big(\bord^2\G\times V\big)\big)_{t\in\R}$ by $$t\cdot(x,y,v)=(x,y,v-tu_\varphi).$$

\begin{assu}\label{Pat-SulEx}
	We assume there exists:
	\begin{itemize}
		\item[-] a dual cocycle $(\varphi\circ c)^*,$ 
		\item[-] a Gromov product $[\,,\,]_\varphi$ for such a pair,
		\item[-] Patterson-Sullivan measures, $\ps^\varphi$ and ${\ov\ps}^\varphi,$ respectively for each of the cocycles; (the exponent of both measures is then necessarily $\hJ\varphi=1$ S. \cite[Proposition 3.3.2]{dichotomy}).
	\end{itemize}
\end{assu}

Consider then the $\varphi$-\emph{Bowen-Margulis} measure $\BM^\varphi$ on $\G\/\big(\bord^2\G\times V\big)$ defined as the measure  induced on the quotient by the measure
\begin{equation}\label{forBM}
	e^{-[\cdot,\cdot]_\varphi}{\ov\ps}^\varphi\otimes\ps^\varphi\otimes \Fundeb_{V},
\end{equation} for a $V$-invariant Lebesgue measure on $V.$ 
We denote by $\cal K(\df^\varphi)$ the \emph{recurrence set} of the directional flow $\df^\varphi$:
$$\cal K(\df^\varphi):=\{\left.p\in\G\/\big(\bord^2\G\times V\big)\right|\; \exists B \text{ open bounded}, t_n\to \infty \text{ with } \df^\varphi_{t_n}(p)\in B \}.$$

\begin{cor}[{S. \cite[Cor. 3.6.1]{dichotomy}}]\label{d=2} 
	Assume that $c$ is non-arithmetic, and that there exists $\varphi\in\cal Q_c$ satisfying Assumptions  \ref{Pat-SulEx}. If  $\dim V\leq2$ then the directional flow $\df^\varphi$ is $\BM^\varphi$-ergodic, and $\cal K(\df^\varphi)$ has total mass. If $\dim V\geq4$ then $\cal K(\df^\varphi)$ has measure zero.
\end{cor}

\section{Subspace conicality for Anosov representations: Theorem \ref{Ahyper}}\label{Anosov}

\subsection{Gromov hyperbolic groups and cone types}\label{cont}
Let $\G=\langle S\rangle$ be a finitely generated non-elementary Gromov hyperbolic group, and recall from \S\ref{cocycles} that we denote by $\bord^2\G$ the set of distinct pairs in its Gromov boundary $\bord\G$.

\begin{defi}\label{conicalsequence}
	A divergent sequence $\{\g_n\}_{n\in\N}\subset\G$ converges to a point $x\in\bord\G$ \emph{conically} if for every $y\in\bord\G-\{x\}$ the sequence $(\g_n^{-1}y,\g_n^{-1}x)$ remains on a compact set of $\bord^2\G.$
\end{defi}
\begin{remark}\label{r.conicalgeodesic}
	It is easy to verify that a sequence $\{\g_n\}_{n\in\N}$ converges conically to $x\in\bord\G$ if and only if it lies in an uniform neighborhood of any geodesic ray $(\alpha_n)_0^\infty$ converging to $x$, namely there exists $K>0$ and a subsequence $\{\alpha_{n_k}\}$ such that for all $k$ one has $d_\G(\alpha_{n_k},\g_{k})<K.$
\end{remark}

Given $\g\in\G$ we denote by $\cone(\g)$ the \emph{cone type} of $\g\in\G$, namely
$$\cone(\g):=\{h\in\G|\; d(e,\g h)=d(e,\g)+d(e,h)\}.$$
Cannon showed \cite{CannonCones} the set of cone types of a Gromov hyperbolic group is finite, see for example Bridson-Haefliger's book \cite[P. 455]{BH}.
 We denote by $\cone_\infty(\g)\subset{\bord\G}$ the set of points $x$ that can be represented by a geodesic ray contained in $\cone(\g)$.

 We will also need a coarse version of these. Recall that a sequence $(\alpha_j)_0^\infty$ is a $(c,C)$-quasigeodesic if for every pair $j,l$ it holds
$$\frac1{c}|j-l|-C\leq d_\G(\alpha_j,\alpha_l)\leq c|j-l|+C.$$
The \emph{coarse cone type at infinity} of an element $\g$ is the set of endpoints at infinity of quasi-geodesic rays based at $\g^{-1}$ passing through the identity:   
$$
	\cone^{c}_\infty(\g) =   \Big\{[(\alpha_j)_0^\infty]\in\bord\G|\,  (\alpha_i)_0^\infty \text{ is a $(c,c)$-quasi-geodesic, } \alpha_0=\g^{-1}, e\in\{\alpha_j\}\Big\}.
$$
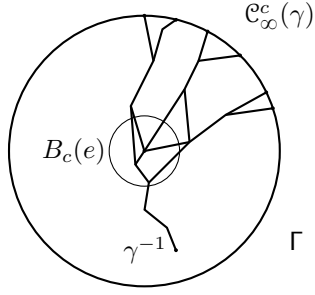
\begin{figure}[hbt]
	\centering
	\begin{tikzpicture}[scale = 0.6]
	\draw [thick] circle [radius = 3];\draw [fill] circle [radius = 0.03];	
	\draw  circle [radius = 0.78];
	\node [left] at (-0.65,0) {$B_{c}(e)$};
	\draw [thick] (0.7,-2.2) -- (0.5,-1.7) -- (0.004,-1.3) -- (0.1,-0.7) -- (-0.2,-0.3) -- (0,0); 
	\node [left] at (0.7,-2.2) {$\g^{-1}$}; \draw [fill] (0.7,-2.2) circle [radius = 0.03]; 
	\draw [thick] (1,0.2) -- (0.1,-0.7);
	\draw [thick] (-0.2,-0.3) -- (-0.3,1);
	\draw [thick] (0,0) -- (-0.3,1) -- (0.2,2) -- (0,3); 
	\draw [thick] (0.2,2) -- (0.4, 2.7) -- (0.7,2.91); 
	\draw [thick] (1,0.2) -- (0.89,1.37);
	\draw [thick] (1.8,0.8) -- (2.7,1.3);
	\draw [thick] (0,0) -- (1,0.2) -- (1.8,0.8) -- (2.846,0.948);
	
	\draw [fill] (0.7,2.91) circle [radius = 0.03];
	\draw [fill] (2.7,1.3) circle [radius = 0.03];
	\draw [fill] (2.12,2.12) circle [radius = 0.03];
	\draw [fill] (0,3) circle [radius = 0.03];
	\draw [fill] (2.846,0.948) circle [radius = 0.03];	
	\draw [thick] (0,0) -- (0.89,1.37) -- (1.2,2) -- (2.12,2.12);
	\draw [thick] (1.2,2) -- (1.4,2.65);
	\draw [fill] (1.4,2.65) circle [radius = 0.03];
	\node [right] at (3,-2) {$\G$};
	\node [right] at (2,3) {$\cone^{c}_\infty(\g)$};
	\end{tikzpicture}
	\caption{The coarse cone type at infinity, picture  from P.-S.-Wienhard \cite{PSW2}.}\label{fig:1}
\end{figure}

\subsection{Anosov representations}\label{s.Anosov}
Fix a subset $\t\subset\simple.$ Let $\G$ be a finitely generated group and denote by $|\,|$  the  word-length associated to a finite generating set $S$.  

\begin{defi}\label{AnosovDefi}
	Following\footnote{See also Bochi-Potrie-S. \cite{BPS} and Guéritaud-Guichard-Kassel-Wienhard \cite{GGKW}.} Kapovich-Leeb-Porti \cite{KLP}, a representation $\rho:\G\to\sf G$ is \emph{$\t$-Anosov} if there exist positive  constants $C$ and $\mu$ such that for all $\g\in\G$ and $\sroot\in\t$ one has 
	$$\sroot(\cartan(\rho\g))\geq \mu|\g|-C.$$ The constants $\mu$ and $C$ are usually referred to as \emph{the domination constants} of $\rho$.
If $\sf G=\PGL(d,\R)$ and $\t=\{\slroot_1\}$ we say that $\rho$ is \emph{projective Anosov}. In order to easy the notation we ill identify in what follows $\g$ with $\rho(\g)$. 
\end{defi}

Anosov representations were introduced by Labourie \cite{labourie} and further developed by Guichard-Wienhard \cite{GW}. They have played a central role in understanding the Hitchin component of split groups (see below) and are considered nowadays as the higher-rank generalization of convex co-compact groups. We refer the reader to the surveys by Kassel \cite{FannyICM18} and Wienhard \cite{Anna-ICM} for further information.

\begin{obs}\label{wallsAnosov} A Zariski-dense representation $\rho:\G\to\sf G$ is $\t$-Anosov if and only if $\rho$ is a quasi-isometric embedding and its limit cone $\calL_\rho$ does not meet any wall $\ker\sroot$ for $\sroot\in\t:$ this follows from the definition since by Benoist \cite{limite}, if $\rho(\G)$ is Zariski-dense then the limit cone $\calL_\rho$ equals the asymptotic cone. 
\end{obs}

A useful property of $\t$-Anosov representations is that their limit set $\Lambda_\G\subset \calF_{\theta}$, namely the minimal $\G$-invariant subset in $\calF_{\theta}$, is parametrized by the Gromov boundary of the group $\G,$ see Kapovich-Leeb-Porti \cite{KLP},  Guéritaud-Guichard-Kassel-Wienhard \cite{GGKW}. We will need the following precise statement.

\begin{prop}[{Bochi-Potrie-S. \cite[Proposition 4.9  ]{BPS}}]\label{conical}
If $\rho:\G\to\sf G$ is $\t$-Anosov, then for any geodesic ray $(\alpha_n)_0^\infty$ with endpoint $x$, the limits
$$\xi^\t_\rho(x):=\lim_{n\to\infty}U_\t(\alpha_n)\quad \xi^{\ii\t}_\rho(x):=\lim_{n\to\infty}U_{\ii\t}(\alpha_n)$$
exist and do not depend on the ray; they define continuous $\rho$-equivariant transverse maps $\xi^{\t}:\bord\G\to\cal F_\t$, $\xi^{\ii\t}:\bord\G\to\cal F_{\ii\t}$.  If $\g\in\G$ is hyperbolic, then $\g$ is $\t$-proximal with attracting point $\xi^\t(\g^+)=(\g)_\t^+.$
\end{prop}

	We conclude the section with a number of quantitative results that will be needed in the paper. For an Anosov representation $\rho$ there exists a constant $\delta_\rho$ quantifying transversality of Cartan-attractors along (quasi)-geodesic rays:
	\begin{prop}[{Bochi-Potrie-S. \cite[Lemma 2.5]{BPS}}]\label{fundamentalConstant}
		If $\rho:\G\to\sf G$ is $\t$-Anosov and $c>0$ is given, then there exist $\ov L\in\N$ and $\delta_{\rho,c}>0$, depending only $c$ and the domination constants of $\rho$, such that for every $(c,c)$-quasi-geodesic segment through the identity $\{\alpha_i\}_{-m}^k$ with $k,m\geq L$ one has, for all $\sroot\in\t, $
		that 
		$$\peso_\sroot\Gr_\t\big(U_{\ii\t}(\alpha_{-m}),U_\t(\rho\alpha_k)\big)\geq\log\delta_{\rho,c}.$$
		\end{prop}
Combining Proposition \ref{conical} and Proposition \ref{fundamentalConstant} we obtain:
\begin{cor}\label{c.fundamentalConstant}
	Up to decreasing $\delta_{\rho,c}$, for every $\g\in\G$ and every $x\in\cone_\infty^c(\g)$ one has  $$\peso_\sroot\Gr_\t\big(U_{\ii\t}(\g^{-1}),\xi^\t_\rho(x)\big)\geq\log\delta_{\rho,c}.$$ In particular, if we let $\alpha=-\log\delta_{\rho,c}$ then (recall Equation \eqref{e.CartBasin}) 
		\begin{equation}\label{e.conesinCartan}
			\xi^\theta_\rho(\cone_{\infty}^{c}(\g))\subset B_{\theta,\alpha}(\g).
		\end{equation}

\end{cor}	

\begin{defi}\label{FConstant} Let $\rho:\G\to\sf G$ be $\t$-Anosov and $c>0$, then the constant $\delta_{\rho,c}$ verifying both Proposition \ref{fundamentalConstant} and Corollary \ref{c.fundamentalConstant} will be called \emph{the fundamental constant} of $\rho$ and $c$. If we consider geodesics instead of quasi-geodesics (i.e. $(c,C)=(1,0)$) we let $\delta_\rho$ be the \emph{fundamental constant} associated to $\rho$.
\end{defi}

The following two results will be needed in Section \ref{s.6.1}.

\begin{prop}[cfr. P.-S.-Wienhard {\cite[\S 5.1]{PSW1}}]\label{p.coneinBall}\label{Lipschitz-compatibleA}
	Let $\rho:\G\to\SL(d,\K)$ be projective Anosov and consider  $c>0$. Then there exists a constant $K$, depending on  $c$ and on $\rho$ such that for every large enough $\g\in\G$ one has  
	$$\xi^1_\rho\big(\g \cone_\infty^{c}(\g)\big)\subset B\big(U_1(\g), Ke^{-\slroot_1(\cartan(\g))}\big).$$ 
	\end{prop}

\begin{proof} Using Corollary \ref{c.fundamentalConstant} for $\theta=\{\slroot_1\}$, the result follows as in P.-S.-Wienhard \cite[\S 5.1]{PSW1}.\end{proof}

\begin{prop}\label{Lipschitz-compatibleB}
	 	Let $\rho:\G\to\SL(d,\K)$ be projective Anosov. For every $\alpha>0$ there exist $C$ and $\mu>0$ such that for every $\ell_1,\ell_2\in\P(\K^d)$ with $$\Gr\big(\ell_i,U_{d-1}(\g^{-1})\big) \geq-\alpha, \quad i=1,2$$ it holds $d_\P(\rho(\g)\ell_1,\rho(\g)\ell_2)\leq Ce^{-\mu|\g|}d(\ell_1,\ell_2).$
\end{prop}

\begin{proof} For an Hermitian product on $\C^d,$ and every $\alpha>0$ there exists $C>0$ such that if $h\in\GL(d,\C)$ is such that $\slroot_1(\cartan(h))>0,$ then for all $\ell_1,\ell_2\in\P(\C^d)$ with {$\angle(\ell_i,U_{d-1}(h^{-1}))>\alpha$} one has $$d_\P(h\ell_1,h\ell_2)\leq Ce^{-\slroot_1(\cartan(h))}d_\P(\ell_1,\ell_2),$$ (a proof follows, for instance, by applying \cite[Lemma 2.8]{PSW1} to $g=h^{-1},$ $P=U_1(h)$ and $Q=hU_{d-1}(h)$). The result then follows by applying Definition \ref{AnosovDefi}.\end{proof}

The following technical result  will be  useful in the proof of Proposition \ref{conK}. Given an Anosov representation, we can use the Gromov product to determine the endpoint of a conical sequence (recall Definition \ref{conicalsequence}):

\begin{lemma}\label{UdU1}
	Let $\rho:\G\to\sf G$ be $\t$-Anosov. If $\{\g_n\}\subset\G$ is a conical sequence, $x\in\bord\G,$ and there exists $\sroot\in\t$ such that $\peso_\sroot\Gr_\t\big(U_{\ii\t}(\g_n),\xi^\t(x)\big)\to-\infty,$ then $\g_n\to x$.
\end{lemma}
\begin{proof}
	We denote by $y$ the endpoint of the conical sequence $\g_n$. Proposition \ref{conical} and Remark \ref{r.conicalgeodesic} imply that $U_{\ii\t}(\g_n)\to\xi^{\ii\t}(y)$. Since, however, $\peso_\sroot\Gr_\t\big(U_{\ii\t}(\g_n),\xi^\t(x)\big)\to-\infty$, we deduce that $\xi^{\ii\t}(y)$ is not transverse to $\xi^\t(x)$ (recall Remark \ref{fty}). Since $\xi^\t$ is transverse, we deduce   that $x=y$.
\end{proof}	
It will be useful in the proof of Proposition \ref{conK} to know that the endpoints of conical sequences belong to pushed Cartan basins:
\begin{lemma}\label{l.conicbasin} Let $\rho:\G\to G$ be $\theta$-Anosov, $x\in\bord\G$. If $\g_n\to x$ conically, then there exists $\alpha$ only depending on the sequence and the representation $\rho$ such that for every $n$, $\xi^\theta(x)\in \g_nB_{\t,\alpha}(\g_n)$.
\end{lemma}	
\begin{proof}
	We know from Remark \ref{r.conicalgeodesic} that $\g_n$ is contained in a neighbourhood of a geodesic ray to $x$, or equivalently there exist a constant $c$ such that $\gamma^{-1}x\in\cone^{c}_\infty(\gamma_n)$. The result is then a consequence of Equation \eqref{e.conesinCartan}. 
\end{proof}	

\subsection{Patterson-Sullivan theory of Anosov representations}\label{rfr}
If $\rho$ is a $\t$-Anosov representation, then we can pullback the Buseman-Iwasawa cocycle of $\sf G$ using the equivariant maps: the \emph{refraction cocycle} associated to a $\t$-Anosov representation  $\rho:\G\to\sf G$  is $\rfr:\G\times\bord\G\to\frak a_\t$ given by 
$$\rfr(\g,x)=\rfr_{\t,\rho}(\g,x)=\bus_\t\big(\rho(\g) ,\xi^\t_\rho(x)\big).$$ 

\noindent  Bridgeman-Canary-Labourie-S. \cite[Theorem 1.10]{pressure} show that the Mineyev geodesic flow of a group $\G$ admitting an Anosov representations is metric-Anosov, and thus \S\,\ref{cocycles} applies to $\rfr.$ Moreover, the following fact places $\rfr$ in the assumptions required in \S\,\ref{Sreal} and \S\,\ref{QQ}, see S. \cite{dichotomy} for details.

\begin{fact}
	The periods of the refraction cocycle equal the $\theta$-Jordan projection: $\rfr(\g,\g^+)=\lambda_\t(\g).$ For any $\sroot\in\t$ the real valued cocycle $\peso_\sroot\rfr$ has finite entropy.
\end{fact}

We import the following concepts of cocycles to the setting of Anosov representations: 

\begin{itemize}
	\item[-] The limit cone of $\rfr$ will be denoted by $\calL_{\t,\rho}$ and referred to as \emph{the $\t$-limit cone of $\rho$}; it is the smallest closed cone that contains the projected Jordan projections $\{\lambda_\t(\g):\g\in\G\}.$ 
	\item[-] The \emph{interior of the dual cone} $\inte\cdc{\t,\rho}\subset\frak a_\t^*$ consists of linear forms whose \emph{entropy} $$\hJ\varphi= \lim_{t\to\infty}\frac1t\log\#\big\{[\g]\in[\G]:\varphi(\lambda_\t(\g))\leq t\big\}$$ is finite. 
	\item[-] The $\t$-\emph{critical hypersurface}, resp. $\t$-\emph{convergence domain}, of $\rfr$  will be denoted by 
	\begin{alignat*}{2}\label{qh1}
		\cal Q_{\t,\rho}  & =\Big\{\varphi\in\inte\cdc{\t,\rho}:\hJ\varphi=1\Big\},\\ 
		\cal D_{\t,\rho} & = \Big\{\varphi\in\inte\cdc{\t,\rho}:\hJ\varphi\in(0,1)\Big\}.	\end{alignat*} 	\item[-] If $\calL_{\t,\rho}$ has non-empty interior, then we have a \emph{duality} diffeomorphism between $\cal Q_{\t,\rho}$ and $\inte\P(\calL_{\t,\rho})$ given by $$\varphi\mapsto\sf u_\varphi=\sf T_\varphi\cal Q_\rho.$$
\end{itemize}

More information on these objets can be found on S. \cite[\S\,5.9]{dichotomy}.

\begin{remark}\label{r.entropy=crit}It it proven in Glorieux-Monclair-Tholozan \cite[Theorem 2.31 (2)]{GMT} (see also S. \cite[Corollary 5.5.3]{dichotomy}) that if $\rho$ is $\t$-Anosov then for every $\varphi\in\inte\cdc{\t,\rho}$ the entropy $\hJ\varphi$ equals the critical exponent
	$$\hC\varphi:= \lim_{t\to\infty}\frac1t\log\#\big\{\g\in\G:\varphi(\cartan(\g))\leq t\big\}.$$
	In particular the $\t$-convergence domain is also given by 
	$$	\cal D_{\t,\rho} =\Big\{\varphi\in (\a_\t)^*:\sum_{\g\in\G}e^{-\varphi(\cartan(\g))}<\infty\Big\},$$
	see S. \cite[\S\,5.7.2]{dichotomy}.
\end{remark}

We observe that for  $\varphi\in\inte\cdc{\t,\rho}$ Assumptions \ref{Pat-SulEx} are  guaranteed for $\rfr_\varphi:=\varphi\circ\rfr.$  Indeed the cocycle $$\ov\rfr{}(\g,x)=\ii\bus_{\ii\t}\big(\g ,\xi^{\ii\t}(x)\big)$$ is dual to $\rfr,$  from Equation (\ref{forGr}) the function $[\cdot,\cdot]_\varphi:\bord^2\G\to\R$ 
$$	[x,y]_\varphi=\varphi\Big(\Gr_\t\big(\xi^{\ii\t}(x),\xi^\t(y)\big)\Big)$$
is a  Gromov product for the pair $(\ov\rfr_{\varphi},\rfr_{\varphi}),$ and we have the following result guaranteeing existence of Patterson-Sullivan measures $\ps^\varphi$ and $\ov\ps^\varphi$, as well as their values on Cartan basins defined in Equation \eqref{e.CartBasin}.

\begin{cor}[{S. \cite[Cor. 5.5.3+Lemma 5.7.1]{dichotomy}}]\label{existe} For every $\varphi\in\inte\cdc{\t,\rho}$ there exists a $\rfr_\varphi$-Patterson-Sullivan measure $\ps^\varphi$ of exponent $\hJ\varphi,$ moreover for every $\alpha$ there exists a constant $C$ such that for every $\g\in\G$ one has 
$$	\ps^\varphi\big((\xi^\theta)^{-1}(\g B_{\theta,\alpha}(\g))\leq Ce^{-\hJ\varphi\varphi\big(\cartan(\g)\big)}.
$$
\end{cor}

\subsection{Subspace-conicality}\label{sandwich} 

In this section we are interested in a notion of conicality along higher dimensional subspaces of the ambient Levi space. 

\begin{defi}\label{W-conical}Let $\rho:\G\to\sf G$ be $\t$-Anosov and consider a subspace $\sf W\subset\a_\t$. A point $x\in\bord\G$ is \emph{$\sf W$-conical} if there exists a conical sequence $\{\g_n\}_0^\infty\subset\G$ converging to $x,$ a constant $K$ and $\{w_n\}_0^\infty\subset\sf  W$ such that for all $n$ one has $$\big\|\cartan_\t\big(\g_n)- w_n\big\|\leq K.$$ The set of such points will be denoted by $\bord_{\sf W,\rho}\G=\bord_{\sf W}\G.$ \end{defi}

Assume from now on that $\sf W$ intersects the relative interior of $\calL_{\t,\rho},$
and consider $\varphi\in\inte\cdc{\t,\rho}$ with $\sf u_\varphi\subset\sf W.$ The intersection $\sf W_\varphi=\sf W\cap\ker\varphi$ has co-dimension $1$ in $\sf W$ and has trivial intersection with the limit cone $\calL_{\t,\rho}.$ Consider the quotient space 
$$ V=\a_\t/\sf W_\varphi$$
 equipped with the quotient projection $\Pi:\frak a_\t\to V.$  We say that $\rho$ is $(\sf W,\varphi)$-\emph{non-arithmetic} if the group spanned by $\big\{\Pi(\lambda_\t(\g)):\g\in\G\big\}$ is dense in $V.$ In this section we prove the following.

\begin{thm}\label{hyper}\label{ps-sandwich} Let $\rho:\G\to\sf G$ be $\t$-Anosov, $\sf W$ be a subspace of $\a_\t$ intersecting non-trivially the relative interior of $\calL_{\t,\rho},$ and  $\varphi\in(\a_\t)^*$ with $\sf u_\varphi\subset \sf W.$ Assume $\rho$ is $(\sf W,\varphi)$-non-arithmetic, then:
	\begin{itemize}
		\item if $\sf W$ has codimension $1$ then $\mu^\varphi(\bord_\sf W\G)=1;$ 
		\item if $\codim\sf W\geq3$ then $\mu^\varphi(\bord_{\sf W}\G)=0.$
		\end{itemize}
\end{thm}

\begin{obs} If $\rho$ is Zariski-dense then Theorem \ref{densidad}  (Benoist \cite{benoist2}) guarantees $(\sf W,\varphi)$-non-arithmeticity for every $\varphi\in(\a_\t)^*$ with $\sf u_\varphi\in\P(\sf W),$ thus Theorem \ref{hyper} readily implies Theorem \ref{Ahyper}. 
	\end{obs}

The remainder of the section is devoted to the proof of Theorem \ref{hyper}. Let $$V^*=\Ann(\sf W_\varphi)=\{\psi\in(\a_\t)^*:\psi|\sf W_\varphi\equiv0\},$$ with a slight abuse of notation we will identify the dual of $V$ with $V^*\subset(\a_\t)^*\subset \a^*$ (recall from Section \ref{s.Levi} that we are identifying $(\a_\t)^*$ with the  subspace of $\a^*$ consisting of $\pi_\t$-invariant linear forms).

The composition of the refraction cocycle of $\rho$ with $\Pi$ is a $V$-valued H\"older cocycle $\cc:\G\times\bord\G\to V,$ $$\cc(\g,x)=\Pi\big(\rfr(\g,x)\big).$$ Its periods are $\cc(\g,\g_+)=\Pi\big(\lambda_\t(\g)\big),$ and thus its limit cone is $\calL_\cc=\Pi(\calL_{\t,\rho}).$ By $(\sf W,\varphi)$-non-arithmeticity, $\cal L_{\cc}\subset V$ has non-empty interior.

The heart of the proof of Theorem \ref{hyper} consits on relating $(\sf W,\varphi)$-conical points with elements of $\widetilde{\cal K}\big(\df^{\varphi}\big),$ where $\df^{\varphi}$ is the directional flow on $\G\/\bord^2\G\times V$ associated to the cocycle $\cc$ as in \S\,\ref{directional}. The first step is thus to observe that we can apply Corollary \ref{d=2} to $\cc,$ a task we enter at this point.

Since $\varphi\in\cal Q_{\t,\rho},$ it has in particular finite entropy. Moreover, by definition of $V^*$ one has $\varphi\in V^*.$ Consequently, the cocycle $\cc$ verifies assumptions in Corollary \ref{strictly}. One can moreover transfer existence properties from $\rfr$ to $\cc,$ indeed one has the following.

\begin{prop} The cocycle $\ov\cc=\Pi\circ\ov\rfr$ is a dual cocycle for $\cc.$ For each $\psi\in\cal Q_\cc$ there exist Paterson-Sullivan measures for $\cc$ and $\ov \cc$ and the projection $\psi\big(\Pi\big([\cdot,\cdot]\big)\big)$ is a Gromov product for the pair $\psi\circ\cc,\psi\circ\ov\cc.$
\end{prop}

\begin{proof} Since $\psi\in\cal Q_\cc=\cal Q_{\t,\rho}\cap V^*$ we can apply Corollary \ref{existe} to $\psi$ to obtain the desired Patterson-Sullivan measure, the remaining statements follow trivially as the equalities are linear.
\end{proof}

Since we are assuming $(\sf W,\varphi)$-non-arithmeticity, the cocycle $\cc$ is non-arithmetic and thus  Corollary \ref{d=2} gives the following dynamical information, observe that $\dim V=\codim \sf W+1.$

\begin{cor}\label{dicoCC} If $\codim \sf W\leq1$ then the directional flow $\df^\varphi$ is $\BM^\varphi$-ergodic, in particular $\cal K(\df^\varphi)$ has total mass. If $\codim \sf W\geq3$ then $\cal K(\df^\varphi)$ has measure zero.
\end{cor}

Observe  that modulo the understood identifications $\cal Q_\cc=\cal Q_{\t,\rho}\cap V^*,$ hence 
$$\sf T_\varphi\cal Q_\cc=(  \sf T_\varphi\cal Q_{\t,\rho})\cap V^*$$ and thus the map  $\sf u^\cc:\cal Q_\cc\to\inte\P(\cal L_\cc)$  from Corollary \ref{strictly} verifies $\sf u^\cc_\varphi=\Pi(\sf u_\varphi).$ So measuring $\sf W$-conicality with respect to $\mu^\varphi$ translates to directional conicality along the direction $\sf u^\cc_\varphi,$ which we now recall. We fix an arbitrary norm $\|\,\|$ on $V$ and define, for $\ell\in\P(V)$ and $r>0,$ the $r$-\emph{tube about $\ell$} by 
$$\T_r(\ell):=\{v\in V|\, \exists w\in\ell, \|v-w\|<r\}.$$ 

\begin{defi}\label{defCon}  A point $y\in\bord\G$ is \emph{$\sf u_\varphi^\cc$-conical} if there exists $r>0$ and a conical sequence $\{\g_n\}_0^\infty\subset\G$ with $\g_n\to y$ such that for all $n$ one has $\Pi\big(\cartan_\theta(\rho(\g_n))\big)\in\T_r(\sf u_\varphi^\cc).$
\end{defi}

The next statement follows from the definitions.

\begin{lemma} A point $y\in\bord\G$ is $\sf W$-conical if and only if it is $\sf u^\cc_\varphi$-conical.
\end{lemma}

If we are allowed to worsen the constants, we can replace, in Definition \ref{defCon}, the conical sequence $(\g_n)$ with an infinite subset of a geodesic ray:
\begin{lemma}\label{conGeo} 
	A point $y\in\bord\G$ is $\sf u_\varphi^\cc$-conical if and only if there exists $r>0,$ a geodesic ray $(\alpha_i)_0^\infty$ converging to $y$ and an infinite set of indices $\I\subset\N$ such that for all $k\in\I$ one has$$ \Pi\big(\cartan_\t(\alpha_k)\big)\in\T_r(\sf u_\varphi^\cc).$$ \end{lemma}

\begin{proof}
	Assume $y$ is $\sf u_\varphi^\cc$-conical, then since $\{\g_n\}_0^\infty$ is conical, for any geodesic ray $(\alpha_n)_0^\infty$ converging to $y$ there exists $K>0$ and a subsequence $\{\alpha_{n_k}\}$ such that for all $k$ one has $d_\G(\alpha_{n_k},\g_{k})<K$ (Remark \ref{r.conicalgeodesic}). Proposition \ref{p.Ben} implies then that for all $k$ one has $$\|\cartan(\alpha_{n_k})-\cartan(\g_k)\|$$ is bounded independently of $k.$ This implies the result.\end{proof}

We now relate $\sf u^\cc_\varphi$-conicality with the recurrence set $\cal K(\df^\varphi).$ By definition of $\cal K(\df^\varphi),$ a point $(x,y,v)\in\bord^2\G\times V$  projects to $\cal K(\df^\varphi)$ if and only if there exist divergent sequences $(\g_n)\subset\G$ and $t_n\to+\infty$ in $\R$ such that 
\begin{equation}\label{sequence}
	\df^{\varphi}_{t_n}\g_n^{-1}(x,y,v)=\big(\g_n^{-1}x,\g_n^{-1}y,v-\cc(\g_n^{-1},y)-t_nu_{\varphi}\big)
\end{equation} 
is contained in a subset of the form $\{(z,w)\in\bord^2\G:d(z,w)\geq\kappa\}\times B(v,K)$ for some distance $d$ on $\bord\G$. One has the following

\begin{prop}\label{conK}
	A point $y\in\bord\G$ is $\sf u_\varphi^\cc$-conical if and only if for every $x\in\bord\G-\{y\}$ and $v\in V$ the point $(x,y,v)$ projects to $\cal K(\df^\varphi).$ 
\end{prop}

\begin{proof}
The implication ($\Rightarrow$) follows exactly as in the proof of S. \cite[Proposition 5.13.4]{dichotomy}. The other implication also follows similarly but with a minor difference we now explain. 

Assume that $(x,y,v)$ projects to $\cal K(\df^\varphi)$ and consider sequences $\{\g_n\}$ and $t_n$ as in Equation (\ref{sequence}). 
Since $\big(\g_n^{-1}x,\g_n^{-1}y\big)$ remains in a compact subset of $\bord^2\G$, the sequence $\{\g_n\}$ is conical, we will show now that $\g_n\to y.$ Indeed, since $t_n\to+\infty$  necessarily $\cc(\g_n^{-1},y)\to-\infty.$ 

Consider now any root $\sroot\in\t,$ with associated fundamental weight $\peso_{\sroot}\in\cdc{\t,\rho}$, and Tits representation $\Fund_{\sroot}:\sf G\to V$. Since $\rho$ is $\theta$-Anosov,  the H\"older cocycle $\rfr_{\peso_{\sroot}}$ has positive periods and finite entropy. Since $\cc(\g_n^{-1},y)\to-\infty$ Proposition \ref{-infty}  implies that 
 $$\rfr_{\peso_{\sroot}}(\g_n^{-1},y\big)\to-\infty.$$ 
 By definition of the cocycle $\rfr_{\peso_{\sroot}}$ and Equation (\ref{busnorma}) we have
 \begin{equation}\label{normto0} 
 	\frac{\|\Fund_{\sroot}(\g_n^{-1})v\|}{\|v\|}\to0 
 \end{equation}
 for a non-zero $v\in\zeta_\sroot(\xi(y)),$ (recall that the map $	\zeta_\sroot:\cal F_{\sroot}(\sf G )\to \P(V)$ was defined in Equation \eqref{maps}). Setting  $\dim V=d$, a standard linear algebra computation (for example in Bochi-Potrie-S. \cite[Lemma A.3]{BPS}) gives \begin{alignat*}{2}
 \frac{\|\Fund_{\sroot}(\g_n^{-1})v\|}{\|v\|} & \geq\big\|\Fund_{\sroot}(\g_n^{-1})\big\|\sin\angle\big(\zeta_\sroot\xi(y),U_{d-1}(\Fund_{\sroot}\g_n)\big)\\ & \geq e^{l_\sroot\peso_\sroot\Gr_\t\big(U_\t(\g_n),y\big)}\end{alignat*} and thus, by Equation (\ref{normto0}) and Lemma \ref{UdU1} one has $\g_n\to y,$ as desired. 

The point $\xi(y)$ lies then in the pushed Cartan basin $\g_nB_{\t,\alpha}(\g_n)$ for an $\alpha$ independent of $n$ (Lemma \ref{l.conicbasin}),
and thus Equation (\ref{comparision-shadow}) gives a constant $K$ such that for all $n$ one has 
$$K\geq\big\|\cartan_\t(\g_n)-\rfr\big(\g_n,\g_n^{-1}y\big)\big\|=\big\|\cartan_\t(\g_n)+\rfr(\g_n^{-1},y)\big\|$$ implying, by Equation \eqref{sequence}, that $y$ is $\sf u_\varphi^\cc$-conical, as desired.\end{proof}

The proof of Theorem \ref{hyper} follows now along the same lines as in S. \cite[Theorem 5.13.3]{dichotomy}. We include the arguments here for completeness.

For $y\in\bord_{\sf W,\rho}\G, x\in\bord\G-\{y\}$ we fix neighbourhoods $A^-$ and $A^+$ of $x$ and $y$ respectively and $T>0$ small enough so that the quotient projection $\sf p:\bord^2\G\times V\to\G\backslash\bord^2\G\times V$ is injective on $\tilde{\mathrm{B}}=A^-\times A^+\times B(0,T).$ We can thus use Equation \eqref{forBM}  to compute the measure of $\mathrm{B}=\sf p(\tilde{\mathrm{B}})$. 

For $\tilde{\cal K}({\df}^\varphi)=\sf p^{-1}\big(\cal K({\df}^\varphi)\big),$  Proposition \ref{conK} asserts  $$A^-\times(A^+\cap\bord_{\sf W,\rho}\G)\times B(0,T)=\tilde{\cal K}({\df}^\varphi)\cap\tilde{\mathrm{B}}.$$

\noindent 
If $\codim \sf W=1$ by Corollary \ref{dicoCC}  $\BM^\varphi(\tilde{\mathrm{B}})=\BM^\varphi\big(\tilde{\cal K}({\df}^\varphi)\cap\tilde{\mathrm{B}}\big),$ which implies that  $\mu^\varphi(A^+\setminus\bord_{\sf W,\rho}\G)=0$ and thus  $\mu^\varphi(\bord_{\sf W,\rho}\G)=1.$ On the other hand, if $\codim\sf W\geq3$, then we have $\BM^\varphi\big(\tilde{\cal K}({\df}^\varphi)\big)=0$ so $\mu^\varphi(A^+\cap\bord_{\sf W,\rho}\G)=0$ and the theorem is proved.

\section{Locally conformal representations: Hausdorff dimension of $\hol$-conical points}\label{s.Hffbcon}
In this section we let $\K=\R, \C$ or $\H$, the non-commutative field of Hamilton's quaternions. A Cartan subspace $\a$ of $\SL(d,\K)$ is the subspace of $\R^d$ consisting of vectors whose coordinates sum 0.  For $g\in\SL(d,\K)$ we denote by $$\cartan(g)=\big(\cartan_1(g),\cdots,\cartan_d(g)\big)\in\a^+$$ the coordinates of the Cartan projection. We recall Definition \ref{d.hyp}.

\begin{defi}\label{locConf} Let $p\in\lb2,d-1\rb.$ A $\{\slroot_1,\slroot_{d-p}\}$-Anosov representation $\rho:\G\to\SL(d,\K)$ is \emph{$(1,1,p)$-hyperconvex} if, for every pairwise distinct triple $(x,y,z)\in\bord\G^{(3)}$, one has $$\big(\xi^1(x)+\xi^1(y)\big)\cap\xi^{d-p}(z)=\{0\}.$$ If in addition one has $\cartan_2(\rho(\g))=\cartan_p(\rho(\g))$ $\forall\g$, we say that $\rho$ is \emph{locally conformal}. As before, we identify from now on $\g$ and $\rho(\g)$.
\end{defi}

The terminology is justified by Proposition \ref{conetypesBalls} below stating that for such representations pushed coarse cone types are coarsely balls, a small refinement of an analogous result from P.-S.-Wienhard \cite{PSW1}.

In this section we will study conicality from \S\,\ref{sandwich} on a specific situation that we now explain. Later, in \S\,\ref{teoLCdiff}, we will relate this section to the notion of $\hol$-concavity and in \S\,\ref{s.Hffproof} to differentiability properties of the map $\ov\xi\circ\xi^{-1}$.

Consider $\ov\K\in\{\R,\C,\H\}$ and two locally conformal representations  $\rho:\G\to\SL(d,\K)$ and $\ov\rho:\G\to\SL(\ov d,\ov\K)$, with projective equivariant maps \begin{alignat*}{2}\xi:&\bord\G\to\P(\K^d)\\ \ov\xi:&\bord\G\to\P(\ov\K{}^{\ov d} ).\end{alignat*} The product representation $(\rho,\ovrho):\G\to\SL(d,\K)\times\SL(\ov d,\ov\K)$ is $\t$-Anosov for $\t=\{\slroot_1,\slroot_p,\ov{\slroot}_1,\ov{\slroot}_p\}$ with $\{\slroot_1,\ov{\slroot_1}\}$-limit map the "graph map" $$\varXi=\big(\xi,\ov{\xi}\big):\bord\G\to\P(\K^{d})\times\P(\ov\K^{\ov d}).$$ We consider a Cartan subspace of the product group $\SL(d,\K)\times\SL(\ov d,\ov\K)$ and let $\frak a_\t$ be the associated Levi space. Its dual $(\a_\t)^*$ is spanned by the fundamental weights of roots in $\t.$ We let 
\begin{alignat*}{2}
	\slroot:=\frac{ p\peso_{\slroot_1}-\peso_{\slroot_p}}{p-1},\\ 
	\ov\slroot:=\frac {p\peso_{\ov\slroot_1}-\peso_{\ov\slroot_p}}{p-1}.
\end{alignat*} 
Both $\slroot,\ov\slroot\in(\a_\t)^*$ and under the assumption $\cartan_2(\g)=\cartan_p(\g)$ for all $\g$ of Definition \ref{locConf}, it holds on $\calL_{\rho} $ that $\slroot_1=\slroot$ and $\ov\slroot=\ov{\slroot}_1$ (if $p=2$ the equality holds on $\frak a$).

\begin{defi}\label{flat} Fix $\hol\in(0,1]$. A point $x\in\bord\G$ is \emph{$\hol$-conical} if it is conical as in Definition \ref{W-conical} for the product representation $(\rho,\ov\rho)$ with respect to the hyperplane $$\{v\in\a_\t:\hol\slroot(v)=\ov\slroot(v)\}=\ker(\hol\slroot-\ov\slroot).$$ Equivalently, there exist $R,$ a geodesic ray $(\alpha_n)_0^\infty\subset\G$ with $\alpha_n\to x,$ and a subsequence $\{n_k\}$ such that for all $k$ one has $$\big|\hol\slroot\big(\cartan(\alpha_{n_k})\big)-\ov{\slroot}\big(\cartan(\ov\alpha_{n_k})\big)\big|\leq R.$$\end{defi}

Consider also the critical exponent
$$\hC{\infty,\hol}=\lim_{t\to\infty}\frac 1t\log\#\big\{\g\in\G:\max\big\{\hol\slroot(\cartan(\g)),\ov\slroot(\cartan(\ov\g))\big\}\leq t\big\},$$ and recall from Equation \eqref{defiII} the dynamical intersection defined by 
\begin{equation}\label{IIreps}\II_{\slroot}(\ov\slroot)=\lim_{t\to\infty}\frac1{\#\sf R_t(\slroot)}\sum_{\g\in \sf R_t(\slroot)}\frac{\ov\slroot(\lambda(\ov\g))}{\slroot(\lambda(\g))},
\end{equation} where $\mathsf R_t(\slroot)=\big\{[\g]\in[\G]:\slroot\big(\lambda(\g)\big)\leq t\big\}.$

In this section we compute the Hausdorff dimension of the image under the graph map $\varXi$ of the set of $\hol$-conical points with respect to a Riemannian metric:

\begin{thm}\label{Hffconical} Let $\rho,\ov\rho$ be  locally conformal representations over $\K$ and $\ov\K$ respectively. Assume the group generated by $\{(\slroot(\lambda(\g)),\ov\slroot(\lambda(\ov\g))):\g\in\G\}$ is dense in $\R^2$. Then, for every $\hol\in(0,1]$ with 
		$$\II_{\slroot}(\ov\slroot)>\hol>1/\II_{\ov\slroot}(\slroot),$$ 
		one has 
	\begin{alignat*}{2}\hol\hC{\infty,\hol}\leq\Hff\varXi\big(\{\hol\mathrm{-conical\ points}\}\big)&\leq\min\{\hC{\infty,\hol},\hol\hC{\infty,\hol}+(1-\hol) \}\\&<\min\{\hJ{\ov\slroot},\hJ{\slroot}/\hol\}\\&\leq\Hff(\varXi(\bord\G))\\ & =\max\{\hJ{\slroot},\hJ{\ov\slroot}\}.
	\end{alignat*}
\end{thm}

The proof of the above result is completed in \S\,\ref{proofHffconical}.

Recall that if $\hC{\slroot_1}=\hC{\ov\slroot_1}$ and the representations are not gap-isospectral, then Proposition \ref{ineq} gives $\II_{\ov\slroot_1}(\slroot_1)>1$. Theorem \ref{Hffconical} studies then $\hol$-conical points for any $\hol$ with $\II_{\ov\slroot_1}(\slroot_1)>1/\hol\geq1$. As the following result shows, the equality between entropies is rather natural for $\K=\R$. 

\begin{thm}[{P.-S.-Wienhard \cite{PSW1}}]\label{hyperh=1} Let $\rho:\G\to\SL(d,\K)$ be locally conformal, then $$\hJ{\slroot}=\Hff\big(\xi(\bord\G)\big).$$ Moreover, when $\K=\R$ and $\bord\G$ is homeomorphic to a $p-1$-dimensional sphere, $\scr h_{\slroot}=p-1.$
\end{thm}

When $\G$ is a surface group we can also weaken the assumption  on the density of periods:

\begin{cor}\label{flatsurface}Assume $\bord\G$ is homeomorphic to a circle and let $\rho$ and $\ov\rho$ be non-gap-isospectral real $(1,1,2)$-hyperconvex representations of $\G$. Then $$\Hff\varXi\big(\{1\mathrm{-conical\ points}\}\big)=\hC{\infty}<1.$$
\end{cor}

\begin{proof}Proposition \ref{nonA} below states that under our assumptions the group generated by $\{(\slroot(\lambda(\g)),\ov\slroot(\lambda(\ov\g))):\g\in\G\}$ is dense in $\R^2$. Theorem \ref{hyperh=1} guarantees that  $\II_\slroot(\ov\slroot)\geq1$. The equality then follows from Theorem \ref{Hffconical}.  \end{proof}

Kim-Minsky-Oh \cite{KMO-HffDir} have established realted Hausdorff dimension computations when $\rho$ and $\ov\rho$ are convex-co-compact representations in $\SO(n,1)$ without any assumption on $\II$.

\subsection{Cone types are coarsely balls}
In \cite{PSW1} P.-S.-Wienhard gave a concrete description of the images under the boundary map of the cone types at infinity. We discuss here a slight extension of that result adapted to our needs. We denote by $d_\P$ the distance on $\P(\K^d)$ induced by the choice of an inner (Hermitian) product on $\K^d$ and by $B(\ell,r)\subset\P(\K^d)$ the associated ball of radius $r$ about $\ell.$

\begin{prop}\label{conetypesBalls}
	Let $\rho:\G\to\SL(d,\K)$ be locally conformal. Then  there exist positive constants $ c, k_1,k_2$ and $L\in\N$ such that for every $x\in\bord\G$, every geodesic ray $(\alpha_n)_0^\infty$ with endpoint $x$ and every $n>L$ one has 
	$$B\Big(\xi(x),k_1e^{-\slroot_1(\cartan(\alpha_n))}\Big)\cap\xi(\bord\G)\subset\xi\Big(\alpha_n\cone_\infty^{c}(\alpha_n)\Big) \subset B\Big(\xi(x),k_2e^{-\slroot_1(\cartan(\alpha_n))}\Big).$$
\end{prop}

\begin{proof}The desired inclusions are proven in \cite{PSW1} for thickened cone types at infinity. We briefly explain here how to deduce from it the result we need.

Following \cite{PSW1} we denote by $X_\infty(\g)$,  for $\g\in\G$, the \emph{thickened cone type at infinity}, namely the tubular neighborhood in $\P(\K^d)$ of $\xi\big(\cone_\infty(\g)\big)$ of radius $\delta_\rho/2,$ where $\delta_\rho$ is the fundamental constant from Definition \ref{FConstant}. In \cite[Corollary 5.10]{PSW1} it is established that there exists $c_1>0$ and $L_0>0$ only depending on the domination constants of $\rho$ such that for all $i\geq L_0$ one has $$ B\Big(\xi(x),c_1e^{-\slroot_1(\cartan(\alpha_i))}\Big)\cap\xi(\bord\G)\subset \alpha_iX_\infty(\alpha_i).$$

By definition the thickened cone type $X_\infty(\g)$ is contained in the Cartan basin $B_{\{\slroot_1\},\alpha}(\g)$ for  $\alpha=-2\log\delta_\rho$. So P.-S.-Wienhard \cite[Proposition 3.3]{PSW2} provides the existence of   $c$ and $L_0$ such that for $\g\in\G$ with $|\g|>L_0$, one has $$X_\infty(\g)\cap\xi(\bord\G)\subset\xi\big(\cone_\infty^{c}(\g)\big).$$ Combining both equations one has, for all $i\geq L_0$ that 

\begin{equation}\label{bola}
	B\Big(\xi(x),c_1e^{-\slroot_1(\cartan(\alpha_i))}\Big)\cap\xi(\bord\G)\subset\xi\Big(\alpha_i\cone_\infty^{c}(\alpha_i)\Big) \subset B\Big(\xi(x), Ke^{-\slroot_1(\cartan(\alpha_i))}\Big),
\end{equation} 
the second inclusion following from Proposition \ref{p.coneinBall}. This concludes the proof. \end{proof}

\subsection{Hausdorff dimension and related concepts}
Recall that, given a metric space $(X,d)$ and a real number $s>0$, the \emph{$s$-capacity} of $X$ is
$$\cal H^s(X,d)=\lim_{\epsilon\to 0}\inf\left\{\sum_{U\in\cal U} \diam U^s\left|\, \cal U \text{  open covering of $\Lambda$,  } \sup_{U\in\cal U}\diam U<\epsilon\right.\right\}$$
and that 
\begin{equation}\label{HffDef}\Hff(X)=\inf \{s|\, \cal H^s(X)=0\}=\sup \{s|\, \cal H^s(X)=\infty\}.\end{equation}
The following can be verified directly from the definition:
\begin{lemma}\label{l.Hffunion}
	If $X=\bigcup_{n\in\N} X_n$ then
	$$\Hff(X)=\sup \Hff(X_n).$$
\end{lemma}	

We will use the following consequence of Theorem 1.5.14 from Edgar's book \cite{Edgar}:

\begin{cor}\label{edgar}Let $E\subset\R^d$ be a measurable subset equipped with a probability measure $\nu.$ If the upper density $$\ov{D}^\alpha(x)=\limsup_{r\to0}\frac{\nu\big(B(x,r)\cap E\big)}{r^\alpha}$$ is  $\nu$-essentially bounded above, then $\Hff(E)\geq\alpha.$
\end{cor}

\subsection{The lower bound $\Hff(\varXi\{\hol\mathrm{-conical\ points}\})\geq\scr{b}\hC{\infty,\hol}$}\label{s.lower}

We import some tools from the proof of Theorem \ref{hyper}. Consider the vector space $$V^*:=\spa\{\slroot,\ov{\slroot}\}$$ together with its radical $\ann(V^*)=\ker\slroot\cap\ker\ov\slroot$ and the quotient vector space $V=\a_\t/\ann(V^*).$ Any element of $V^*$ vanishes on $\ann(V^*)$ and thus $V^*$ is naturally identified with the dual space of $V.$ Using the preferred basis $\{\slroot,\ov\slroot\}$ of $V^*$ we  identify $V$ and $\R^2$ via the isomorphism $v\mapsto \big(\slroot(v),\ov\slroot(v)\big)$ and we let $$\Pi:\a_\t\to\R^2$$ be the quotient projection (composed with the above isomorphism). The image of the hyperplane $\ker\hol\slroot-\ov\slroot$ under the composition of $\Pi$ and the  identification of $V$ with $\R^2$ is the line passing through $(1,\hol)$,  $$\Pi\big(\ker(\hol\slroot-\ov\slroot)\big)=\big\{v\in V:\hol\slroot(v)=\ov\slroot(v)\big\}.$$ We  consider the quadrant $$V^+=\{\slroot\geq0\}\cap\{\ov{\slroot}\geq0\}.$$ Let  $\cc=\cc_{(\rho,\ovrho)}:\G\times\bord\G\to V$ be the composition of the refraction cocycle $\rfr_{(\rho,\ov\rho)}$ of the pair  with $\Pi.$ Its periods are $$\cc(\g,\g_+)=\Big(\slroot\big(\lambda(\g)\big),\ov\slroot\big(\lambda(\ov\g)\big)\Big),$$ 
so by assumption $\cc$ is non-arithmetic. As in \S\,\ref{sandwich} one has $\cal Q_\cc=V^*\cap\cal Q_{\t,\rho}$; by non-arithmeticity, the cone $\calL_\cc$ has non-empty interior and thus Corollary \ref{strictly} gives that $\cal Q_\cc$ is a strictly convex curve. We consider the max norm $\|v\|_{\infty,\hol}=\max\{\hol|\slroot(v)|,|\ov \slroot(v)|\}$ on $V$, and its dual (operator) norm on $V^*$ denoted by $\|\,\|^{1,\hol}$. Let $\vi\in\cal Q_\cc$ be the unique form such that 
$$\|\vi\|^{1,\hol}=\inf\{\|\varphi\|^{1,\hol}:\varphi\in\cal Q_\cc\}.$$

In the following lemma the role of the assumptions on dynamical intersection in Theorem \ref{Hffconical} becomes clear:
 \begin{lemma}\label{paraphi} 
	The functional $\vi/\|\vi\|^{1,\hol}$ is a convex combination $s\hol\slroot+(1-s)\ov\tau$ with $s\in(0,1)$ if and only if 
		\begin{equation}\label{e.Iassumption}
			\II_\slroot(\ov\slroot)> \hol>1/\II_{\ov\slroot}(\slroot).
		\end{equation}
In this case one has  $\sf T_{\vi}\cal Q_\cc=\spa\{ \hol\slroot-\ov{\slroot}\}$.
\end{lemma}

\begin{proof} Recall from Corollary \ref{strictly} that $\sf T_{\hJ\slroot\slroot}\cal Q_\cc=\ker\II_{\hJ\slroot\slroot}$ and $\cal Q_\cc$ is strictly convex. 
Furthermore, by definition the functional $\vi$ is the point of $\cal Q_\cc$, that minimizes the norm $\|\; \|^{1,\hol}$. The level set $\{\|\varphi \|^{1,\hol}=1\}$ is a rhombus with vertices $(\hol\slroot,\ov\slroot)$ (in blue in Figure \ref{phinfty}), the tangent to $\cal Q_\cc$ at $\hJ\slroot\slroot,$ in red in Figure \ref{phinfty}, is the level set $\II_{\hJ\slroot\slroot}(\cdot)=1$,  whence its intersection with the $\ov\slroot$-axis is $\ov\slroot/\II_{\hJ\slroot\slroot}(\ov\slroot)$, and the the tangent to to $\cal Q_\cc$ at $\hJ{\ov\slroot}\ov\slroot$ is the level set $\II_{\hJ{\ov\slroot}\ov\slroot}(\cdot)=1$,  and it intersects the $\slroot$-axis is $\slroot/\II_{\hJ{\ov\slroot}\ov\slroot}(\slroot)$.
	
Equation \eqref{e.Iassumption} is thus equivalent to the fact that the slope of the side of the rhombus, equal to $-1/\hol$, is between the slope of the tangent at $\hJ\slroot\slroot$, which is equal to $-\hJ\slroot/\II_{\hJ\slroot\slroot}(\ov\slroot)=-1/\II_{\slroot}(\ov\slroot)$, and the slope of the tangent at $\hJ{\ov\slroot}\ov\slroot$, which is equal to $-\II_{\hJ{\ov\slroot}\ov\slroot}(\slroot)/\hJ{\ov\slroot}=-\II_{\ov\slroot}(\slroot).$ 

Strict convexity of $\cal Q_\cc$ ensures that this is equivalent to having a unique point in $\cal Q_\cc\cap\{t\ov\slroot:t>0\}\times\{s\slroot:s>0\}$ tangent to the side of the rhombus, which is the desired functional $\vi$. \end{proof}

 \begin{figure}[h]\centering
\includegraphics[width=90mm]{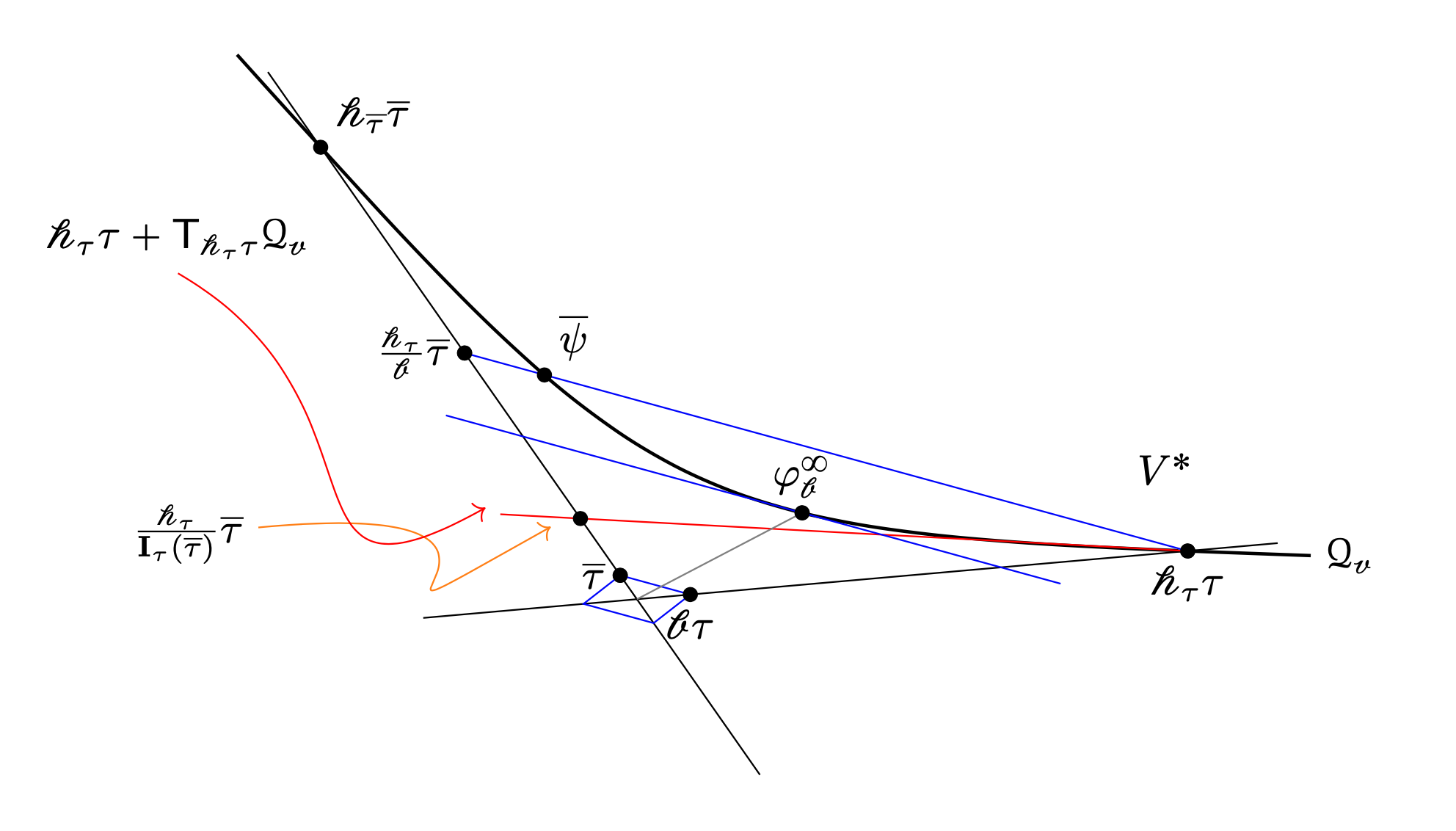}
 \caption{The situation of Lemma \ref{paraphi}.} \label{phinfty}
 \end{figure}

We thus obtain the following key properties of $\vi$:
\begin{lemma}\label{phi>min} Under the assumptions of Theorem \ref{Hffconical} one has  		\begin{enumerate}
	\item $\sf u_\vi^\cc=\Pi(\ker(\hol\slroot-\ov\slroot))$;
	\item  for any $v\in V^+$ one has $$ \vi(v)\geq \hC{\infty,\hol}\hol\min\{\slroot(v),\ov{\slroot}(v)\}.$$
	\end{enumerate}
 Moreover one has $\hC{\infty,\hol}<\min\{\hJ{\ov\slroot},\hJ{\slroot}/\hol\}$.
\end{lemma}

\begin{proof} Lemma \ref{paraphi}  implies that
	\begin{enumerate}
		\item $\sf T_{\vi}\cal Q_\cc=\spa\{\hol\slroot-\ov{\slroot}\}$ and thus $$\sf u_\vi^\cc=\Ann\big(\R\cdot(\hol\slroot-\ov{\slroot})\big)=\Pi(\ker(\hol\slroot-\ov\slroot)).$$
		\item  $\vi/\|\vi\|^{1,\hol}=s\hol\slroot+(1-s)\ov{\slroot}$ for some $s\in(0,1)$ and hence\footnote{ Indeed, if $x,y\geq0,$ $s\in(0,1)$ and $\hol\in(0,1]$ one has: $s\hol x+(1-s)y\geq\hol\min(x,y)$: Assume for example that $y\geq x$ (the other case follows smilarly), then $$s\hol x+(1-s)y-\hol x\geq (1-s)(1-\hol)x\geq0.$$}, since $\hol\in(0,1]$, $$\vi\geq\|\vi\|^{1,\hol}\hol\min\{\slroot,\ov{\slroot}\}.$$
	\end{enumerate}
 
In order to prove item (ii), we need to show that $ \hC{\infty,\hol}\leq\|\vi\|^{1,\hol}$.  
Since  $\vi\big(\cartan_\t((\rho,\ov\rho)\g)\big)\leq \|\Pi(\cartan_\t((\rho,\ov\rho)\g))\|_{\infty,\hol}\|\vi\|^{1,\hol}$, we deduce,  for all $s>\|\vi\|^{1,\hol}$, 
$$\sum_{\g\in\G}e^{-s\|\Pi(\cartan_\t((\rho,\ov\rho)\g))\|_{\infty,\hol}}\leq \sum_{\g\in\G}e^{-(s/\|\vi\|^{1,\hol})\vi\big(\cartan_\t((\rho,\ov\rho)\g)\big)}<\infty$$ where  last inequality holds as $\hC\vi=1$ (by Equation \eqref{qh1} and Remark \ref{r.entropy=crit}).

The last assertion follows directly from the definitions:
\begin{alignat*}{2}
	\hC{\infty,\hol}&=\lim_{t\to\infty}\frac 1t\log\#\big\{\g\in\G:\max\big\{\hol\slroot(\cartan(\g)),\ov\slroot(\cartan(\ov\g))\big\}\leq t\big\}\\
	&\leq\lim_{t\to\infty}\frac 1t\log\#\big\{\g\in\G:\hol\slroot(\cartan(\g))\big\}\leq t\big\}=\hC{\slroot}/\hol=\hJ{\slroot}/\hol,
	\end{alignat*}
 where the last equality follows from Remark \ref{r.entropy=crit}. The inequality $	\hC{\infty,\hol}\leq \hJ{\ov\slroot}$ is analogous.
\end{proof}

Let $\ps^\vi$ be the Patterson-Sullivan measure associated to $\vi$ by Corollary \ref{existe}. Combining Equation \eqref{e.CartBasin}, Equation \eqref{e.conesinCartan} and Corollary \ref{existe} we deduce that, for every $\g\in\G$,  
\begin{equation}\label{shadow}
	\ps^\vi\big(\g\cone_\infty^{c}(\g)\big)\leq Ce^{-\vi\big(\cartan_\t((\rho,\ov\rho)\g)\big)}\leq Ce^{-\hC{\infty,\hol}\hol\min\big\{\slroot(\cartan(\g)),\ov{\slroot}(\cartan(\ov\g))\big\}},
\end{equation} 
where the last inequality comes from Lemma \ref{phi>min}.

By Proposition \ref{conetypesBalls} there exist constants $ c,k_1$ and $\ov k_1$ such that if $(\alpha_i)_0^\infty$ is a geodesic ray from $\id$ to $x$ then for all $i$ the subsets 
$$ \xi\big(\alpha_i\cone_\infty^{c}(\alpha_i)\big)\textrm{ and }\ov{\xi}\big(\alpha_i\cone_\infty^{c}(\alpha_i)\big)$$ contain balls on the corresponding projective spaces of radi $$ k_1e^{-\slroot(\cartan(\alpha_i))} \textrm{ and } \ov k_1e^{-\ov{\slroot}(\cartan(\ov\alpha_i))}$$
respectively where $k_1, \ov k_1$ depend on the representations but not on $i$. Since $\varXi(\bord\G)$ is a graph, the preceding radius computation implies that the image of the cone type $\varXi\big(\alpha_i\cone_\infty^{c}(\alpha_i)\big)$ contains the intersection of $\bord\G\times\bord\ov\G$ with a ball, for the product metric on $\P(\K^d)\times\P(\K^{\ov d})$, of radius
\begin{equation}\label{radiusgraph}
	ke^{-\min\big\{\slroot(\cartan(\alpha_i)),\ov{\slroot}(\cartan(\ov\alpha_i))\big\}},
\end{equation}

\noindent
for some uniform constant $k.$ This set of balls forms a fine set of neighbourhoods around any point $x\in\partial\G$. Combining this with Equation (\ref{shadow}) and the fact that $\mu^\vi$ is supported on $\bord\G$, one has, possibly enlarging the constant $C,$ that for all $r$ the measure of the ball of radius $r$ about $\varXi(x)$ is $$ \ps^\vi\big(B(x,r)\big)\leq Cr^{-\hC{\infty,\hol}\hol}.$$

Since $\dim V^*=2$ and $\cc_{(\rho,\ovrho)}$ is assumed non-arithmetic, Theorem \ref{ps-sandwich} states that the subset of $\hol$-conical points has full $\ps^\vi$ measure. Applying Corollary \ref{edgar} one concludes that $$\Hff\big(\varXi\{\hol-\textrm{conical points}\}\big)\geq\hol\hC{\infty,\hol}.$$

\subsection{The upper bound}
We now prove the second inequality. 
\begin{prop}\label{upper} Let $\rho,\ov\rho$ be locally conformal representations over $\K$ and $\ov\K$. For every $\hol\leq 1$,
	$$\Hff\big(\varXi\{\hol-\mathrm{conical\ points}\}\big)\leq\min\{\hC{\infty,\hol},\hol\hC{\infty,\hol}+(1-\hol) \}.$$
\end{prop}	

\begin{proof}We say that a point $x$ is \emph{$(R,\hol)$-conical} if there exists a geodesic ray $(\alpha_i)_{i\in\N}$ converging to $x$ and such that for an infinite subset $\I\subset \N$ of indices and for every $k\in \I$
\begin{equation}\label{e.5.11}\Big|\hol\slroot\big(\cartan(\alpha_k)\big)-\ov{\slroot}\big(\cartan(\ov\alpha_k)\big)\Big|\leq R.
\end{equation}
	We denote by $\sf C_\hol^R$ the set of $(R,\hol)$-conical points. By Lemma \ref{conGeo}	one has
		$${\displaystyle\bigcup_{R>0}\sf C_\hol^R=\{x\in\bord\G: x\textrm{ is }\hol-\textrm{conical}\}},$$ and thus by Lemma \ref{l.Hffunion} it suffices to show that for every $R$  one has
	$$\Hff\left(\sf C_\hol^{R}\right)\leq \hC\infty.$$
	
	For any constant $K>0$ and any $\g\in\G$ we denote by $B_{\g}^{\max, K}$ the open ball of $\P(\K^{d})\times \P(\K^{\ov d})$ given by:
	$$ B_{\g}^{\max, K}:= B\Big(\big(U_1(\g),U_1(\ov\g)\big), K e^{-\max{\{\hol\slroot(\cartan(\g)), \ov{\slroot}(\cartan(\ov\g))\}}}\Big),$$
	and denote by 
	$$\cal U_T^K:=\big\{B_{\g}^{\max, K}|\, |\g|\geq T\big\}.$$	
	Let $K$, resp. $\ov K$, be the constants given by Proposition \ref{p.coneinBall} for the representation $\rho$ (resp. $\ovrho$).
	
	We first observe that for $C=2e^R\max\{K,\ov K\}$ and every $T>0$, the set $\cal U_T^C$ covers $\varXi(\sf C_\hol ^R)$. Indeed, if $x\in\sf C_\hol^R$ consider the  geodesic ray $(\alpha_i)_{i\in\N}$ converging to $x$, and the set $\I$ of indices for which Equation \eqref{e.5.11} holds. Then for every $k\in\I$ one has,  since $\hol\leq1$, that \begin{alignat}{2} \label{e.5.12.a} 
	\slroot\big(\cartan(\rho\alpha_k))& \geq\hol\slroot\big(\cartan(\alpha_k))>\max\big\{\hol\slroot\big(\cartan(\alpha_k)\big), \ov\slroot\big(\cartan(\ov\alpha_k)\big)\big\}-R, \\ \label{e.5.12} \ov\slroot\big(\cartan(\ov\alpha_k)) & >\max\big\{\hol\slroot\big(\cartan(\alpha_k)\big), \ov\slroot\big(\cartan(\ov\alpha_k)\big)\big\}-R.
	\end{alignat} Let now $T$ be fixed and choose $k\in \I$, $k>T$. Since $x\in\alpha_k\cone_\infty^{c}(\alpha_k)$, Proposition \ref{p.coneinBall} together with Equation \eqref{e.5.12} give 
\begin{alignat*}{2} 
	d\big(\xi(x),U_1(\alpha_k)\big) & \leq Ce^{-\max\big\{\hol\slroot\big(\cartan(\alpha_k)\big), \ov\slroot\big(\cartan(\ov\alpha_k)\big)\big\}}\\ 
	d\big(\ov\xi(x),U_1(\ov\alpha_k)\big) & \leq Ce^{-\max\big\{\hol\slroot\big(\cartan(\alpha_k)\big), \ov\slroot\big(\cartan(\ov\alpha_k)\big)\big\}},
\end{alignat*} as desired.
	
Furthermore, by definition of $\hC{\infty,\hol}$, for every $s>\hC{\infty,\hol}$,  
	$$\sum_{U\in \cal U_T^C}\diam U^s\leq2^sC^s\sum_{|\g|\geq T}e^{-s\max\{\hol\slroot(\cartan(\g)),\ov\slroot(\cartan(\ov\g))\}}<+\infty,$$ 
	whence, Equation \eqref{HffDef} yields
	$\Hff(\sf C_\hol^R)\leq\hC{\infty,\hol}.$ 
In order to obtain the second upper bound we observe that, if $\alpha\in\G$ satisfies Equation \eqref{e.5.11},  the set $\varXi(\alpha\cone_\infty^{c}(\alpha))$ can be covered with $e^{(1-\hol)\slroot(\cartan(\alpha))}$ balls of radius $2Ce^{-\slroot(\cartan(\alpha))}$. We denote by $\cal U_T$ the collection of open balls, that only take into account elements $\alpha\in\G$ with $|\alpha|>T$ that verify \eqref{e.5.11}, which in particular covers the set $\sf C_\hol^R$. Using Equation \eqref{e.5.12.a} we obtain  
\begin{alignat*}{2} 
\sum_{U\in \cal U_T}\diam U^s
&\leq 2^sC^s\sum_{|\g|\geq T}e^{(1-\hol)\slroot(\cartan(\g))}e^{-s\slroot(\cartan(\g))}\\
&\leq 2^sC^s\sum_{|\g|\geq T}e^{-(s-(1-\hol))\slroot(\cartan(\g))}\\
&\leq 2^sC^se^{\frac{R(s-1+\hol)}\hol}\sum_{|\g|\geq T}e^{-\frac{(s-(1-\hol))}\hol\max\{\hol\slroot(\cartan(\g)),\ov\slroot(\cartan(\ov\g))\}}.
\end{alignat*}

Since the latter quantity is finite whenever $\frac{(s-(1-\hol))}\hol>\hC{\infty,\hol}$, we deduce $$\Hff(\sf C_\hol^R)<\hol\hC{\infty,\hol}+(1-\hol).$$
\end{proof}

We conclude this subsection computing the Hausdorff dimension of the image of the whole boundary through the graph map. See \cite{HffGraph} for examples of homeomorphisms between Cantor sets for which the Hausdorff dimension of the graph exceeds the maximal Hausdorff dimension of the factors.
\begin{prop}\label{p.uppergeneral}
	Let $\rho:\G\to\SL(d,\K)$, $\ov\rho:\G\to\SL(\ov d,\ov\K)$ be locally conformal. Then
	$$\Hff(\varXi(\partial\G))=\max\{\hC{\tau},\hC{\ov \tau}\}$$
\end{prop}	
\begin{proof}
	This follows as in the proof of Proposition \ref{upper} considering the covers of $\varXi(\partial\Gamma)$ given by $\cal U_T^C:=\big\{B_{\g}^{\min, C}|\, |\g|\geq T\big\}$ with	
		$$B_{\g}^{\min, K}:= B\Big(\big(U_1(\g),U_1(\ov\g)\big), K e^{-\min{\{\slroot(\cartan(\g)), \ov{\slroot}(\cartan(\ov\g))\}}}\Big),$$
\noindent	
and $C=2\max\{K, \ov K\}$ where $K$ (resp. $\ov K$) is the constant given by Proposition \ref{p.coneinBall} for the representations $\rho$ (resp. $\ov \rho$). To conclude it is enough to observe that
	$$\displaystyle\hC{\min\{\slroot, \ov\slroot\}}=\max\{\hC{\slroot}, \hC{\ov \slroot}\},$$ a fact proven for example in P.-S.-Wienhard \cite[Lemma 5.1]{PSW2}.
\end{proof}	
It is easy to generalize Proposition \ref{p.uppergeneral} to an arbitrary number of factors. as an application we get.
\begin{cor}
	Let $\rho:\G\to\SL(d,\K)$ and $\theta\subset\Delta$ be such that for all $\slroot_i\in\theta$, $\Phi_{\slroot_i}\circ\rho$ is $(1,1,2)$-hyperconvex. Then
	$$\Hff(\xi^\theta_\rho(\partial\G))=\max_{\slroot_i\in\theta}\hC{\slroot_i}.$$ 
\end{cor}	

\subsection{Proof of Theorem \ref{Hffconical}}\label{proofHffconical} The first inequality is established in \S\,\ref{s.lower}, the second inequality is proven in Proposition \ref{upper}, the third inequality follows from Lemma \ref{phi>min} and the fourth from Theorem \ref{hyperh=1}. The last equality was stablished in Proposition \ref{p.uppergeneral}.

\section{$\hol$-concavity and $\hol$-conicality: Final steps for the proof of Theorem \ref{LCdiff}}\label{teoLCdiff}

The goal of this section is to prove the following more general version of Theorem \ref{LCdiff}. As before, fix $\{\K,\ov\K\}\subset\{\R,\C,\H\}$ together with locally conformal representations $\rho:\G\to\SL(d,\K)$ and $\ov\rho:\G\to\SL(\ov d,\ov\K)$ of an arbitrary word-hyperbolic group $\G$. For $\hol\in(0,1]$ recall that $\Xi:\xi(\bord\G)\to\ov\xi(\bord\G)$ is \emph{$\hol$-concave} at $x\in\bord\G$ if there exists $y_k\to x$ such that the incremental quotients \begin{equation}\label{incrlimit}\frac{d_{\P}\big(\ov\xi(x),\ov\xi(y_k)\big)}{d_{\P}\big(\xi(x),\xi(y_k)\big)^\hol}\end{equation} are bounded away from $0$ and $\infty$ (independently of $k$). We also let $\Ext_{(\rho,\ov\rho)}^\hol$ be the set of $x\in\bord\G$  that are $\hol$-concavity points of $\Xi$. Finally, recall that $\rho$ and $\ov\rho$ are \emph{not gap-isospectral} if there exists $\g\in\G$ such that $\slroot(\cartan(\g))\neq\ov\slroot(\cartan(\ov\g)).$

\begin{thm}\label{tutti.LCdiff}
	Let $\rho, \ov\rho$ be locally conformal representations acting irreducibly, on $\K^d$ and ${\ov\K}{}^{\ov d}$ respectively, as real vector spaces, and that are not gap-isospectral. Consider any $\hol\in(0,1]$ with $\II_{\slroot}(\ov\slroot)>\hol>(\II_{\ov\slroot}(\slroot))^{-1}$, then
		\begin{itemize}\item[-] if $\{\K,\ov\K\}\subset\{\R,\C\}$ one has 
			\begin{alignat}{2}\label{holdineq}
				\hol\hC{\infty,\hol}\leq\Hff({\Ext}_{\rho,\ov\rho}^\hol)&\leq\min\{\hC{\infty,\hol},\hol\hC{\infty,\hol}+(1-\hol) \}\nonumber\\&<\min\{\hJ{\ov\slroot},\hJ{\slroot}/\hol\}\nonumber\\&\leq\Hff(\varXi(\bord\G))\\&=\max\{\hJ{\slroot},\hJ{\ov\slroot}\};\end{alignat}
	
			\item[-] if $\K=\H$ (resp. $\ov\K=\H$), Equation \eqref{holdineq} holds if we further assume that the real Zariski closure of $\rho(\G)$ (resp. of $\ov\rho(\G)$) does not have compact factors. \end{itemize}
\end{thm}

\subsection{Hyperplane conicality and the concavity condition}

We commence with a lemma relating $\hol$-conicality to the desired concavity properties of the equivariant map $\Xi:\xi(\bord\G)\to\ov\xi(\bord\G)$.

\begin{lemma}\label{generalcase} Let $\rho$ and $\ov\rho$ be locally conformal representations over $\K$ and $\ov\K$ respectively, and $\hol\in(0,1]$. Then one has $\{\hol\mathrm{-conical\ points\ of\ }(\rho,\ov\rho)\}=\Ext_{\rho,\ov\rho}^\hol.$
\end{lemma}

\begin{proof}  Let $(\alpha_i)_{i\in\N}$ denote a geodesic ray converging to $x$. Proposition \ref{conetypesBalls} gives constants $C_1,C_2, \ov C_1, \ov C_2$ and $L\in\N$ such that, for every $n\in\N$ and every $y_n\in \alpha_n\cone_{\infty}^{c}(\alpha_n)\setminus\alpha_{n+L}\cone_\infty^{c}(\alpha_{n+L})$, it holds \begin{alignat}{3}\label{e.nestcone} C_1e^{-\slroot(\cartan(\alpha_n))} & <d_\P\big(\xi(y_n),\xi(x)\big) &<C_2e^{-\slroot(\cartan(\alpha_n))}, \nonumber \\ \ov C_1e^{-\ov{\slroot}(\cartan(\ov\alpha_n))}& <d_\P\big(\ov\xi(y_n),\ov\xi(x)\big) &<\ov C_2e^{-\ov\slroot(\cartan(\ov\alpha_n))}.\end{alignat}

Assume first that $x$ is $\hol$-conical. By Definition \ref{flat} we obtain a geodesic ray $(\alpha_i)_0^\infty,$ an infinite set of indices $\I\subset\N$ and a number $R,$ such that for all $k\in\I$ one has 
\begin{equation}\label{e.conical}
|\hol\slroot(\cartan(\alpha_k))-\ov\slroot(\cartan(\ov\alpha_k))|<R.
\end{equation} 

For each such $k$ we choose a point $y_k\in \alpha_k\cone_\infty^{c}(\alpha_k)\setminus \alpha_{k+L}\cone_\infty^{c}(\alpha_{k+L})$. By construction $y_k$ converges to $x$. Combining both equations, for every $k\in \I$ it holds $$ e^{-R}\frac{\ov C_1}{ {C_2}^\hol}\leq \frac{d_\P\big(\ov\xi(y_k),\ov\xi(x)\big)}{d_\P\big(\xi(y_k),\xi(x)\big)^\hol} \leq e^{R}\frac{\ov C_2}{ {C_1}^\hol},$$ so the incremental quotient \eqref{incrlimit} is uniformly far from $0$ and $\infty$. Whence $\{\hol-\textrm{conical points}\}\subset\Ext_{\rho,\ov\rho}^\hol$.

Conversely, assume that $x$ is not $\hol$-conical. The Cartan projections of two consecutive elements $\alpha_i$ and $\alpha_{i+1}$  make uniformly bounded gaps (Proposition  \ref{p.Ben}), and thus there exists $C$ such that for all $n\in\N$ one has $$\big|\slroot\big(\cartan(\alpha_{n+1})\big)-\slroot\big(\cartan(\alpha_{n})\big)\big|<C.$$ As a consequence, we can assume,  up to switching the roles of $\rho$ and $\ov\rho$, that for any $R$ there exists $n_R$ such that for every  $n> n_R$ one has $$ \hol\slroot\big(\cartan(\alpha_{n})\big)-\ov{\slroot}\big(\cartan(\ov\alpha_{n})\big)>R.$$ 
In turn this implies, thanks to Equation \eqref{e.nestcone}, that for every $y\in \alpha_{n_R}\cone_\infty^{c}(\alpha_{n_R})$, 

$$ \frac{d_\P\big(\ov\xi(y),\ov\xi(x)\big)}{d_\P\big(\xi(y),\xi(x)\big)^\hol}\leq e^{-R}\frac {\ov C_2}{  {C_1}^\hol}.$$

Since $R$ is arbitrary, and the sets $\alpha_{n_R}\cone_\infty^{c}(\alpha_{n_R})$ form a system of neighborhoods of the point $x$, we deduce that the limit in Equation \eqref{incrlimit} exists and equals $0$. This concludes the proof. \end{proof}

\subsection{Non-arithmeticity of periods}

In this section we establish a non-arithme\-ti\-ci\-ty condition, necessary to apply later Theorem \ref{Hffconical}. This is established in a rather general setting. Recall that a subgroup $\grupo<\SL(d,\K)$ is \emph{$\K$-proximal} if it contains a $\K$-proximal element, i.e. there exists $g\in\grupo$ such that $\slroot_1(\lambda(g))>0$.

\begin{prop}\label{nonA} Let $\grupo$ be a finitely generated group. Let $\rho:\grupo\to\SL(d,\K)$ and $\ov\rho:\grupo\to\SL(\ov d,\ov\K)$ be two $\K$-proximal representations that act irreducibly on $\K^d$ and $\ov\K{}^{\ov d}$ respectively, as real vector spaces. Assume there exists $\g\in\grupo$ such that $\slroot_1(\lambda(\rho \g))\neq\ov\slroot_1(\lambda(\ov \rho \g))$. If  $\{\K,\ov\K\}\subset\{\R,\C\}$, then the group generated by the pairs $$\Big\{\Big(\slroot_1\big(\lambda(\rho \g)\big),\ov\slroot_1\big(\lambda(\ov \rho \g)\big)\Big):\g\in\grupo\Big\}$$ is dense in $\R^2$.  If $\K=\H$ we further assume that the Zariski closure over $\R$ of $\rho(\grupo)$ has no compact factors, and the same for $\ov\rho(\grupo)$ if moreover $\ov\K=\H$, then the same conclusion holds.\end{prop}

To prove the proposition we need Lemmas \ref{semisimple} and \ref{semi} below.

\begin{lemma}\label{semisimple} Let $\K$ be either $\R$ or $\C$. Let $\grupo<\SL(d,\K)$ be a subgroup acting irreducibly on $\K^d$ as a real vector space and assume $\grupo$ contains a $\K$-proximal element. Then the real Zariski closure of $\grupo$ is semi-simple, has finite center and without compact factors. 
\end{lemma}

\begin{proof} If $\K=\R$ the Lemma is the content of S. \cite[Lemma 8.6]{entropia} and the proof over $\C$ is a slight modification of the latter. Indeed, let $\sf G$ be the Zariski closure of $\rho(\grupo)$ over the reals, by the irreducibility assumption it is a reductive (real-algebraic) group. By Schur's Lemma the elements commuting with $\grupo$ consist only on homotheties, but since we're in special linear group one has that the center of $\sf G$ is finite.

The group $\sf G$ is then semi-simple and we let $K$ be the identity component of the product of all the compact simple factors of $\sf G$. We also let $H$ be the identity component of the product of all the non-compact simple factors of $\sf G$. The groups $H$ and $K$ commute and one has $HK$ has finite index in $\sf G$.

Consider a proximal $g\in\sf G$, up to a fixed power we may write $g=kh$ with $k\in K$ and $h\in H$. Since $K$ is compact, its eigenvalues have modulus one so we conclude that $h$ is proximal and that $g_+=h_+$. The attracting line of $h$ is thus invariant under $K$. Since $K$ is connected, an element of $ K$ acts on $h_+$ as multiplication by some element of $\circle$. 

By irreducibility we may find a basis of $\C^d$ consisting on fixed attracting lines of proximal elements of $H$. This basis simultaneously diagonalizes $K$, so we get an injective map from $K$ to a compact group isomorphic to a $d$-dimensional torus. Consequently $K$ is abelian, and since it commutes with $H$ we conclude that $K$ is contained in the identity component of the center of $\sf G$, which we proved earlier to be trivial.\end{proof}

\begin{lemma}\label{semi} Let $G$ be a semi-simple real-algebraic Lie group with finite center and no compact factors. Fix $\vartheta,\ov\vartheta\subset\simple_G$ two non-empty subsets with $\vartheta\cap\ov\vartheta=\emptyset$. Let $\grupo$ be a group and $\scr r:\grupo\to G$ a representation with Zariski-dense image. Then, for every closed cone with non-empty interior $\scr C\subset\inte\calL_{\scr r(\grupo)}$, the group spanned by the pairs $$\Big\{\big(\min_{\sigma\in\vartheta}\sigma\big(\lambda(\scr r g)\big),\min_{\ov\sigma\in\ov\vartheta}\ov\sigma\big(\lambda(\scr r g)\big)\big):g\in\grupo\textrm{ and }\lambda(\scr rg)\in\scr C\Big\}$$ is dense in $\R^2$.
\end{lemma}

\begin{proof} Define the piecewise linear maps $\slroot,\ov\slroot:\a^+\to\R$ by:
\begin{alignat*}{2}\slroot(v) & =\min\big\{\sigma(v):\sigma\in\vartheta\big\}\nonumber\\ \ov\slroot(v) & =\min\big\{\ov\sigma(v):\ov\sigma\in\ov\vartheta\big\}.\end{alignat*}

The vanishing set of the difference $\slroot-\ov\slroot$ is contained the union of $\ker(\sroot-\bb)$ for arbitrary $\sroot\in\vartheta$ and $\bb\in\ov\vartheta$. Since $\vartheta$ and $\ov\vartheta$ are disjoint, this is a union of hyperplanes of $\a$, from which we deduce that the set of zeroes of $\slroot-\ov\slroot$ has empty interior.

Since $\scr C\subset \inte\calL_{\scr r(\grupo)}$ has non-empty interior, the difference $\slroot-\ov\slroot$ does not identically vanish on $\scr C$. Since $\slroot$ and $\ov\slroot$ are piecewise linear, we can choose a possibly smaller closed cone with non-empty interior $$\scr C'\subset\scr C,$$ and $\sroot\in\vartheta$, $\ov\bb\in\ov\vartheta$ such that for all $v\in\scr C'$ one has $$\slroot\times\ov\slroot(v):=(\slroot(v),\ov\slroot(v))=(\sroot(v),\ov\bb(v)).$$ Since $\sroot$ and $\bb$ are distinct simple roots the map $ (\sroot,\bb):\a\to\R^2$ is surjective.

By Benoist \cite[Proposition 5.1]{limite} there exists a sub-semigroup $\grupo'<\grupo$ such that $\scr r(\grupo')$ is a Zariski-dense Schottky  semi-group with $\calL_{\scr r(\grupo')}=\scr C'$. In particular, for all $\g\in\grupo'$ one has $$\slroot\times\ov\slroot\big(\lambda(\scr r\g)\big)=\big(\sroot\big(\lambda(\scr r\g)\big),\ov\bb\big(\lambda(\scr r\g)\big)\big).$$ By Benoist's Theorem \ref{densidad}, stating that the group generated by the Jordan projections $\lambda(\scr r\g)$, for $\g\in\grupo'$, is dense in $\a$, we conclude that the group spanned by $$\Big\{\Big(\big(\sroot\big(\lambda(\scr r\g)\big),\ov\bb\big(\lambda(\scr r\g)\big)\big)\Big):\g\in\grupo'\Big\}$$ is dense in $\R^2$, giving in turn the desired conclusion.\end{proof}

\begin{proof}[Proof of Proposition \ref{nonA}]
Denote by $\sf G$ and $\ov{\sf G}$ the Zariski closures of $\rho(\grupo)$ and $\ov\rho(\grupo)$ respectively. Both $\sf G$ and $\ov{\sf G}$ are semi-simple, have finite center, and don't have compact factors: if $\{\K,\ov\K\}\subset\{\R,\C\}$ then this is the content of Lemma \ref{semisimple}, if either $\K$ and/or $\ov\K$ equals $\H$ then this is an assumption. We let $\iota:\grupo\to\sf G$ and $\ov\iota:\grupo\to\ov{\sf G}$ be the respective inclusions.

If we let $\phi:\sf G\to\SL(d,\K)$ and $\ov\phi:\ov{\sf G}\to\SL(\ov d,\ov\K)$ be the associated real representations, so that $\rho=\phi\circ\iota$ and $\ov\rho=\ov\phi\circ\ov\iota$, we have from \S\ref{representaciones}  two subsets of simple roots $\t:=\t_\phi$ and $\ov\t:=\t_{\ov\phi}$ such that for all $a\in\a_\sf G^+$ and $b\in\a_{\ov{\sf G}}^+$ one has \begin{alignat}{2}\label{piecewise}\slroot(a):=\slroot_1(\phi(a)) & =\min\big\{\sroot(a):\sroot\in\t\big\}\nonumber\\ \ov\tau(b):=\ov\slroot_1(\ov\phi (b)) & =\min\big\{\ov\sroot(b):\ov\sroot\in\ov\t\big\}.\end{alignat} In particular, for every $\g\in\grupo$ one has  $\slroot_1(\lambda(\g))=\slroot(\lambda_\sf G(\iota\g))$, and similarly for $\ov\rho$.

Since $\phi$ and $\ov\phi$ are faithful, $\t$ and $\ov\t$ contain at least one root of each factor of, respectively, $\sf G$ and $\ov{\sf G}$. If $\vartheta\subset\t$ then we let $$\slroot^{\vartheta}(v)=\min_{\sigma\in\vartheta}\sigma(v),\ v\in\a_{\sf G}.$$

If $\sf H$ is a non-trivial product of simple factors of $\sf G$ then we let $\iota_\sf H:\grupo\to\sf H$ be the composition of $\iota$ with the projection of $\sf G$ onto $\sf H$. By Zarisk-density of $\iota(\grupo)$, each representation $\iota_\sf H$ has Zariski-dense image (though unlikely to be discrete). We also let $$\t_{\sf H}  =\t\cap\simple_{\sf H}.$$ Each $\t_\sf H$ is non-empty. We analogously define $\ov\iota_{\ov{\sf H}}$, $\t^{\ov{\sf H}}$ and $\ov\slroot^{\ov{\sf H}}$.

We now let $\sf L$ be the largest product of simple factors, simultaneously of $\sf G$ and $\ov{\sf G}$, so that $\iota_\sf L$ is conjugated (up to finite index) to $\ov\iota_\sf L$. Let $\sf H$ and $\ov{\sf H}$ be the remaining factors of $\sf G$ and $\ov{\sf G}$ respectively, i.e. $$\sf G=\sf L\times \sf H\textrm{ and }\ov{\sf G}=\sf L\times\ov{\sf H},$$ and moreover, by definition of $\sf L$, the representation $\scr r:\grupo\to \sf L\times\ov{\sf H}\times\sf H$ \begin{equation}\label{repsZ1} \scr r:g\mapsto \big(\iota_{\sf L}(g),\ov\iota_{\ov{\sf H}}(g),\iota_\sf H(g)\big)\end{equation} has Zariski-dense image, see for example Bridgeman-Canary-Labourie-S. \cite[Corollary 11.6]{pressure}. We remark that we are not assuming that any of $\sf L$, $\ov{\sf H}$ or $\sf H$ is non-trivial (they can't, of course, be all trivial).

If $(u,v,w)\in\a_{\sf L}\times\a_{\ov{\sf H}}\times\a_{\sf H}$ we naturally think of $(u,v)$ as an element of $\a_{\ov{\sf G}}$ and of $(u,w)$ as an element of $\a_{\sf G}$.  We now write \begin{alignat*}{2}\Theta & =\t_\sf L\cap\ov\t_\sf L,\\ \Theta_{\sf L} &=\t_\sf L\setminus\Theta,\\\ov\Theta_{\sf L} & =\ov\t_{\sf L}\setminus\Theta.\end{alignat*}

One has, for all $(u,v,w)\in\a_{\sf L}\times\a_{\ov{\sf H}}\times\a_{\sf H}$ that \begin{alignat}{2}\label{tauL} \slroot(u,w) & =\min\big\{\slroot^{\Theta_\sf L}(u),\slroot^{\Theta}(u),\slroot^{\t_\sf H}(w)\big\}\nonumber\\ \ov\slroot(u,v) & =\min\big\{\slroot^{\ov\Theta_\sf L}(u),\slroot^{\Theta}(u),\ov\slroot^{\t_{\ov{\sf H}}}(v)\big\}.\end{alignat}

By assumption, there exists $g\in\grupo$ such that $\rho(g)$ and $\ov\rho(g)$ are proximal and 
$\slroot(\lambda_{\sf G}(\iota g))\neq\ov\slroot(\lambda_{\ov{\sf G}}(\ov\iota g)).$ Assume, without loss of generality, that \begin{equation}\label{eqA}  \slroot(\lambda_{\sf G}(\iota g))<\ov\slroot(\lambda_{\ov{\sf G}}(\ov\iota g)).\end{equation} By means of Equations \eqref{tauL} we see that in this situation one has $$\slroot^{\Theta_{\sf L}\cup\t_\sf H}(\lambda_\sf G(\iota g))=\slroot(\lambda_\sf G(\iota g))<\ov\slroot(\lambda_{\ov{\sf G}}(\ov\iota g)),$$ in particular the union $\Theta_\sf L\cup\t_\sf H$ must be non-empty. Moreover, this strict inequality yields the existence of a small closed cone with non-empty interior $\scr C_0\subset\calL_\rho\subset\a_{\sf G}^+$ about $\R_+\lambda_\sf G(\rho g) $ such that  \begin{equation}\label{tauH=tau1}\slroot^{{\Theta_{\sf L}\cup\t_\sf H}}(a)=\slroot_1(a)\,\forall a\in\scr C_0.\end{equation}

Consider now the representation $\scr r:\grupo\to\sf L\times\ov{\sf H}\times\sf H$ from \eqref{repsZ1} and a closed cone with non-empty interior $\scr C \subset\calL_{\scr r(\grupo)}\subset\a_{\sf L}^+\times\a_{\ov{\sf H}}^+\times\a_{\sf H}^+$ whose natural projection onto $\a_\sf G^+=\a_{\sf L}^+\times\a_{\sf H}^+$ is $\scr C_0$. 

Lemma \ref{semi} applied to the group $G=\sf L\times\ov{\sf H}\times\sf H$, the representation $\scr r$, the disjoint non-empty subsets $\vartheta=\Theta_\sf L\cup\t_\sf H$ and $\ov\vartheta=\ov\t_\sf L\cup\t_{\ov{\sf H}}$ and the cone $\scr C$, provides the desired conclusion.\end{proof}

We conclude with the following Corollary that we don't need but is of independent interest.

\begin{cor}\label{corraiz}Let $\rho:\G\to\SL(d,\K)$ and $\ov\rho:\G\to\SL(\ov d,\ov\K)$ be $\R$-irreducible and $\{\slroot_1,\slroot_2\}$-Anosov and $\{\ov\slroot_1,\ov\slroot_2\}$-Anosov respectively. If $\K=\H$ assume moreover the Zariski closure of $\rho(\G)$ does not contain compact factors, and analogously for $\ov\rho$. If $\rho$ and $\ov\rho$ are not gap-isospectral then $$\II_{\ov\slroot_1}(\slroot_1)>{\hJ{ \ov\slroot_1}}/{\hJ{ \slroot_1}}.$$
\end{cor}

\begin{proof} Since both representations are projective-Anosov they are $\K$-proximal. Pro\-po\-si\-tion \ref{nonA} implies then that, since they are not gap-isospectral, the group spanned by the pairs $\big\{\big(\slroot_1(\lambda(\g)),\ov\slroot_1(\lambda(\ov\g))\big):\g\in\G\big\}$ is dense. Since both representations are also Anosov with respect to 2-dimensional stabilizers, the functionals $\slroot_1$ and $\ov\slroot_1$ lie in the Anosov-Levy space of $\rho$ and $\ov\rho$ respectively, we can apply Proposition \ref{ineq} to obtain the desired strict inequality.
\end{proof}

\subsection{Proof of Theorem \ref{tutti.LCdiff}}\label{ProofA} Theorem \ref{tutti.LCdiff} follows from Proposition \ref{nonA} giving the desired non-arithmecity of periods, Lemma \ref{generalcase} identifying the set $ \Ext_{\rho,\ov\rho}^\hol$ with the set of $\hol$-conical points of $(\rho,\ov\rho)$ and Theorem \ref{Hffconical} computing the Hausdorff dimension of the latter when the periods are non-arithmetic. The last equality is a direct consequence of Proposition \ref{p.uppergeneral}.\qed

\section{Theorem \ref{t.Zcl}: Zariski closures of real-hyperconvex surface-group representations}\label{s.6}

In this section we prove Theorem \ref{t.Zcl} giving a preliminary classification of Zariski closures of irreducible real $(1,1,2)$-hyperconvex representations of surface groups. For most of the section we work with a pair of $(1,1,2)$-hyperconvex representations and eventually reduce the proof of Theorem \ref{t.Zcl} to a situation like this; we will crucially use Theorem \ref{t.C1}.

\subsection{When $\Xi$ has oblique derivative}\label{s.6.1}We prove here a result of independent interest, albeit possibly known to experts. This subsection only requires \S\,\ref{cont} and \S\,\ref{s.Anosov} and will be needed not only for Theorem \ref{t.Zcl} but also for Theorems \ref{tutti} and \ref{thm.tuttiC}.

Either we let $\G$ have boundary homeomorphic to a circle, either we let it be a Kleinian group. In the first case we let $$\rho,\ov\rho:\G\to\Diff^{1+\nu}(\circle)$$ be  H\"older conjugated to action of $\G$ on its boundary; if instead $\G<\PSL(2,\C)$ is a Kleinian group we let $\rho,\ov\rho:\G\to\PSL(2,\C)$ be two convex co-compact representations that lie in the same connected component of the subset of the character variety $\frak X(\G,\PSL(2,\C))$ consisting of convex cocompact representations. 

We let $X$ be either the circle or $\bord\H^3.$ To simplify notation we will denote the action  of $\g\in\G$ on $X$ via $\rho$  by $\g,$  the action via $\ov\rho$ by $\ov \g$, and the limit sets of $\rho$ and $\ov\rho$ by $\bord\G,\ov{\bord\G}\subset X$  respectively.

In both  situations there exists a H\"older-continuous map $$\Xi:X\to X$$ conjugating $\rho$ and $\ov\rho.$ Indeed while in the surface case this holds by definition, in the Kleinian case this is a theorem by Marden \cite{Marden}, see also Anderson's survey \cite[page 32]{SurveyAnderson}:  the equivariant limit map $\Xi:\bord\G\to\bord\G$ conjugating the actions $\rho$ and $\ov\rho$ on their respective limit sets extends to a $\G$-equivariant, Hölder continuous homeomorphism of the whole Riemann sphere $\bord\H^3.$ We study differentiability points of $\Xi$ with oblique derivative.

We let $d$ be either a visual distance on $X$ (in the complex case) or a distance inducing the chosen $\class^1$ structure on the circle $\circle.$ 

\begin{defi}\label{assuA} An action $\rho$ admits a \emph{Lipschitz-compatible cover} if there exists a finite open cover $\cal B$ of $X$ and a map $\G\to\cal B$, $\gamma\mapsto\cal B_\infty(\g)$ such that
	\begin{enumerate}
		\item for any $a,b\in \G$ so that $|ab|=|a|+|b|$ one has 
				\begin{enumerate}
			\item$b\cal B_{\infty}(ab)\subset \cal B_\infty(a)$, 
			\item $\cal B_{\infty}(ab)\subset\cal B_{\infty}(b)$;
		\end{enumerate}
	\item 	there exist   $\lambda>0,$ $C$ and $L\in\N$ such that if 
	$|\g|\geq L$ and $x,y\in\cal B_\infty(\g)$ then $$d(\g x,\g y)\leq Ce^{-|\g|\lambda}d(x,y);$$
	\item there exist constants $r_1,r_2$ and a function $\slroot:\G\to\RR$ with $\slroot(\g)\geq\lambda|\g|$ such that for every $\g\in\G$ and every $x\in\cone_\infty(\g)$, 
	$$B(x,r_1e^{-\slroot(\g)})\subset\g\cal B_\infty(\g)\subset B(x, r_2e^{-\slroot(\g)}).$$ 
	\end{enumerate}	

\end{defi}

The goal of the subsection is to prove the following result, similar arguments can be found in Guizhen \cite{cui} in the context of conjugacies of expanding circle maps.

\begin{prop}\label{c.indepPer} Let $\rho,\ov\rho$ be as above and assume both admit a Lipschitz compatible cover. If there exists $p\in\bord\G$ such that $\Xi$ has a finite non-vanishing derivative (complex derivative in the Kleinian case) at $p$ then $\Xi|\bord\G$ is bi-Lipschitz.
\end{prop}

We work under the assumptions of Proposition \ref{c.indepPer} and  begin its proof with the following lemma. For $\g\in\G$ we denote its derivative at $x\in X$ by $\g'(x)\in\K$ defined, according our two situations, by 
\begin{itemize}
	\item[$X=\mathbb S^1$\phantom{$\partial$}:] the derivative $\tilde\g'(\tilde x)$ of a lift of $\g$ to the universal cover $\R$ of $\mathbb S^1$, and a lift $\tilde x\in\R$ of $x$, the number $\tilde \g'(\tilde x)$ is independent of these choices;
	\item[$X=\bord\H^3$:] we fix an arbitrary point $\infty\notin\bord\G$, identify $X-\{\infty\}$ with $\K$ via the stereographic projection and let $\g'(x)$ be the standard complex derivative.
\end{itemize}

\begin{lemma}\label{controlC} 
	Let $\rho:\G\to\Diff^{1+\nu}(X)$ admit a Lipschitz compatible cover. There exists a constant $\kappa>0$ and $N\in\N$ such that for all $\g\in\G$ with $|\g|\geq N$ and $x,y\in\cal B_\infty(\g)$ one has  $$\big|\log|\g'(x)|-\log|\g'(y)|\big|\leq \kappa d(x,y)^\nu.$$\end{lemma}

\begin{proof} We consider $L$ from Definition \ref{assuA}, 	so that for every  $\eta\in\G$ with $|\eta|\geq L$ and $x,y\in\cal B_\infty(\eta)$ one has 
	\begin{equation}\label{lip}
		d(\eta x,\eta y)\leq Ce^{-|\eta|\lambda}d(x,y).
	\end{equation} 

Since the action is $\class^{1+\nu}$ we can find a positive $K$ such that for every $\beta$ with $|\beta|\leq L$ and $u,w\in X$ one has \begin{equation}\label{derivativeGen}
	\big|\log|\beta'(u)|-\log|\beta'(w)|\big|\leq Kd(u,w)^\nu.\end{equation}

We let then  $K'=\max\{K,KC^\nu\}.$ We begin by showing, by induction on $k$, that if $|\g|=kL$ then for all  $x,y\in\cal B_\infty(\g),$ one has \begin{equation}\label{induc}
	\big|\log|\g'(x)|-\log|\g'(y)|\big|\leq K'\big(\sum_{i=0}^{k-1}e^{-\nu\lambda L i}\big)d(x,y)^\nu.
\end{equation}

Equation \eqref{derivativeGen} gives the base case, so assume that the result holds up to $k-1$. 
We write $\g=\beta\eta$ with $|\beta|=L$, $|\eta|=(k-1)L$. By Definition \ref{assuA} (ib) we have

\begin{equation}\label{conocont}
	\cal B_\infty(\g)\subseteq\cal B_\infty(\eta).
\end{equation}
Applying the chain rule gives that for every $u\in X$ one has $$ \log|\g'(u)|=\log|(\beta)'(\eta u)|+\log|(\beta )'(u)|$$ and thus, when $x,y\in\cal B_\infty(\g),$

\begin{alignat*}{2}
	\big|\log|\g'(x)|-\log|\g'(y)|\big| &  \leq \big|\log|\beta'(\eta x)|-\log|\beta'(\eta y)|\big|+\big|\log|\eta'(x)| -\log|\eta'(y)|\big|\\ 
	& \leq K d(\eta x,\eta y)^\nu  + K'\big(\sum_{i=0}^{k-2}e^{-\nu\lambda L i}\big)d(x,y)^\nu\quad (\textrm{by \eqref{derivativeGen} and induction})\\ 
	& \leq KC^\nu e^{-|\eta|\nu\lambda}d(x,y)^\nu+ K'\big(\sum_{i=0}^{k-2}e^{-\nu\lambda L i}\big)d(x,y)^\nu\quad (\textrm{by  \eqref{conocont} and \eqref{lip}}).
\end{alignat*}

This shows Equation \eqref{induc} which implies that for $\kappa_0=K'/(1-e^{-\nu\lambda L})$,  every $\g\in\G$ whose word-length is an integer multiple of $L$, and $x,y\in\cal B_\infty(\g)$ one has $$\big|\log|\g'(x)|-\log|\g'(y)|\big|\leq \kappa_0 d(x,y)^\nu.$$

To conclude the lemma we consider an arbitrary $\g$ with $|\g|=mL+t$ and $t< L$. We write  $\g=\beta \eta$ with $|\beta|=mL$. By Definition \ref{assuA} (ia) it holds
\begin{equation}\label{e.6.6}
	\eta \cal B_\infty(\g)\subset\cal B_\infty(\beta).
	\end{equation}
Applying the chain rule gives then
\begin{alignat*}{2}
	\big|\log|\g'(x)|-\log|\g'(y)|\big|& \leq \big|\log|\beta'(\eta x)|-\log|\beta'(\eta y)|\big| +\big|\log|\eta'(x)| -\log|\eta'(y)|\big|\\ 
	& \leq \kappa_0d(\eta x,\eta y)^\nu  + Kd(x,y)^\nu \qquad (\textrm{by \eqref{derivativeGen} and \eqref{e.6.6}})\\ 
	& \leq (\kappa_0C^\nu e^{-mL\lambda}+K) d(x,y)^\nu \qquad(\textrm{by \eqref{lip}})
\end{alignat*}
so taking $\kappa=K+\kappa_0C^\nu e^{-L\lambda}$ we conclude the proof.\end{proof}

\begin{proof}[Proof of Proposition \ref{c.indepPer}]
Let $p\in\bord\G$ be such that $\Xi$ has a derivative at $p$ that is neither horizontal nor vertical. Fix a geodesic ray $(\alpha_n)_{0}^\infty$ through the identity with $\alpha_{n}\to p$. By definition for all $n$ one has $p\in\alpha_n\cone_\infty(\alpha_n)$. Without loss of generality we may also assume that $$p=0=\Xi(0)$$ and we may write the derivative as the incremental limit 
$$\Xi'(0)=\lim_{y\to 0}\frac{\Xi(y)}{y}\in\K-\{0\}.$$ 
For each $n$ we let $s_n=r_1e^{-\slroot(\alpha_n)}$, so that by Definition \ref{assuA} (iii), $$B(0,s_n)\subset\alpha_n\cal B_\infty (\alpha_n).$$ 

We consider the scaling map $$g_n:B(0,1)\to \alpha_n\cal B_\infty(\alpha_n)$$ defined by $g_n(z)=s_n z.$ 

Let $a_n$ be an arbitrary point at distance $s_n$ from $0$ and let $\tilde s_n=\Xi(a_n)$. Observe that since $\Xi$ is differentiable at zero, for $n$ big enough the image  $\Xi(B(0,s_n))$  is coarsely a ball around zero of size comparable to that of $\ov\alpha_n\cone_\infty(\ov\alpha_n)$, and in particular we can assume, since the cover $\{\cal B_\infty(\ov \g)\}$ is Lipschitz compatible (Definition \ref{assuA} (iii)), that $\Xi(B(0,s_n))$ is contained in $\ov\alpha_n\cal B_\infty(\ov\alpha_n)$. Furthermore we deduce that there exist positive constants $d,D$ such that for every $n$ 
$$d<\frac{\ov r_2e^{-\ov\slroot(\ov\alpha_n)}}{|\tilde s_n|}<D.$$
 
Here we denote by $\ov r_i, \ov \lambda, \ov C, \ov\slroot$ the  constants and function associated to the Lipschitz compatible cover $\{\cal B_\infty(\ov \g)\}$ for the action $\ov \rho$. 
We consider the scaling map 
$$\tilde g_n:B(0,D)\to B(0,|\tilde s_n|D)$$ by $z\mapsto z\tilde s_n$.

 Since $s_n\to0$ and $\Xi'(0)\notin\{0,\infty\}$ exists, the composition 
$$\tilde g_n^{-1}\Xi g_n(z)=\frac{\Xi(zs_n)}{\tilde s_n}\cdot\frac{s_nz}{s_nz}=\frac{\Xi(zs_n)}{s_nz}\cdot\frac{s_n}{\tilde s_n}\cdot z=\frac{\Xi(zs_n)}{s_nz}\cdot\frac{s_n}{\Xi(s_n)}\cdot z$$ 
converges uniformly on compact subsets to the identity map.

On the other hand, one has 
$$\tilde g_n^{-1}\Xi g_n=\tilde g_n^{-1}\overline\alpha_{n}\Xi\alpha_{n}^{-1} g_n.$$  We now study the maps $f_n:=\alpha_{n}^{-1}\circ g_n$ and $\tilde f_n:=\tilde g_n^{-1}\circ\overline\alpha_{n}$. Since the coverings $\cal B$ and $\ov {\cal B}$ are finite, we can assume, up to extracting a subsequence that there exists sets $\cal B_\infty\in \cal B$, $\ov{\cal B}_\infty\in \ov {\cal B}$ so that, for every $n$, $\cal B_\infty(\alpha_n)=\cal B_\infty$ (resp. $\cal B_\infty(\ov\alpha_n)=\ov{\cal B}_\infty$).

Observe that for every $x\in B(0,1)$ one has $$\log| f_n'(x)|=\log|(\alpha_n^{-1})'(g_nx)|+\log|s_n|= -\log|\alpha_n'(\alpha_n^{-1}g_nx)|+\log|s_n|.$$ 
Now by definition of $g_n$, we have that $g_nx \in\alpha_n\cal B_\infty(\alpha_n)$ and thus $\alpha_n^{-1}(g_nx)\in\cal B_\infty(\alpha_n)$. For $n$ large enough we can apply Lemma \ref{controlC} to $\alpha_n$ to obtain $\kappa$ so that for every pair $x,y\in B(0,1)$ it holds  $$\big|\log|f_n'(x)|-\log|f_n'(y)|\big|\leq \kappa d(x,y)^\nu.$$

\noindent
We conclude that the family of maps $(f_n)$ is uniformly bi-Lipschitz on $B(0,1)$ and thus, since $(f_n0)$ is bounded, Arzela-Ascoli's Theorem applies to give a subsequence (still denoted by $f_n$) that converges to a bi-Lipschitz map $f$ defined on $B(0,1)$.

A similar reasoning applies to the maps $\tilde f$ defined on $\ov{\cal B}_\infty$, and we obtain that, about $0$, $\Xi$ can be written as a composition of bi-Lipschitz maps and is thus bi-Lipschitz. Using the action of $\G$ we extend the Lipschitz property of $\Xi$ to the whole $\bord\G,$ concluding the proof.
\end{proof}

The following Lemma guarantees we can later apply the results of this section to the situation of our interest.

\begin{lemma}\label{usar} \item\begin{itemize}\item[-] Assume $\bord\G$ is homeomorphic to a circle and let $\rho:\G\to\SL(d,\R)$ be $(1,1,2)$-hyperconvex. Then the induced action of $\rho(\G)$ on the $\class^{1+\nu}$ circle $\xi(\bord\G)$ admits a Lipschitz compatible covering. \item[-] If $\G$ is a convex-co-compact Kleinian group then the action of $\G$ on $\bord_\infty\H^3$ admits a Lipschitz compatible cover.\end{itemize}
\end{lemma}
	
\begin{proof} Recall from Section \ref{cont} that we have fixed a word metric on $\G$ and we denote by $\cone_\infty(\g)\subset\partial\G\subset X$ the set of endpoints of geodesic rays contained in the cone type $\cone(\g)$. 

Let $\delta_\rho$ be the fundamental constant  of $\rho$ from Definition \ref{FConstant}, and let $\cal B_{\infty}(\g)=X_\infty(\gamma)$  be the $\delta_\rho/2$-neighbourhood of $\cone_\infty(\g)$ inside $\circle$. This is the thickened cone type at infinity considered in \cite[Section 5]{PSW1} (see also the proof of Proposition \ref{conetypesBalls}). It is a proper subset of $\circle$ by Corollary \ref{c.fundamentalConstant}. The cover $\cal B$ is finite since there are only finitely many cone types \cite[p. 455]{BH}.

Property (i) holds since the same property holds for $\cone_\infty(\g)$, Property (ii) is a consequence of Proposition \ref{Lipschitz-compatibleB}. Finally, Property (iii) was proven in \cite[Corollary 5.10]{PSW1} choosing $\slroot(\g):=\slroot_1(\cartan\rho(\g))$ (see also the proof of Proposition \ref{conetypesBalls}). Observe that in the real case by considering $X=\circle$ we are implicitly considering only the intersection with the limit set, while in the Kleinian group case it is not necessary to intersect with the limit set since the $\G$-action on the whole $X$ is conformal.
\end{proof}

We now establish the following corollary that will be used in the sequel.

\begin{cor}\label{diffpointreal} Assume $\bord\G$ is homeomorphic to a circle. Let $\rho:\G\to\PGL(d,\R)$ and $\ov\rho:\G\to\PGL(\ov d,\R)$ be $(1,1,2)$-hyperconvex, consider the map between $\class^{1+\nu}$ circles $$\Xi=\ov\xi\circ\xi^{-1}:\xi(\bord\G)\to\ov\xi(\bord\G).$$ If $\Xi$ has a differentiability point with finite non-vanishing derivative then $\rho$ and $\ov\rho$ are gap-isospectral.
\end{cor}

\begin{proof}By Lemma \ref{usar} we can apply Proposition \ref{c.indepPer} to obtain that $\Xi$ is bi-Lipschitz. The following standard lemma from linear algebra (see for example Benoist \cite{convexes1} and S. \cite[Lemma 3.4]{exponential}) gives the period computation completing the proof.\end{proof}

\begin{lemma}\label{lambda2} 
Let $g\in\PGL(d,\R)$ be proximal with attracting point $g_+\in\mathbb P(\R^d)$ and repelling hyperplane $g_-\in\mathbb P((\R^d)^*)$. Let $V_{\lambda_2(g)}$ be the sum of the  characteristic spaces of $g$ whose associated eigenvalue is of modulus $\exp\lambda_2(g),$  Then for every $v\notin\P(g_-),$ with non-zero component in $V_{\lambda_2(g)},$ one has 
$$\lim_{n\to\infty}\frac{\log d_\P(g^n(v),g_+)}n= -\slroot_1(\lambda(g)).$$
\end{lemma}

\subsection{Limit curves in non-maximal flags} We proceed with another intermediate step for the proof of Theorem \ref{t.Zcl} describing differentiability points of boundary maps in partial flag manifolds $\cal F_{\{\sroot,\bb\}}$ for $\{\sroot,\bb\}$-Anosov representations. This step follows from the combination of Theorem \ref{t.C1} and Corollary \ref{diffpointreal}.

Let $\sf G$ be real-algebraic and semi-simple. Let $\{\sroot,\bb\}\subset\simple$ be two distinct simple roots. The partial flag space $\cal F_{\{\sroot,\bb\}}$ carries two transverse foliations that are the level sets of the natural projections $\cal F_{\{\sroot,\bb\}}\to\cal F_{\{\sroot\}}$ and $\cal F_{\{\sroot,\bb\}}\to\cal F_{\{\bb\}}.$ We will refer to these as the \emph{canonical foliations} of $\cal F_{\{\sroot,\bb\}}.$

\begin{cor}\label{nonC1} Let $\sf G$ be real-algebraic and semi-simple and let $\{\sroot,\bb\}\subset\simple$ distinct. Let $\rho:\piS\to\sf G$ be Zariski-dense and $\{\sroot,\bb\}$-Anosov. If both curves $\xi^\sroot(\bord\piS)$ and $\xi^{\bb}(\bord\piS)$ are $\class^1$ then every differentiability point of $\xi^{\{\sroot,\bb\}}(\bord\piS)$ is tangent to one of the canonical foliations of $\cal F_{\{\sroot,\bb\}}.$
\end{cor}

\begin{proof} By Benoist's Theorem \ref{densidad} the limit cone of $\rho$ has non-empty interior, in particular there exists $\g\in\piS$ such that \begin{equation}\label{isospe}\sroot(\lambda(\g))\neq\bb(\lambda(\g)).\end{equation}

Consider the Tits representations $\Fund_\sroot$ and $\Fund_\bb$ associated to $\sroot$ and $\bb.$ Since $\rho(\piS)$ is Zariski-dense, both representation $\Fund_\sroot\rho$ and $\Fund_\bb\rho$ are irreducible and since $\rho$ is $\{\sroot, \bb\}$-Anosov  both representation $\Fund_\sroot\rho$ and $\Fund_\bb\rho$ are projective Anosov. Recall that by definition of $\Fund_\sroot,$ for every $g\in\sf G$ one has $$\slroot_1\big(\lambda\big(\Fund_\sroot(g)\big)\big)=\sroot\big(\lambda(g)\big),$$ so by Equation \eqref{isospe} the representations $\Fund_\sroot\rho$ and $\Fund_\bb\rho$ are not gap-isospectral.

Since the maps $\zeta_\sroot$ and $\zeta_\bb$ are analytic, both projective curves $\zeta_\sroot\xi^{\sroot}(\bord\piS)$ and $\zeta_\bb\xi^{\bb}(\bord\piS)$ are $\class^1$ and thus by Zhang-Zimmer's Theorem \ref{t.C1} the representations $\Fund_\sroot\rho$ and $\Fund_\bb\rho$ are $(1,1,2)$-hyperconvex.

The natural embedding $\cal F_{\{\sroot,\bb\}}\to\P(V_\sroot)\times\P(V_\bb)$ sends $\xi^{\{\sroot,\bb\}}$ to the graph of the map $\Xi$ from Corollary \ref{diffpointreal} and thus the corollary implies the result.\end{proof}

\subsection{Proof of Theorem \ref{t.Zcl}}\label{s.pB}\label{s.Zd}
The goal of the section is to prove Theorem \ref{t.Zcl}, stating that the Zariski closure $\sf G$ of the image of an irreducible $(1,1,2)$-hyperconvex representation $\rho:\piS\to\PGL(d,\R)$ is simple and the highest weight of the induced representation $\Fund:\sf G\to\PGL(d,\R)$ is a multiple of a fundamental weight associated to a root whose root-space is one-dimensional.

It is known   that an irreducible subgroup $\sf G<\PGL(d,\R)$ containing a proximal element is semi-simple without compact factors (see S. \cite[Lemma 8.6]{entropia} for an explicit argument following a suggestion by Quint).

We consider the induced representation $\rho_0:\G\to\sf G$ and denote by $\Fund:\sf G\to\PGL(d,\R)$ the linear representation so that $\rho=\Fund\rho_0.$ Let $\chi=\chi_\Fund\in\frak a^*$ be the  highest weight of $\Fund$. As in Definition \ref{trep} we consider
$$\t=\t_\Fund=\{\sroot\in\Delta: \chi-\sroot \text{ is a weight of }\Fund \}=\{\sroot\in\simple:\<\chi,\sroot\>\neq0\}.$$ It is enough to show that $\t$ is reduced to a single root $\{\sroot_0\};$ indeed, if this is the case, upon writing $\chi$ in the basis of fundamental weights $\{\peso_\sroot:\sroot\in\simple\}$ (recall their defining Equation \eqref{pesoFund}) one has $$\chi=\sum_{\sroot\in\simple}\<\chi,\sroot\>\peso_\sroot=\<\chi,\sroot_0\>\peso_{\sroot_0},$$ Moreover this gives: 
\begin{itemize}
	\item[-]  $\sf G$ is simple by Lemma \ref{simple};
	\item[-] the weights on the first level consist solely on $\chi-\sroot$ and its associated weight space is $\phi(\ge_{-\sroot})V_{\chi_\Fund}$. Since $\rho(\G)$ is $\{\slroot_2\}$-Anosov one has that $\phi(\ge_{-\sroot})V_{\chi_\Fund}$ is one-dimensional, but by Lemma \ref{not} no element of $\ge_{-\sroot}$ acts trivially on $V_{\chi_\Fund}$ so $\ge_{-\sroot}$ is $1$-dimensional, as desired.
\end{itemize}

We proceed now to show that in the present situation $\t$ consists of only one element. By definition of $\t$ one has, for every $g\in\sf G$, that $$\slroot_1\big(\lambda\big(\Fund(g)\big)=\min_{\sroot\in\theta}\big\{\sroot(\lambda_\sf G(g))\big\}.$$ Consequently, the limit cone $\calL_{\rho_0}\subset\a^+_\sf G$ does not intersect the walls of elements in $\t$ and, since $\rho_0:\G\to\sf G$ is a quasi-isometry, Remark \ref{wallsAnosov} implies that the representation $\rho_0$ is $\t$-Anosov.

Recall from Equation \eqref{maps} that we have a $\Fund$-equivariant analytic embedding $\zeta_\t:\sf G/\sf P_{\theta}\to \P(\R^{d}).$ One has moreover that $\xi^1_\rho=\zeta_\t\circ\xi^\t_{\rho_0}.$ In particular the boundary map $\xi^\theta$ has $\class^1$-image. Composing with the projections $\cal F_\t\to\cal F_{\t'}$ one sees that, for any $\t'\subset\t$  the curve $\xi_{\rho_0}^{\t'}(\bord\G)$ is a $\class^1$ circle.

Assume now there exists two distinct roots $\sroot,\bb$ in $\t.$ By the previous paragraph the curve $\xi^{\{\sroot,\bb\}}(\bord\G)$ is $\class^1.$ Corollary \ref{nonC1} gives then that $\xi^{\{\sroot,\bb\}}(\bord\G)$ is necessarily contained in one of the leaves of the canonical foliations of $\cal F_{\{\sroot,\bb\}},$ thus giving that one of the maps $\xi^{\sroot}$ or $\xi^\bb$ is constant, achieving a contradiction. \qed

\section{Non-differentiability and $1$-conicality: The proof of Theorem \ref{tutti}}\label{s.Hffproof}

\subsection{Non-differentiability and $1$-conicality}

By means of \S\,\ref{s.6.1} we can improve Lemma \ref{generalcase} when we deal with a pair of real hyperconvex representations of surface groups, this is the missing ingredient for Theorem \ref{tutti}:

\begin{cor}\label{nondiffbconicalR} Assume $\bord\G$ is homeomorphic to a circle. Let $\rho,\ov\rho$ two $(1,1,2)$-hyperconvex representations over $\R$ of $\G$ that are not gap-isospectral. Then, the set of non-differentiability points of $\Xi$ coincides with the set of $1$-conical points.
\end{cor}

\begin{proof} We choose a $\class^1$ identification of the $\class^1$ torus $\xi(\bord\G)\times\ovxi(\bord\G)\subset\P(\R^d)\times \P(\R^{\ov d})$ with the quotient of the square $[-1,1]\times[-1,1]$ preserving the product structure, and such that the point $(x,\Xi(x))$ corresponds to $(0,0)$. In these coordinates the graph of $\Xi$ is a monotone curve $[-1,1]\to[-1,1]$ passing through the origin. Since the chosen identification is $\class^1$, it is in particular $K$-bi-Lipschitz for some $K$, so we can write (coarsely in a small neighbourhood of $x$) $d(\xi(y),\xi(x))=|y|$ and $d(\ov\xi(y),\ov\xi(x))=|\Xi(y)|$.

From Lemma \ref{generalcase} we know that $x$ is $1$-conical if and only if either $\lim_{y\to x}\frac{|\Xi(y)|}{|y|}$ exists and is far from $0$ and $\infty$, either it does not exist. The proposition is settled if we show that the first situation cannot happen, so let's assume it does. However, since $\Xi$ is monotone we can remove the $|\,|$ and we get that $x$ is a differentiability point of $\Xi$ with oblique derivative. Corollary \ref{diffpointreal} implies then that for all $\g\in\G$ one has $\slroot_1(\lambda(\g))=\ov\slroot_1(\lambda(\ov\g))$, contradicting our assumption.\end{proof}

\subsection{Proof of Theorem \ref{tutti} and an analogous for Kleinian groups}
We begin with the proof of Theorem \ref{tutti} by recalling the following result from Beyrer-P. \cite{BeyP2}

\begin{cor}[Beyrer-P. \cite{BeyP2}]\label{l.weaklyirred} Assume $\bord\G$ is homeomorphic to a circle and let $\rho:\G\to \PGL(d,\R)$ be $(1,1,2)$-hyperconvex. Then there exists an irreducible $(1,1,2)$-hyperconvex representation $\rho_0:\G\to \PGL(m,\R)$ such that, for every $\g\in\G$ one has
$$\slroot_1\big(\lambda(\g)\big)=\slroot_1\big(\lambda(\rho_0\g)\big).$$	
\end{cor}

We now prove Theorem \ref{tutti}. Since there exists $\g\in\G$ with $\slroot_1\big(\lambda(\g)\big)\neq\ov{\slroot}_1\big(\lambda(\ov\g)\big),$ Corollary \ref{l.weaklyirred} allows us to apply Proposition \ref{nonA} to obtain the density assumption in Theorem \ref{Hffconical}, so one has $$\Hff\Xi\big(\{1\textrm{-conical points}\}\big)=\hC{\max\{\slroot,\ov{\slroot}\}}.$$ Corollary \ref{nondiffbconicalR} states that the set of $1$-conical points coincides with the set of non-differentiability points of $\Xi.$ The inequality $\hC\infty <1$ follows from the strict convexity of the critical hypersurface $\cal Q_\cc$, where $\cc$ is the cocycle studied in Section \ref{s.lower}. This completes the proof of Theorem \ref{tutti}.

\subsection{Proof of Corollary \ref{cor.B}}
We conclude the paper proving Corollary \ref{cor.B}. Recall from Section \ref{representaciones} that for every simple root $\sroot$ of $\sf G$ we chose a Tits representation $\Phi_\sroot:\sf G\to\PSL(V_\sroot)$. 

\begin{cor}\label{c.Hk} Assume $\bord\G$ is homeomorphic to a circle and let $\sf G$ be a simple Lie group. Let $\rho:\G\to\sf G$ have Zariski-dense image. If for $\sroot,\bb\in\simple$ the representations $\Phi_\sroot\circ\rho$ and $\Phi_\bb\circ \rho$ are $(1,1,2)$-hyperconvex,
then 
\begin{enumerate}
		\item the image of the limit curve $\xi^{\{\sroot,\bb\}}:\bord\G\to\cal F_{\{\sroot,\bb\}}$ is Lipschitz   and the Hausdorff dimension of the points where it is non-differentiable is $\hC{\max\{\sroot,\bb\}}.$ 
		\item If the opposition involution $\ii$ on $\ge$ is non-trivial and $\bb=\ii\sroot$ then $$\hC{\max\{\sroot,\bb\}}=\hC{(\sroot+\bb)/2}.$$
		\end{enumerate}
\end{cor}
\begin{proof}\ 

(i) Since the map $\Phi_\sroot:\calF_{\sroot}\to\P(V_{\sroot})$ is analytic, and $\Phi_\sroot\circ\xi^\sroot(\bord\G)$ is a $\class^1$-submanifold (Theorem \ref{t.C1}), $\xi^\sroot(\bord\G)$ is a $\class^1$ submanifold as well. The curve $\xi^{\{\sroot,\bb\}}:=\cal F_{\{\sroot,\bb\}}\cap (\xi^{\sroot}(\bord\G)\times\xi^{\bb}(\bord\G))$ is the graph of the homeomorphism $\Xi$ and is thus a Lipschitz curve. The second claim is then a direct consequence of Theorem \ref{tutti}.

(ii) Assume the opposition involution $\ii$ of $\ge$ is non-trivial and that $\bb=\ii\sroot$. Using notation from Section \ref{s.lower} with $\sroot=\slroot$ and $\bb=\ii\sroot=\ov{\slroot}$ we let $V^*=\spa\{\sroot,\bb\}$, $V=\a_\t/\Ann(V^*)$, $\Pi:\a_\t\to V$ the quotient projection, $\|\,\|_\infty=\max\{|\sroot|,|\bb|\},$ $\|\,\|^1$ its dual norm on $V^*$ and $\vi\in\cal Q_\cc$ the only form minimizing $\|\,\|^1$.

Since $\ii\sroot=\bb$, the space $V^*$ is preserved by $\ii$ and the fact that $\lambda(g^{-1})=\ii\lambda(g)$ (for all $g\in \sf G$) implies that $\cal Q_\cc$ is $\ii$-invariant. Moreover, the norm $\|\,\|^1$ is $\ii$-invariant and by definition of $\vi$ one has $\ii\vi=\vi$. However, $(\sroot+\bb)/2$ is also $\ii$-invariant and $\hC{(\sroot+\bb)/2}(\sroot+\bb)/2\in\cal Q_\cc$ whence $$\vi=\hC{(\sroot+\bb)/2}(\sroot+\bb)/2.$$

In order to prove the result it is thus enough to show that \begin{equation}\label{e.Quint}\hC{\max\{\sroot,\bb\}}=\|\vi\|^1.\end{equation} We conclude the proof deducing this equality from Quint's \cite[Proposition 3.3.3]{Quint-Div}.

We consider  the counting measure $$\nu=\sum_{\g\in\G}\delta_{\Pi\cartan_\t(\g)}$$ on the vector space $\mathcal E=V$ and the norm $N=\|\,\|_\infty$. We then have, in the notation of \cite[\S 3]{Quint-Div}, that $\tau^N_\nu=\hC{\max\{\sroot,\bb\}}$ and, by Remark \ref{r.entropy=crit},  $\sigma^N_\nu=\inf_{\varphi\in\cal Q_\cc}\|\varphi\|^1$. Thus, in order to deduce Equation \eqref{e.Quint} from \cite[Proposition 3.3.3]{Quint-Div} it is enough to verify that the counting measure $\nu$ is of concave growth as in \cite[\S 3.2]{Quint-Div}. In turn this is a consequence of Lemma \ref{l.Quint} below, an adaptation of \cite[Proposition 2.3.1]{Quint-Div} (see also Kim-Oh-Wang \cite[Lemma 3.8]{KimOhWang} where similar arguments are explained for the $\a_\t$ counting measure). 
\end{proof}	
\begin{lemma}\label{l.Quint}
Let $\|\,\|$ be a norm on $V$. Let $\grupo<\sf G$ be Zariski-dense and $\{\sroot,\bb\}$-Anosov. Then there exists a product map $m:\grupo\times\grupo\to\grupo$ with the following properties:
\begin{enumerate}
	\item there exists a real number $\kappa\geq 0$ such that, for all $\g_1,\g_2\in\grupo$,
$$\|\Pi\cartan_\t(m(\g_1,\g_2))-\Pi\cartan_\t(\g_1)-\Pi\cartan_\t(\g_2)\|\leq\kappa;$$	 
	\item for every real $R\geq 0$ there exists a finite subset $H$ of $\grupo$ such that, for $\g_1,\g_2\g_1',\g_2'$ in $\grupo$ with $\|\Pi\cartan_\t(\g_i)-\Pi\cartan_\t(\g_i')\|<R$ for $i=1,2$, then
	$$m(\g_1,\g_2)=m(\g_1',\g_2') \Rightarrow \g_i'\in\g_i H, \text{ for $i=1,2$}.$$
\end{enumerate}		
\end{lemma}	
\begin{proof}
It is enough to consider the \emph{generic product} map $\pi:\grupo\times\grupo\to\grupo$ constructed in \cite[Proposition 2.3.1]{Quint-Div}, which satisfies the analogous properties with respect to the Cartan projection $\cartan:\sf G\to\frak a$ and a norm $\|\,\|$ on $\frak a$. The first property is satisfied since we can assume that the projection $\Pi\circ\pi_\theta:\frak a\to V$ is norm non-increasing. The second follows from the Anosov property: by the construction in \cite[Proposition 2.3.1]{Quint-Div}  one can choose $H$ to be the set of elements $\g$ such that $\|\Pi\cartan_\t(\g)\|<R'$ for some $R'$ depending on $R$. Such set is finite because, by definition of $\Pi$, there exists $R''$ depending on $R'$ and the norm $\|\,\|$ such that if   $\|\Pi\cartan_\t(\g)\|<R'$ then $\sroot(\cartan(\gamma))<R''$, which in turn implies by Definition \ref{AnosovDefi} that $|\gamma|<R''/\mu+C$, and thus $\g$ belongs to a finite subset. 
\end{proof}	

\subsection{The $\PSL(2,\C)$-case}\label{klein}

If $\rho,\ov\rho:\G\to\PSL(2,\C)$ are convex co-compact representations that are connected by convex-co-compact representations, it was proven by Marden \cite{Marden} that the natural map $\Xi:\Lambda_\rho\to\Lambda_{\ov\rho}$  conjugating the respective actions extends to a Hölder homeomorphism $\Xi:\C\P^1\to\C\P^1$ that is $(\rho,\ov\rho)$-equivariant. We consider in this case the complex derivative of such an extension $\Xi$ and say that $\Xi$ is $\C$-\emph{differentiable} at a given $x\in\Lambda_\rho$ if, conformally  identifying $\bord\H^3-\{\textrm{point}\}$ to $\C$, the limit $$\Xi'(x):=\lim_{y\to x} \frac{\Xi(x)-\Xi(y)}{x-y}$$ exists or is infinite. We let now $\Ndiff_{\rho,\ov\rho}$ be the set of points $x\in\Lambda_\rho$ where the extended conjugating map $\Xi$ is not $\C$-differentiable and let

The proof of the following works verbatim as in Corollary \ref{nondiffbconicalR}.
\begin{prop}\label{nondiffbconicalC}
	Let $\rho,\ov\rho:\G\to\PSL(2,\C)$ be non-gap-isospectral and in the same connected component of $$\big\{\varrho:\G\to\PSL(2,\C):\,\varrho\mathrm{\ is\ convex\ co-compact}\big\}.$$ Then, the set of non-$\C$-differentiability points of $\Xi$ coincides with the set of $1$-conical points.
\end{prop}

Density of the group generated by the pairs $\{(\lambda(\g),\lambda(\bar\g)):\g\in\G\}$ follows readily from Benoist \cite{benoist2} (see Theorem \ref{densidad}), from this point on the exact same proof of Theorem \ref{tutti} gives the following.

\begin{thm}\label{thm.tuttiC}
	Let $\rho,\ov\rho:\G\to\PSL(2,\C)$ be non-gap-isospectral convex co-compact representations that are connected by convex co-compact representations. Assume without loss of generality that $\hJ{\ov\slroot}\geq \hJ{\slroot}$. If $\II_{\slroot}(\ov\slroot)>1$, then $\Hff(\Ndiff_{\rho,\ov\rho})=\hC\infty.$
\end{thm}

\bibliographystyle{plain}

\smallskip
\author{\vbox{\footnotesize\noindent 
	Beatrice Pozzetti\\
Ruprecht-Karls Universit\"at Heidelberg\\ Mathematisches Institut, Im
Neuenheimer Feld 205, 69120 Heidelberg, Germany\\
	\texttt{pozzetti@uni-heidelberg.de}
\bigskip}}

\author{\vbox{\footnotesize\noindent 
	Andr\'es Sambarino\\
	Sorbonne Universit\'e \\ IMJ-PRG (CNRS UMR 7586)\\ 
	4 place Jussieu 75005 Paris France\\
	\texttt{andres.sambarino@imj-prg.fr}
\bigskip}}

\end{document}